\documentclass[11pt]{amsart}
\usepackage{hyperref}
\usepackage[psamsfonts]{amssymb}
\usepackage{amsmath,amsfonts,amssymb,amscd}
\usepackage{latexsym,epsfig}
\usepackage[latin1]{inputenc}
\usepackage[T1]{fontenc}
\usepackage{enumerate}
\usepackage{amsbsy}
\usepackage{graphics}
\usepackage{graphicx}
\usepackage{grffile}
\usepackage{graphicx,subfigure,epsfig}
\usepackage{mathtools}
\usepackage[usenames,dvipsnames]{pstricks} 
\usepackage{graphicx,subfigure,epsfig}
\usepackage{subfigure}
\usepackage{epstopdf}
\usepackage{mathtools}

\newtheorem{theorem}{Theorem}[section]
\newtheorem{corollary}[theorem]{Corollary}
\newtheorem{lemma}[theorem]{Lemma}
\newtheorem{proposition}[theorem]{Proposition}
\newtheorem{problem}[theorem]{Problem}
\newtheorem{example}[theorem]{Example}
\newtheorem{remark}[theorem]{Remark}
\newtheorem{defi}[theorem]{Definition}

\begin{document}

\title[Gordian complexes of knots and virtual knots]{Gordian complexes of knots and virtual knots given by region crossing changes and arc shift moves}

\author{AMRENDRA~GILL}
\address{Department of Mathematics, Indian Institute of Technology Ropar, India} 
\email{amrendra.gill@iitrpr.ac.in}

\author{MADETI~PRABHAKAR}
\address{Department of Mathematics, Indian Institute of Technology Ropar, India} 
\email{prabhakar@iitrpr.ac.in}

\author{ANDREI~VESNIN}
\address{Tomsk State University, Tomsk, 634050, Russia \\ and Sobolev Institute of Mathematics, Novosibirsk, 630090, Russia} 
\email{vesnin@math.nsc.ru} 

\keywords{Gordian complex, virtual knot, local moves }

\subjclass[2010]{57M25, 57M27}

\thanks{The first and second authors were supported by DST-RSF Pro\-ject DST/INT/RUS/RSF/P-02.
The third author was supported in part by the Russian Scientific Foundation (grant number RSF-19-41-02005). }

\begin{abstract}
Gordian complex of knots was defined by Hirasawa and Uchida as the simplicial complex whose vertices are knot isotopy classes in $\mathbb{S}^3$. Later Horiuchi and Ohyama defined Gordian complex of virtual knots using $v$-move and forbidden moves. In this paper we discuss Gordian complex of knots by region crossing change and Gordian complex of virtual knots by arc shift move. Arc shift move is a local move in the virtual knot diagram which results in reversing orientation locally between two consecutive crossings. We show the existence of an arbitrarily high dimensional simplex in both the Gordian complexes, i.e., by region crossing change and by the arc shift move. For any given knot (respectively, virtual knot) diagram we construct an infinite family of knots (respectively, virtual knots) such that any two distinct members of the family have distance one by region crossing change (respectively, arc shift move). We show that that the constructed virtual knots have the same affine index polynomial. 
\end{abstract}

\maketitle

\section{Introduction}
Gordian complex of knots was defined by M.~Hirasawa and Y.~Uchida~\cite{a} as the simplicial complex whose vertices are knot isotopy classes in $\mathbb{S}^3$. 
A natural question arises that, can in some meaningful way notion of an $n$-simplex spanned by $n+1$ knots $K_0, K_1,\ldots, K_n$ be defined? It is here that the idea of unknotting operations in knots is well utilized to define notion of $n$-simplex spanned by knots $K_0, K_1,\ldots, K_n$. 
An unknotting operation is a local change in knot diagram which can be repeatedly used to convert any knot diagram into a trivial knot diagram. For any two knots $K$ and $K'$, it is always possible to find diagrams $D, D'$ of $K, K'$ respectively (we write $D \sim K$ and $D' \sim K'$) such that $D$ can be converted into $D'$ using finite number of unknotting operations or their inverses. 

For an unknotting operation, say $\gamma$, and any two knots $K, K'$ consider the collection of pairs of knots $\Lambda= \{(D_\alpha, D_\beta)\vert D_\alpha\sim K, D_\beta \sim K' \text{ and } D_\alpha$  can be converted into  $D_\beta$  using finitely many  operations $\gamma$ or its inverse$\}$. For any pair of knots $(D_\alpha, D_\beta)$ in $\Lambda$, let $d_{\gamma}(D_\alpha, D_\beta)$ denotes the minimal number of $\gamma$ operations required to convert $D_\alpha$ into $D_\beta$.  The \emph{$\gamma$-distance} between two knots $K $ and $K'$, denoted by $d_{\gamma}(K,K')$, is defined as minimum of all $d_{\gamma}(D_\alpha, D_\beta)$ such that $(D_\alpha, D_\beta)\in \Lambda$. Crossing change, which involves changing the crossing information at a crossing from over/under to under/over is well known to be an unknotting operation for classical knots. H. Murakami \cite{k} defined Gordian distance $d_{G}(K,K')$ between a pair of knots $K$ and $K'$ as the minimal number of crossing changes required to transform $K$ into $K'$. Further, using specific pairs of knots having Gordian distance one, notion of Gordian complex for knots was defined by M.~Hirasawa and Y.~Uchida \cite{a} as follows:

\begin{defi} \cite{a}
The \emph{Gordian complex} $\mathcal{G}$ of knots is a simplicial complex defined by the following:
\begin{enumerate}
\item The vertex set of $\mathcal{G}$ consists of all the isotopy classes of oriented knots in $\mathbb{S}^3$, and
\item A family of n+1 vertices $\{K_0 ,K_1, \ldots, K_n\}$ spans an n-simplex if and only if the Gordian distance $d_{G}(K_i,K_j)=1$ for $0\leq i, j \leq n$, $i \neq j$.
\end{enumerate}
\end{defi}

Moreover, they studied the structure of simplexes in $\mathcal{G}$ and proved following results.

\begin{theorem} \cite{a} 
For any $1$-simplex $e$ of the Gordian complex, there exists an infinitely high dimensional simplex $\sigma$ such that $e$ is contained in $\sigma$.
\end{theorem}

\begin{corollary} \cite{a} 
For any knot $K_0$, there exists an infinite family of knots $K_0, K_1, K_2, \ldots$ such that the Gordian distance $d_{G}(K_i,K_j)=1$, for any $i\neq j$.
\end{corollary}

J.~Hoste, Y.~Nakanishi and K.~Taniyama~\cite{l} defined $H(n)$-\emph{move} presented in Fig.~\ref{fig1} and proved that for any integer $n \geq 2$  every knot diagram can be converted into trivial knot diagram using finite number of $H(n)$-moves. It was also proved in \cite{l} that each $H(n)$-move can be realized using one $H(n+1)$-move and few Reidemeister moves.
\begin{figure}[h]
\centerline{\includegraphics[width=0.8\linewidth]{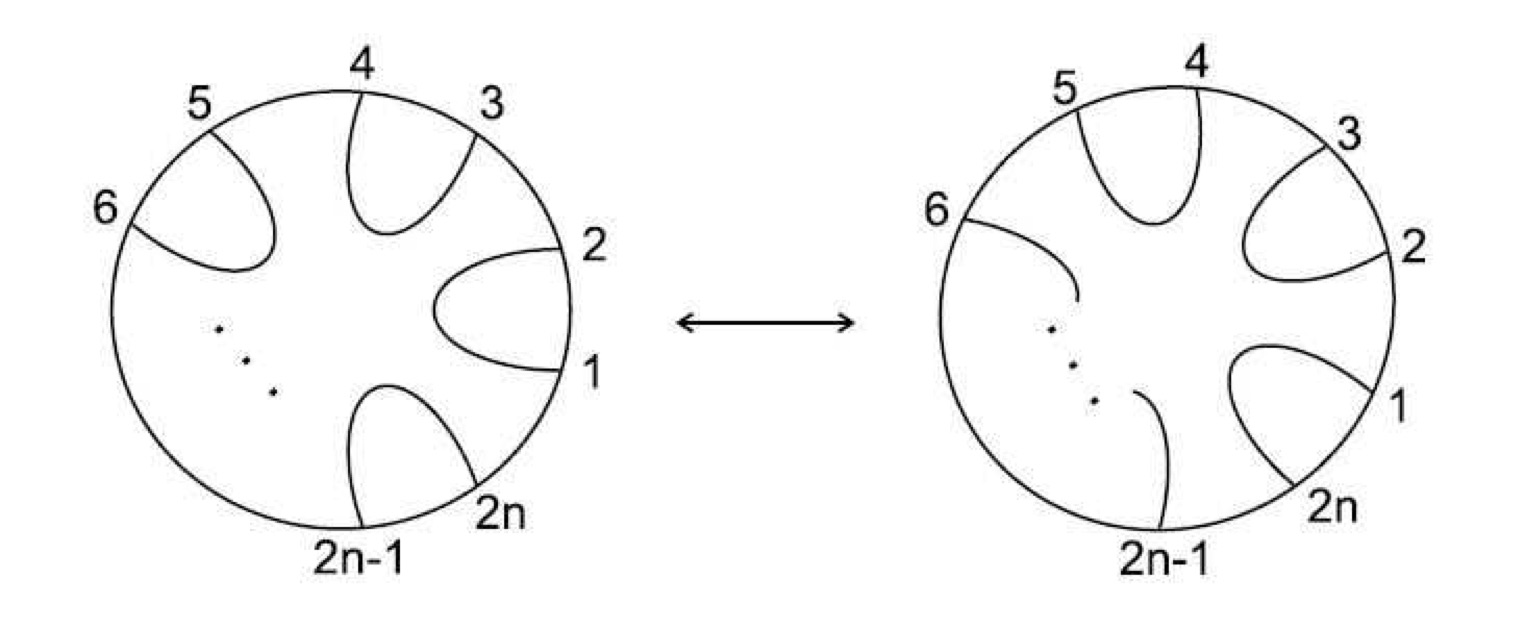}} 
\caption{$H(n)$-move.} \label{fig1}
\end{figure}

K.~Zhang, Z.~Yang and F.~Lei~\cite{f} defined Gordian complex $\mathcal{G}_{H(n)}$ of knots by $H(n)$-move in a similar way by taking all knot isotopy classes as vertex set and defining a set of $n+1$ knots $\{K_0, K_1, \ldots, K_n\}$ as an $n$-simplex if $d_{H(n)}(K_i,K_j)=1$ for distinct $i$ and $j$. It follows that any simplex in $\mathcal{G}_{H(n)}$ also forms a simplex in $\mathcal{G}_{H(n+1)}$ as a $H(n)$-move can be realized using one $H(n+1)$-move. Further, following results were proved for $\mathcal{G}_{H(n)}$.

\begin{theorem} \cite{f}
For any $0$-simplex p of the $H(n)$-Gordian complex, there exists an arbitrarily high dimensional simplex $\sigma$ of $\mathcal{G}_{H(n)}$ such that p is a subcomplex of $\sigma$ for all $n \geq 2$.
\end{theorem}

\begin{corollary} \cite{f} 
For any knot $K_0$, there exists an infinite family of knots $K_0, K_1, K_2, \ldots$ such that the $H(n)$-Gordian distance $d_{H(n)}(Ki,Kj) = 1$ for any $i \neq j $ and all $n \geq 2$.
\end{corollary}

Pass-moves and $\overline{\sharp}$-move are another set of local moves are defined in an oriented  knot diagram. L.~Kauffman \cite{n} studied pass-moves presented in Fig.~\ref{fig2} and proved that every knot can be transformed into either trefoil knot or trivial knot diagram using pass-moves. As a result pass-move fails to become an unknotting operation for knots.
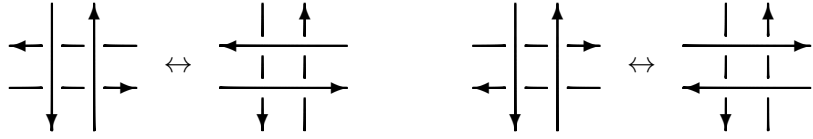
\begin{figure}
\unitlength=.28mm
\begin{center} 
\begin{picture}(0,70)(0,0)
\put(-110,0){\begin{picture}(0,60)
\thicklines
\qbezier(-60,0)(-60,0)(-60,60)
\qbezier(-40,0)(-40,0)(-40,60)
\qbezier(-80,40)(-65,40)(-65,40)
\qbezier(-55,40)(-55,40)(-45,40)
\qbezier(-35,40)(-35,40)(-20,40)
\qbezier(-80,20)(-65,20)(-65,20)
\qbezier(-55,20)(-55,20)(-45,20)
\qbezier(-35,20)(-35,20)(-20,20)
\put(-60,5){\vector(0,-1){5}}
\put(-40,55){\vector(0,1){5}}
\put(-75,40){\vector(-1,0){5}}
\put(-25,20){\vector(1,0){5}}
\put(0,30){\makebox(0,0)[c,c]{$\leftrightarrow$}}
\qbezier(20,40)(20,40)(80,40)
\qbezier(20,20)(20,20)(80,20)
\qbezier(40,0)(40,0)(40,15)
\qbezier(40,25)(40,25)(40,35)
\qbezier(40,45)(40,45)(40,60)
\qbezier(60,0)(60,0)(60,15)
\qbezier(60,25)(60,25)(60,35)
\qbezier(60,45)(60,45)(60,60)
\put(40,5){\vector(0,-1){5}}
\put(60,55){\vector(0,1){5}}
\put(25,40){\vector(-1,0){5}}
\put(75,20){\vector(1,0){5}}
\end{picture}}
\put(110,0){\begin{picture}(0,60)
\thicklines
\qbezier(-60,0)(-60,0)(-60,60)
\qbezier(-40,0)(-40,0)(-40,60)
\qbezier(-80,40)(-65,40)(-65,40)
\qbezier(-55,40)(-55,40)(-45,40)
\qbezier(-35,40)(-35,40)(-20,40)
\qbezier(-80,20)(-65,20)(-65,20)
\qbezier(-55,20)(-55,20)(-45,20)
\qbezier(-35,20)(-35,20)(-20,20)
\put(-60,5){\vector(0,-1){5}}
\put(-40,55){\vector(0,1){5}}
\put(-25,40){\vector(1,0){5}}
\put(-75,20){\vector(-1,0){5}}
\put(0,30){\makebox(0,0)[c,c]{$\leftrightarrow$}}
\qbezier(20,40)(20,40)(80,40)
\qbezier(20,20)(20,20)(80,20)
\qbezier(40,0)(40,0)(40,15)
\qbezier(40,25)(40,25)(40,35)
\qbezier(40,45)(40,45)(40,60)
\qbezier(60,0)(60,0)(60,15)
\qbezier(60,25)(60,25)(60,35)
\qbezier(60,45)(60,45)(60,60)
\put(40,5){\vector(0,-1){5}}
\put(60,55){\vector(0,1){5}}
\put(75,40){\vector(1,0){5}}
\put(25,20){\vector(-1,0){5}}
\end{picture}}
\end{picture}
\end{center}
\caption{Pass-moves} \label{fig2} 
\end{figure}

K.~Zhang and Z.~Yang \cite{m} defined the Gordian complexes $\mathcal{G}_{P}$ and $\mathcal{G}_{\overline{\sharp}}$ of knots by pass-moves and $\overline{\sharp}$-moves presented in Fig.~\ref{fig3}. Further, it was proved that one $\overline{\sharp}$-move can be realized using two of the pass-moves along with some Reidemeister moves and vice versa. Both the Gordian complexes $\mathcal{G}_{P}$ and $\mathcal{G}_{\overline{\sharp}}$ are not connected complexes and have exactly two components corresponding to trefoil knot and unknot. Following results were proved for $\mathcal{G}_{P}$ and $\mathcal{G}_{\overline{\sharp}}$.
\begin{figure}
\unitlength=.28mm
\begin{center} 
\begin{picture}(0,70)(0,0)
\put(-110,0){\begin{picture}(0,60)
\thicklines
\qbezier(-60,0)(-60,0)(-60,60)
\qbezier(-40,0)(-40,0)(-40,60)
\qbezier(-80,40)(-65,40)(-65,40)
\qbezier(-55,40)(-55,40)(-45,40)
\qbezier(-35,40)(-35,40)(-20,40)
\qbezier(-80,20)(-65,20)(-65,20)
\qbezier(-55,20)(-55,20)(-45,20)
\qbezier(-35,20)(-35,20)(-20,20)
\put(-60,5){\vector(0,-1){5}}
\put(-40,55){\vector(0,1){5}}
\put(-25,40){\vector(1,0){5}}
\put(-25,20){\vector(1,0){5}}
\put(0,30){\makebox(0,0)[c,c]{$\leftrightarrow$}}
\qbezier(20,40)(20,40)(80,40)
\qbezier(20,20)(20,20)(80,20)
\qbezier(40,0)(40,0)(40,15)
\qbezier(40,25)(40,25)(40,35)
\qbezier(40,45)(40,45)(40,60)
\qbezier(60,0)(60,0)(60,15)
\qbezier(60,25)(60,25)(60,35)
\qbezier(60,45)(60,45)(60,60)
\put(40,5){\vector(0,-1){5}}
\put(60,55){\vector(0,1){5}}
\put(75,40){\vector(1,0){5}}
\put(75,20){\vector(1,0){5}}
\end{picture}}
\put(110,0){\begin{picture}(0,60)
\thicklines
\qbezier(-60,0)(-60,0)(-60,60)
\qbezier(-40,0)(-40,0)(-40,60)
\qbezier(-80,40)(-65,40)(-65,40)
\qbezier(-55,40)(-55,40)(-45,40)
\qbezier(-35,40)(-35,40)(-20,40)
\qbezier(-80,20)(-65,20)(-65,20)
\qbezier(-55,20)(-55,20)(-45,20)
\qbezier(-35,20)(-35,20)(-20,20)
\put(-60,55){\vector(0,1){5}}
\put(-40,5){\vector(0,-1){5}}
\put(-25,40){\vector(1,0){5}}
\put(-25,20){\vector(1,0){5}}
\put(0,30){\makebox(0,0)[c,c]{$\leftrightarrow$}}
\qbezier(20,40)(20,40)(80,40)
\qbezier(20,20)(20,20)(80,20)
\qbezier(40,0)(40,0)(40,15)
\qbezier(40,25)(40,25)(40,35)
\qbezier(40,45)(40,45)(40,60)
\qbezier(60,0)(60,0)(60,15)
\qbezier(60,25)(60,25)(60,35)
\qbezier(60,45)(60,45)(60,60)
\put(40,55){\vector(0,1){5}}
\put(60,5){\vector(0,-1){5}}
\put(75,40){\vector(1,0){5}}
\put(75,20){\vector(1,0){5}}
\end{picture}}
\end{picture}
\end{center}
\caption{$\overline{\#}$-Moves} \label{fig3} 
\end{figure}
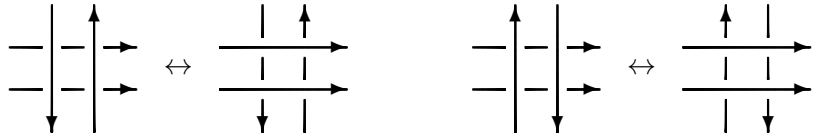

\begin{theorem} \cite{m} 
For any knot $K_0$, there exists an infinite family of knots $\{K_0, K_1, \ldots \}$ such that the Gordian distance $d_{G}(K_i,K_j) = 1$ and pass-move Gordian distance $d_P(K_i,K_j) = 1$ for any $i \neq j$.
\end{theorem}

\begin{theorem} \cite{m} 
For any knot $K_0$, there exists an infinite family of knots $\{K_0, K_1, \ldots \}$ such that $d_{\overline{\sharp}}(K_i,K_j) = 1$ and $d_{H(n)}(K_i,K_j) = 1$ for any $i \neq j$ and $n \geq 2$.
\end{theorem}

It can be observed from the case of Gordian complexes $\mathcal{G}_{P}$ and $\mathcal{G}_{\overline{\sharp}}$, it is the local move which decides the properties like connectedness of a Gordian complex. Also, whether given $n+1$ knots $K_0, K_1,...,K_n$ will form an $n$-simplex or not entirely depends on the local move used to define the Gordian complex. For example, Y. Ohyama \cite{e} defined the Gordian complex of knots by $C_k$-move (see Fig.~\ref{fig4}) for $k\in \mathbb{N}$ and proved that there exists no $n$-simplex for $n\geq2$ in the $C_2$-Gordian complex (follows from~\cite[Proposition~2.3]{e}).
\begin{figure}[th] 
\centerline{\includegraphics[width=0.9\linewidth]{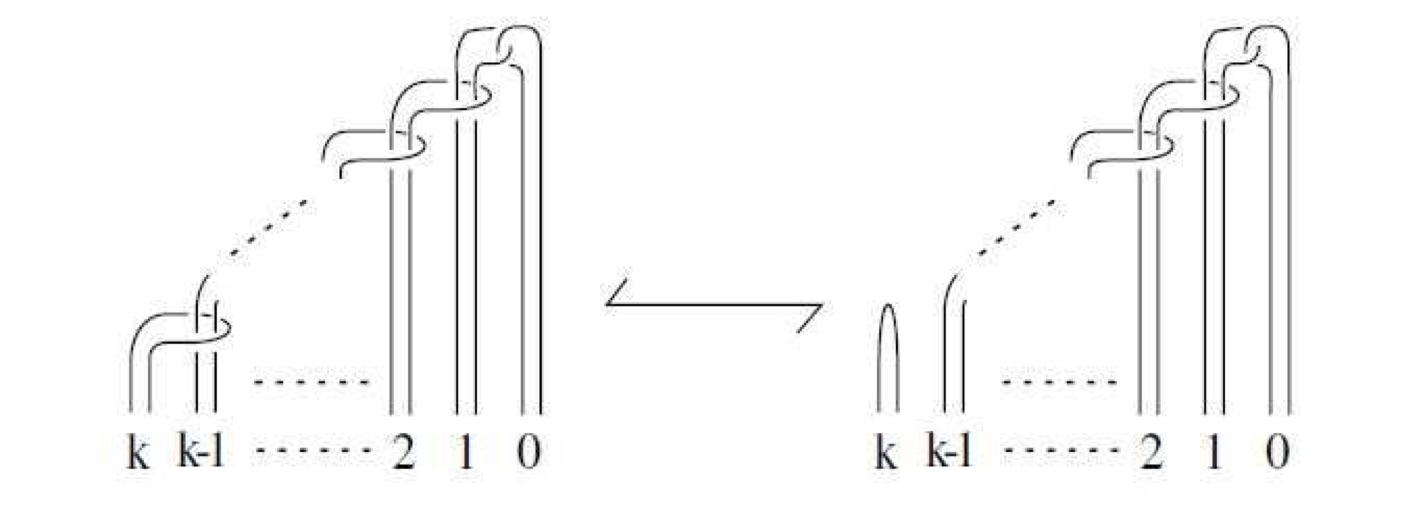}} 
\caption{$C_{k}$-Move.} \label{fig4}
\end{figure} 

The Fig.~\ref{fig5} presents $\sharp$-\emph{move}, an another local move defined in knot diagrams. H.~Murakami~\cite{k} proved that $\sharp$-move is an unknotting operation and defined distance $d_\sharp(K,K')$ between knots by $\sharp$-move. 
\begin{figure}
\unitlength=.28mm
\begin{center} 
\begin{picture}(0,70)(0,0)
\put(0,0){\begin{picture}(0,60)
\thicklines
\qbezier(-60,0)(-60,0)(-60,60)
\qbezier(-40,0)(-40,0)(-40,60)
\qbezier(-80,40)(-65,40)(-65,40)
\qbezier(-55,40)(-55,40)(-45,40)
\qbezier(-35,40)(-35,40)(-20,40)
\qbezier(-80,20)(-65,20)(-65,20)
\qbezier(-55,20)(-55,20)(-45,20)
\qbezier(-35,20)(-35,20)(-20,20)
\put(-60,55){\vector(0,1){5}}
\put(-40,55){\vector(0,1){5}}
\put(-25,40){\vector(1,0){5}}
\put(-25,20){\vector(1,0){5}}
\put(0,30){\makebox(0,0)[c,c]{$\leftrightarrow$}}
\qbezier(20,40)(20,40)(80,40)
\qbezier(20,20)(20,20)(80,20)
\qbezier(40,0)(40,0)(40,15)
\qbezier(40,25)(40,25)(40,35)
\qbezier(40,45)(40,45)(40,60)
\qbezier(60,0)(60,0)(60,15)
\qbezier(60,25)(60,25)(60,35)
\qbezier(60,45)(60,45)(60,60)
\put(40,55){\vector(0,1){5}}
\put(60,55){\vector(0,1){5}}
\put(75,40){\vector(1,0){5}}
\put(75,20){\vector(1,0){5}}
\end{picture}}
\end{picture}
\end{center}
\caption{$\#$-Move.} \label{fig5} 
\end{figure}
It was shown in \cite{k} that for two knots $K$ and $K'$ we have $d_\sharp(K,K') \equiv a_0(K)+a_0(K') (mod \,2)$, where $a_0(K)$ denotes Arf invariant of knot $K$. Therefore, for two knots $K$ and $K'$ having $d_\sharp(K,K')= 1$, Arf invariants $a_0(K)$ and $a_0(K')$ have different parity modulo 2. As a consequence there can exist no 2-simplex in Gordian complex by $\sharp$-move and in general it follows that there does not exist simplexes of dimension $n\geq2$ in the $\sharp$-move Gordian complex.

Virtual knot theory introduced by L.~Kauffman \cite{d} is an extension of classical knots to study knots embedded in thickened  surfaces. Forbidden moves and $v$-move are well known unknotting operations for virtual knots. S.~Horiuchi, K.~Komura, Y.~Ohyama and M.~Shimozawa \cite{b} extended the concept of Gordian complex to virtual knots using $v$-move which changes a classical crossing into virtual crossing. Later S.~Horiuchi and Y.~Ohyama \cite{c} defined Gordian complex of virtual knots using forbidden moves. In both Gordian complexes of virtual knots \cite{b,c}, it was proved that every $0$-simplex $\{K_0\}$ is always contained in a $n$-simplex for each positive integer $n$.

In the present paper, we firstly study Gordian complex of knots by the local move called \emph{region crossing change} often abbreviated as r.c.c.. A. Shimizu \cite{h} proved that every knot diagram can be converted into trivial knot diagram using finite number of r.c.c. operations. Local moves like $\sharp$-move, pass-move, $\overline{\sharp} $-move and $n$-gon move are in fact examples of r.c.c. on specific regions, so r.c.c. is a more general local move.  In Section~2 we survey known results and proof properties of the Gordian complex of a region crossing change in Theorem~\ref{thm1} and Corollary~\ref{corollary2-4}. The proof of the theorem is bases on calculations of Kawauchi's invariant of knots. 
In Section~3 we define the Gordian complex of virtual knots by \emph{arc shift move} and describe its properties in Theorem~\ref{1} and Corollary~\ref{2}. Proof of these statements are presented in Section~4. In Section~5 we calculate Kauffman's affine index polynomials for the  family of virtual knots used in the proof of Theorem~\ref{1} and show that our virtual knots have the same polynomial, see Proposition~\ref{prop51}. 

\section{Gordian complex by region crossing change}

A region in the knot diagram is a connected component of the complement of the projection of knot diagram in $\mathbb{R}^2$ obtained by replacing each crossing with a vertex. \emph{Region crossing change} (r.c.c.) at a region $R$ is defined as changing all the crossings lying at the boundary $\partial R$ of region $R$ (see Fig.~\ref{fig6}). It is a simple observation that r.c.c. applied twice at the same region $R$ in a knot diagram $D$ results back the diagram $D$ itself. We remark here that local moves like $\sharp$-move, pass-move, $\overline{\sharp}$-move are examples of region crossing change at specific regions.
\begin{figure}[h] 
\centerline{\includegraphics[width=0.6\linewidth]{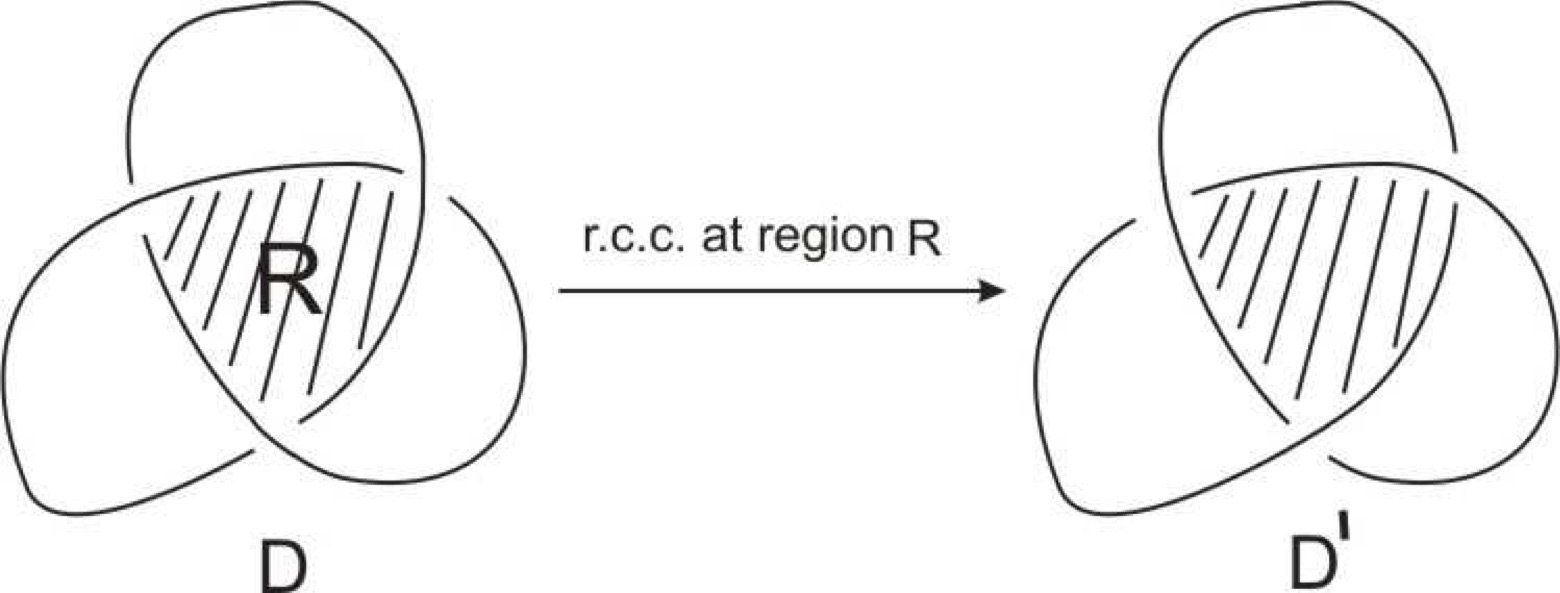}} 
\caption{$D'$ obtained from $D$ by r.c.c. at region $R$.} \label{fig6}
\end{figure} 

A.~Shimizu \cite[Theorem~1.1]{h} proved that a crossing change at any crossing $c$ in a knot diagram $D$ can be realized by applying r.c.c. at finite number of regions in $D$. Therefore, r.c.c. is an unknotting operation for knots. From a result of H.~Aida \cite[Theorem~2.5]{i} it follows that for every knot $K$, there exist a diagram $D$ such that one r.c.c. in $D$ results in the trivial knot. However, such a diagram $D$  may not always be a minimal diagram of $K$. Therefore, it may not be always possible to find two minimal diagrams $D$, $D'$ of the knots $K$, $K'$ such that $D$ can be converted into $D'$ using finite number of r.c.c.. Since r.c.c. is an unknotting operation, so for any two knots say $K$, $K'$ we have regions $\{R_i\}_{i=1}^{n}$ in $K$ and regions $\{R'_i\}_{i=1}^{m}$ in $K'$ such that $K(R_1...(R_n))$ and $K'(R'_1...(R'_m))$ are trivial knot diagrams. Here $K(R_1 \ldots (R_n))$ and $K'(R'_1 \ldots (R'_m))$ denotes respectively the knot diagrams obtained by applying r.c.c. at regions $\{R_i\}_{i=1}^{n}$ in $K$ and $\{R'_i\}_{i=1}^{m}$ in $K'$. Now consider the equivalent knot diagrams $D$ of $K$ and $D'$ of $K'$ as shown in the Fig.~\ref{fig7}.
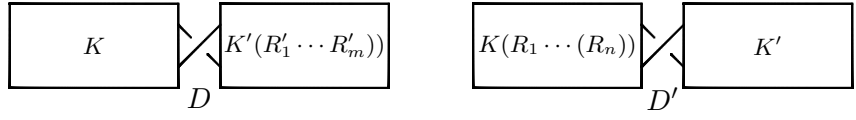
\begin{figure}
\unitlength=.28mm
\begin{center} 
\begin{picture}(0,50)(0,0)
\put(-110,0){\begin{picture}(0,40)
\thicklines
\qbezier(-90,0)(-90,0)(-90,40)
\qbezier(-90,0)(-90,0)(-10,0)
\qbezier(-90,40)(-90,40)(-10,40)
\qbezier(-10,0)(-10,0)(-10,40)
\put(-50,20){\makebox(0,0)[cc]{\footnotesize $K$}}
\qbezier(-10,10)(-10,10)(10,30)
\qbezier(-10,30)(-10,30)(-4,24)
\qbezier(4,16)(4,16)(10,10)
\qbezier(10,0)(10,0)(90,0)
\qbezier(10,0)(10,0)(10,40)
\qbezier(10,40)(10,40)(90,40)
\qbezier(90,0)(90,0)(90,40)
\put(50,20){\makebox(0,0)[cc]{\footnotesize $K' (R'_{1} \cdots R'_{m}))$}}
\put(0,-5){\makebox(0,0)[cc]{$D$}}
\end{picture}}
\put(110,0){\begin{picture}(0,60)
\thicklines
\qbezier(-90,0)(-90,0)(-90,40)
\qbezier(-90,0)(-90,0)(-10,0)
\qbezier(-90,40)(-90,40)(-10,40)
\qbezier(-10,0)(-10,0)(-10,40)
\put(-50,20){\makebox(0,0)[cc]{\footnotesize $K(R_{1} \cdots (R_{n}))$}}
\qbezier(-10,10)(-10,10)(10,30)
\qbezier(-10,30)(-10,30)(-4,24)
\qbezier(4,16)(4,16)(10,10)
\qbezier(10,0)(10,0)(90,0)
\qbezier(10,0)(10,0)(10,40)
\qbezier(10,40)(10,40)(90,40)
\qbezier(90,0)(90,0)(90,40)
\put(50,20){\makebox(0,0)[cc]{\footnotesize $K'$}}
\put(0,-5){\makebox(0,0)[cc]{$D'$}}
\end{picture}}
\end{picture}
\end{center}
\caption{Equivalent diagrams of $D$ and $D'$.} \label{fig7} 
\end{figure}
Observe that r.c.c. at regions $\{R_i\}_{i=1}^{n} \cup \{R'_i\}_{i=1}^{m}$ in $D$ results in $D'$ and r.c.c. at same set of regions in $D'$ results in $D$. Thus for any two knots $K$ and $K'$, it is always possible to find equivalent diagrams $D$ of $K$ and $D'$ of $K'$ such that $D$ can be converted into $D'$ via finite number of region crossing changes. So we define for any two knots $K$, $K'$ distance between $K, K'$ by r.c.c as, $d_R(K,K')$ = Minimum no. of r.c.c. required to convert all such diagrams $D$ into $D'$.

Gordian complex of knots by region crossing change, $\mathcal{G}_R$, is similarly defined as in the case of other local moves, i.e., \emph{vertex set} of $\mathcal{G}_R$ is the set of all knot isotopy classes and, \emph{n-simplex} in $\mathcal{G}_R$ is the collection of knots $\{K_0,...,K_n\}$ such that $d_R(K_i,K_j)=1$ for all $i\neq j \in \{0,1,\ldots,n\}$. Denote unknot by $U$. 

\begin{example}
All pairs of knots $\{U, K\}$ for non trivial knot $K$ have $d_R(U,K)=1$ and hence form 1-simplex in $\mathcal{G}_R$.
\end{example}

\begin{remark}
For any two distinct knots $K$ and $K'$, we have $d_R(K,K') \leq d_R(K,U)+d_R(U,K') = 2$, thus $\mathcal{G}_R$ is bounded complex in the sense that any two vertices can be joined by a path having at most two edges. 
\end{remark}

Now we describe properties of complex $\mathcal{G}_R$. 

\begin{theorem}\label{thm1}
For any knot $K_{0}$ and any integer $n \geq 1$ the exists an  $n$-simplex in $\mathcal G_{R}$ with a vertex $K_{0}$. 
\end{theorem}

\begin{proof}
It is enough to prove the statement for the case the unknot, $K_0 = U$. The general statement will follow by taking connected sums. Consider the family of knots $\sigma_n = \{K_0, K_1, \ldots, K_n \}$ where $K_m$ for $m=1,2, \ldots, n$ is given by Fig.~\ref{fig8}.  
\begin{figure}
\unitlength=.28mm
\centering 
\begin{picture}(0,110)(0,0)
\put(-190,0){\begin{picture}(150,0)
\thicklines
\qbezier(-15,0)(-15,0)(130,0)
\qbezier(-15,0)(-20,0)(-20,5)
\qbezier(-20,5)(-20,5)(-20,40)
\qbezier(-20,60)(-20,60)(-20,95)
\qbezier(-20,95)(-20,100)(-15,100)
\qbezier(-15,100)(-15,100)(85,100)
\qbezier(85,100)(90,100)(90,95)
\qbezier(90,95)(90,95)(90,65)
\qbezier(90,55)(90,55)(90,25)
\qbezier(90,15)(90,10)(95,10)
\qbezier(95,10)(95,10)(130,10)
\qbezier(-20,40)(-20,40)(0,60)
\qbezier(-20,60)(-20,60)(-14,54)
\qbezier(-6,46)(-6,46)(0,40)
\qbezier(0,40)(0,40)(20,60)
\qbezier(0,60)(0,60)(6,54)
\qbezier(14,46)(14,46)(20,40)
\put(6,54){\vector(-1,1){6}}
\put(14,54){\vector(1,1){6}}
\qbezier(20,60)(20,60)(20,85)
\qbezier(20,85)(20,90)(25,90)
\qbezier(25,90)(25,90)(65,90)
\qbezier(65,90)(70,90)(70,85)
\qbezier(70,85)(70,85)(70,65)
\qbezier(20,40)(20,40)(20,15)
\qbezier(20,15)(20,10)(25,10)
\qbezier(25,10)(25,10)(65,10)
\qbezier(65,10)(70,10)(70,15)
\qbezier(70,25)(70,25)(70,55)
\qbezier(130,100)(130,100)(115,100)
\qbezier(115,100)(110,100)(110,95)
\qbezier(110,95)(110,95)(110,85)
\qbezier(110,85)(110,80)(105,80)
\qbezier(105,80)(105,80)(95,80)
\qbezier(85,80)(85,80)(75,80)
\qbezier(65,80)(65,80)(45,80)
\qbezier(45,80)(40,80)(40,75)
\qbezier(40,75)(40,75)(40,60)
\qbezier(40,60)(40,60)(60,40)
\qbezier(60,40)(60,40)(65,40)
\qbezier(75,40)(75,40)(85,40)
\qbezier(95,40)(100,45)(100,45)
\qbezier(100,45)(100,45)(100,55)
\qbezier(100,55)(100,60)(95,60)
\qbezier(95,60)(95,60)(60,60)
\qbezier(60,60)(60,60)(54,54)
\qbezier(46,46)(46,46)(40,40)
\qbezier(40,40)(40,40)(40,25)
\qbezier(40,25)(40,20)(45,20)
\qbezier(45,20)(45,20)(115,20)
\qbezier(115,20)(120,20)(120,25)
\qbezier(120,25)(120,25)(120,85)
\qbezier(120,85)(120,90)(125,90)
\qbezier(125,90)(125,90)(130,90)
\put(80,70){\makebox(0,0)[cc]{$R_{1}$}}
\end{picture}}
\put(-90,0){\begin{picture}(0,100)
\thicklines
\qbezier(20,0)(20,0)(130,0)
\qbezier(20,100)(20,100)(85,100)
\qbezier(85,100)(90,100)(90,95)
\qbezier(90,95)(90,95)(90,65)
\qbezier(90,55)(90,55)(90,25)
\qbezier(90,15)(90,10)(95,10)
\qbezier(95,10)(95,10)(130,10)
\qbezier(20,60)(20,60)(20,85)
\qbezier(20,85)(20,90)(25,90)
\qbezier(25,90)(25,90)(65,90)
\qbezier(65,90)(70,90)(70,85)
\qbezier(70,85)(70,85)(70,65)
\qbezier(20,10)(20,10)(65,10)
\qbezier(65,10)(70,10)(70,15)
\qbezier(70,25)(70,25)(70,55)
\qbezier(130,100)(130,100)(115,100)
\qbezier(115,100)(110,100)(110,95)
\qbezier(110,95)(110,95)(110,85)
\qbezier(110,85)(110,80)(105,80)
\qbezier(105,80)(105,80)(95,80)
\qbezier(85,80)(85,80)(75,80)
\qbezier(65,80)(65,80)(45,80)
\qbezier(45,80)(40,80)(40,75)
\qbezier(40,75)(40,75)(40,60)
\qbezier(40,60)(40,60)(60,40)
\qbezier(60,40)(60,40)(65,40)
\qbezier(75,40)(75,40)(85,40)
\qbezier(95,40)(100,45)(100,45)
\qbezier(100,45)(100,45)(100,55)
\qbezier(100,55)(100,60)(95,60)
\qbezier(95,60)(95,60)(60,60)
\qbezier(60,60)(60,60)(54,54)
\qbezier(46,46)(46,46)(40,40)
\qbezier(40,40)(40,40)(40,25)
\qbezier(40,25)(40,20)(45,20)
\qbezier(45,20)(45,20)(115,20)
\qbezier(115,20)(120,20)(120,25)
\qbezier(120,25)(120,25)(120,85)
\qbezier(120,85)(120,90)(125,90)
\qbezier(125,90)(125,90)(130,90)
\put(80,70){\makebox(0,0)[cc]{$R_{2}$}}
\put(150,100){\makebox(0,0)[cc]{$\cdots$}}
\put(150,90){\makebox(0,0)[cc]{$\cdots$}}
\put(150,10){\makebox(0,0)[cc]{$\cdots$}}
\put(150,0){\makebox(0,0)[cc]{$\cdots$}}
\end{picture}}
\put(60,0){\begin{picture}(0,100)
\thicklines
\qbezier(20,0)(20,0)(130,0)
\qbezier(20,100)(20,100)(85,100)
\qbezier(85,100)(90,100)(90,95)
\qbezier(90,95)(90,95)(90,65)
\qbezier(90,55)(90,55)(90,25)
\qbezier(90,15)(90,10)(95,10)
\qbezier(95,10)(95,10)(125,10)
\qbezier(20,90)(20,90)(65,90)
\qbezier(65,90)(70,90)(70,85)
\qbezier(70,85)(70,85)(70,65)
\qbezier(20,10)(20,10)(65,10)
\qbezier(65,10)(70,10)(70,15)
\qbezier(70,25)(70,25)(70,55)
\qbezier(130,100)(130,100)(115,100)
\qbezier(115,100)(110,100)(110,95)
\qbezier(110,95)(110,95)(110,85)
\qbezier(110,85)(110,80)(105,80)
\qbezier(105,80)(105,80)(95,80)
\qbezier(85,80)(85,80)(75,80)
\qbezier(65,80)(65,80)(45,80)
\qbezier(45,80)(40,80)(40,75)
\qbezier(40,75)(40,75)(40,60)
\qbezier(40,60)(40,60)(60,40)
\qbezier(60,40)(60,40)(65,40)
\qbezier(75,40)(75,40)(85,40)
\qbezier(95,40)(100,45)(100,45)
\qbezier(100,45)(100,45)(100,55)
\qbezier(100,55)(100,60)(95,60)
\qbezier(95,60)(95,60)(60,60)
\qbezier(60,60)(60,60)(54,54)
\qbezier(46,46)(46,46)(40,40)
\qbezier(40,40)(40,40)(40,25)
\qbezier(40,25)(40,20)(45,20)
\qbezier(45,20)(45,20)(115,20)
\qbezier(115,20)(120,20)(120,25)
\qbezier(120,25)(120,25)(120,85)
\qbezier(120,85)(120,90)(125,90)
\qbezier(125,90)(130,90)(130,85)
\put(80,70){\makebox(0,0)[cc]{$R_{m}$}}
\qbezier(125,10)(130,10)(130,15)
\qbezier(130,15)(130,15)(130,85)
\qbezier(130,0)(140,0)(140,10)
\qbezier(140,10)(140,10)(140,90)
\qbezier(130,100)(140,100)(140,90)
\end{picture}}
\end{picture}
\caption{Knot $K_{m}$.} \label{fig8} 
\end{figure}
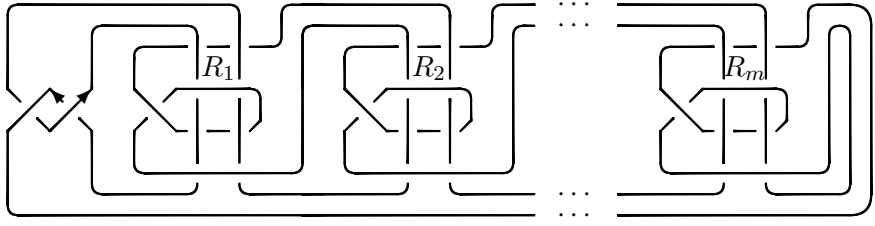 
It is easy to see that the diagram of $K_{m}$ has $m$ equal blocks. 
For any fix $m=1,2, \ldots,n$ observe from the diagram of $K_m$ that, r.c.c at region $R_i$ in $K_m$ results in the knot $K_{i-1}$ for $i=1,2,\ldots,m$ respectively. Therefore $d_R(K_i, K_m) \leq 1$ for $i=0,1, \ldots,m-1$, now varying $m$ from $1$ to $n$ we get $d_R(K_i, K_j) \leq 1$ for $i\neq j \in \{0,1, \ldots,n\}$. To show that $\sigma_n$ is an $n$-simplex in $\mathcal{G}_R$, it is enough to prove that $\{K_m\}_{m=0}^n$ are distinct knots.

In~\cite{j} A.~Kawauchi introduced the sequence of polynomial invariants \linebreak $\{c_n(L; x)\}_{n\geq 0}$ of oriented link $L$ which arise from a series representation of HOMFLY polynomial $P_L(l,m)$ multiplied by $\left(lm\right)^{r-1}$ where $L$ is assumed to be a $r$ component link. $\{c_n(L; x)\}_{n\geq 0}$ are called the \emph{$n$-th coefficient polynomial} of the HOMFLY polynomial $P_L(l,m)$. We compute $c_0(L; x)$ by considering $L$ as the knot $K_m$ and use it to show that knots $K_0, K_1, \ldots, K_n$ are distinct. $c_0(L; x)$ can be computed using the following rules as follows from \cite[Theorem~1.1]{j}.
\begin{enumerate}
\item $c_0(U; x)= 1$.
\item $x  \cdot c_0(L_+;x) - c_0(L_-;x) = c_0(L_0;x)$, where $L_+$, $L_-$ are knots and $L_0$ is the link as shown in the Fig.~\ref{fig9}.
\item For a 2-component link $L = K_1 \cup K_2$ having linking number $\lambda$, $c_0(L;x) = (x-1) x^{-\lambda} \cdot c_0(K_1;x)  \cdot c_0(K_2;x)$.
\end{enumerate}
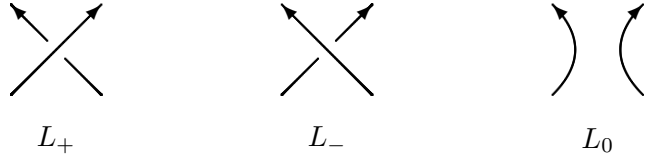
\begin{figure}[!ht]
\centering 
\unitlength=0.6mm
\begin{picture}(0,30)(0,0)
\thicklines
\qbezier(-70,10)(-70,10)(-50,30)
\qbezier(-70,30)(-70,30)(-62,22) 
\qbezier(-50,10)(-50,10)(-58,18)
\put(-65,25){\vector(-1,1){5}}
\put(-55,25){\vector(1,1){5}}
\qbezier(10,10)(10,10)(-10,30)
\qbezier(10,30)(10,30)(2,22) 
\qbezier(-10,10)(-10,10)(-2,18)
\put(-5,25){\vector(-1,1){5}}
\put(5,25){\vector(1,1){5}}
\put(-60,0){\makebox(0,0)[cc]{$L_{+}$}}
\put(0,0){\makebox(0,0)[cc]{$L_{-}$}}
\put(60,0){\makebox(0,0)[cc]{$L_{0}$}}
\qbezier(50,10)(60,20)(50,30)
\qbezier(70,10)(60,20)(70,30) 
\put(55,25){\vector(-1,1){5}}
\put(65,25){\vector(1,1){5}}
\end{picture}
\caption{Skein triple.} \label{fig9}
\end{figure}

Now we consider $K_{m}$ as an oriented knot with the orientation given by vectors presented in Fig.~\ref{fig8}. Consider $K_{m}$ as a knot in the skein triple which with the respect to the crossing indicated by vectors. Two other knots and links form the skein triple, namely knot $K_{m_{-}}$ and link $K_{m_{0}}$ are presented in Fig.~\ref{fig10} and Fig.~\ref{fig11} respectively. Using the rule ($2$) we compute 
$$
c_{0} (K_{m}) = \frac{1}{x} \left( c_{0} (K_{m_{-}}) + c_{0} (K_{m_{0}}) \right). 
$$
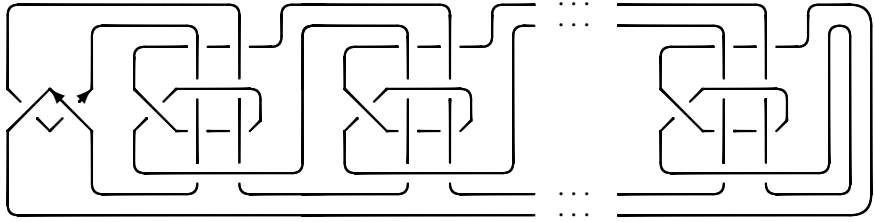
\begin{figure}
\unitlength=.28mm
\centering 
\begin{picture}(0,110)(0,0)
\put(-190,0){\begin{picture}(150,0)
\thicklines
\qbezier(-15,0)(-15,0)(130,0)
\qbezier(-15,0)(-20,0)(-20,5)
\qbezier(-20,5)(-20,5)(-20,40)
\qbezier(-20,60)(-20,60)(-20,95)
\qbezier(-20,95)(-20,100)(-15,100)
\qbezier(-15,100)(-15,100)(85,100)
\qbezier(85,100)(90,100)(90,95)
\qbezier(90,95)(90,95)(90,65)
\qbezier(90,55)(90,55)(90,25)
\qbezier(90,15)(90,10)(95,10)
\qbezier(95,10)(95,10)(130,10)
\qbezier(-20,40)(-20,40)(0,60)
\qbezier(-20,60)(-20,60)(-14,54)
\qbezier(-6,46)(-6,46)(0,40)
\qbezier(0,60)(0,60)(20,40)
\qbezier(0,40)(0,40)(6,46)
\qbezier(14,54)(14,54)(20,60)
\put(6,54){\vector(-1,1){6}}
\put(14,54){\vector(1,1){6}}
\qbezier(20,60)(20,60)(20,85)
\qbezier(20,85)(20,90)(25,90)
\qbezier(25,90)(25,90)(65,90)
\qbezier(65,90)(70,90)(70,85)
\qbezier(70,85)(70,85)(70,65)
\qbezier(20,40)(20,40)(20,15)
\qbezier(20,15)(20,10)(25,10)
\qbezier(25,10)(25,10)(65,10)
\qbezier(65,10)(70,10)(70,15)
\qbezier(70,25)(70,25)(70,55)
\qbezier(130,100)(130,100)(115,100)
\qbezier(115,100)(110,100)(110,95)
\qbezier(110,95)(110,95)(110,85)
\qbezier(110,85)(110,80)(105,80)
\qbezier(105,80)(105,80)(95,80)
\qbezier(85,80)(85,80)(75,80)
\qbezier(65,80)(65,80)(45,80)
\qbezier(45,80)(40,80)(40,75)
\qbezier(40,75)(40,75)(40,60)
\qbezier(40,60)(40,60)(60,40)
\qbezier(60,40)(60,40)(65,40)
\qbezier(75,40)(75,40)(85,40)
\qbezier(95,40)(100,45)(100,45)
\qbezier(100,45)(100,45)(100,55)
\qbezier(100,55)(100,60)(95,60)
\qbezier(95,60)(95,60)(60,60)
\qbezier(60,60)(60,60)(54,54)
\qbezier(46,46)(46,46)(40,40)
\qbezier(40,40)(40,40)(40,25)
\qbezier(40,25)(40,20)(45,20)
\qbezier(45,20)(45,20)(115,20)
\qbezier(115,20)(120,20)(120,25)
\qbezier(120,25)(120,25)(120,85)
\qbezier(120,85)(120,90)(125,90)
\qbezier(125,90)(125,90)(130,90)
\end{picture}}
\put(-90,0){\begin{picture}(0,100)
\thicklines
\qbezier(20,0)(20,0)(130,0)
\qbezier(20,100)(20,100)(85,100)
\qbezier(85,100)(90,100)(90,95)
\qbezier(90,95)(90,95)(90,65)
\qbezier(90,55)(90,55)(90,25)
\qbezier(90,15)(90,10)(95,10)
\qbezier(95,10)(95,10)(130,10)
\qbezier(20,60)(20,60)(20,85)
\qbezier(20,85)(20,90)(25,90)
\qbezier(25,90)(25,90)(65,90)
\qbezier(65,90)(70,90)(70,85)
\qbezier(70,85)(70,85)(70,65)
\qbezier(20,10)(20,10)(65,10)
\qbezier(65,10)(70,10)(70,15)
\qbezier(70,25)(70,25)(70,55)
\qbezier(130,100)(130,100)(115,100)
\qbezier(115,100)(110,100)(110,95)
\qbezier(110,95)(110,95)(110,85)
\qbezier(110,85)(110,80)(105,80)
\qbezier(105,80)(105,80)(95,80)
\qbezier(85,80)(85,80)(75,80)
\qbezier(65,80)(65,80)(45,80)
\qbezier(45,80)(40,80)(40,75)
\qbezier(40,75)(40,75)(40,60)
\qbezier(40,60)(40,60)(60,40)
\qbezier(60,40)(60,40)(65,40)
\qbezier(75,40)(75,40)(85,40)
\qbezier(95,40)(100,45)(100,45)
\qbezier(100,45)(100,45)(100,55)
\qbezier(100,55)(100,60)(95,60)
\qbezier(95,60)(95,60)(60,60)
\qbezier(60,60)(60,60)(54,54)
\qbezier(46,46)(46,46)(40,40)
\qbezier(40,40)(40,40)(40,25)
\qbezier(40,25)(40,20)(45,20)
\qbezier(45,20)(45,20)(115,20)
\qbezier(115,20)(120,20)(120,25)
\qbezier(120,25)(120,25)(120,85)
\qbezier(120,85)(120,90)(125,90)
\qbezier(125,90)(125,90)(130,90)
%
\put(150,100){\makebox(0,0)[cc]{$\cdots$}}
\put(150,90){\makebox(0,0)[cc]{$\cdots$}}
\put(150,10){\makebox(0,0)[cc]{$\cdots$}}
\put(150,0){\makebox(0,0)[cc]{$\cdots$}}
\end{picture}}
\put(60,0){\begin{picture}(0,100)
\thicklines
\qbezier(20,0)(20,0)(130,0)
\qbezier(20,100)(20,100)(85,100)
\qbezier(85,100)(90,100)(90,95)
\qbezier(90,95)(90,95)(90,65)
\qbezier(90,55)(90,55)(90,25)
\qbezier(90,15)(90,10)(95,10)
\qbezier(95,10)(95,10)(125,10)
\qbezier(20,90)(20,90)(65,90)
\qbezier(65,90)(70,90)(70,85)
\qbezier(70,85)(70,85)(70,65)
\qbezier(20,10)(20,10)(65,10)
\qbezier(65,10)(70,10)(70,15)
\qbezier(70,25)(70,25)(70,55)
\qbezier(130,100)(130,100)(115,100)
\qbezier(115,100)(110,100)(110,95)
\qbezier(110,95)(110,95)(110,85)
\qbezier(110,85)(110,80)(105,80)
\qbezier(105,80)(105,80)(95,80)
\qbezier(85,80)(85,80)(75,80)
\qbezier(65,80)(65,80)(45,80)
\qbezier(45,80)(40,80)(40,75)
\qbezier(40,75)(40,75)(40,60)
\qbezier(40,60)(40,60)(60,40)
\qbezier(60,40)(60,40)(65,40)
\qbezier(75,40)(75,40)(85,40)
\qbezier(95,40)(100,45)(100,45)
\qbezier(100,45)(100,45)(100,55)
\qbezier(100,55)(100,60)(95,60)
\qbezier(95,60)(95,60)(60,60)
\qbezier(60,60)(60,60)(54,54)
\qbezier(46,46)(46,46)(40,40)
\qbezier(40,40)(40,40)(40,25)
\qbezier(40,25)(40,20)(45,20)
\qbezier(45,20)(45,20)(115,20)
\qbezier(115,20)(120,20)(120,25)
\qbezier(120,25)(120,25)(120,85)
\qbezier(120,85)(120,90)(125,90)
\qbezier(125,90)(130,90)(130,85)
%
\qbezier(125,10)(130,10)(130,15)
\qbezier(130,15)(130,15)(130,85)
\qbezier(130,0)(140,0)(140,10)
\qbezier(140,10)(140,10)(140,90)
\qbezier(130,100)(140,100)(140,90)
\end{picture}}
\end{picture}
\caption{Knot $K_{m_{-}}$.} \label{fig10} 
\end{figure}
In is easy to see that $K_{m_{-}}$ is the unknot, and $K_{m_{0}}$ is a 2-component link with the linking number equals to zero. 
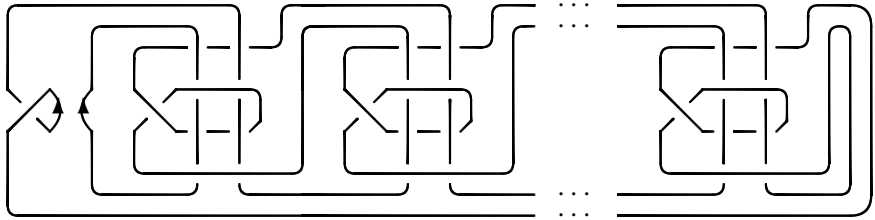
\begin{figure}
\unitlength=.28mm
\centering 
\begin{picture}(0,110)(0,0)
\put(-190,0){\begin{picture}(150,0)
\thicklines
\qbezier(-15,0)(-15,0)(130,0)
\qbezier(-15,0)(-20,0)(-20,5)
\qbezier(-20,5)(-20,5)(-20,40)
\qbezier(-20,60)(-20,60)(-20,95)
\qbezier(-20,95)(-20,100)(-15,100)
\qbezier(-15,100)(-15,100)(85,100)
\qbezier(85,100)(90,100)(90,95)
\qbezier(90,95)(90,95)(90,65)
\qbezier(90,55)(90,55)(90,25)
\qbezier(90,15)(90,10)(95,10)
\qbezier(95,10)(95,10)(130,10)
\qbezier(-20,40)(-20,40)(0,60)
\qbezier(-20,60)(-20,60)(-14,54)
\qbezier(-6,46)(-6,46)(0,40)
\qbezier(0,40)(10,50)(0,60)
\qbezier(20,40)(10,50)(20,60)
\put(4,50){\vector(0,1){6}}
\put(16,50){\vector(0,1){6}}
\qbezier(20,60)(20,60)(20,85)
\qbezier(20,85)(20,90)(25,90)
\qbezier(25,90)(25,90)(65,90)
\qbezier(65,90)(70,90)(70,85)
\qbezier(70,85)(70,85)(70,65)
\qbezier(20,40)(20,40)(20,15)
\qbezier(20,15)(20,10)(25,10)
\qbezier(25,10)(25,10)(65,10)
\qbezier(65,10)(70,10)(70,15)
\qbezier(70,25)(70,25)(70,55)
\qbezier(130,100)(130,100)(115,100)
\qbezier(115,100)(110,100)(110,95)
\qbezier(110,95)(110,95)(110,85)
\qbezier(110,85)(110,80)(105,80)
\qbezier(105,80)(105,80)(95,80)
\qbezier(85,80)(85,80)(75,80)
\qbezier(65,80)(65,80)(45,80)
\qbezier(45,80)(40,80)(40,75)
\qbezier(40,75)(40,75)(40,60)
\qbezier(40,60)(40,60)(60,40)
\qbezier(60,40)(60,40)(65,40)
\qbezier(75,40)(75,40)(85,40)
\qbezier(95,40)(100,45)(100,45)
\qbezier(100,45)(100,45)(100,55)
\qbezier(100,55)(100,60)(95,60)
\qbezier(95,60)(95,60)(60,60)
\qbezier(60,60)(60,60)(54,54)
\qbezier(46,46)(46,46)(40,40)
\qbezier(40,40)(40,40)(40,25)
\qbezier(40,25)(40,20)(45,20)
\qbezier(45,20)(45,20)(115,20)
\qbezier(115,20)(120,20)(120,25)
\qbezier(120,25)(120,25)(120,85)
\qbezier(120,85)(120,90)(125,90)
\qbezier(125,90)(125,90)(130,90)
\end{picture}}
\put(-90,0){\begin{picture}(0,100)
\thicklines
\qbezier(20,0)(20,0)(130,0)
\qbezier(20,100)(20,100)(85,100)
\qbezier(85,100)(90,100)(90,95)
\qbezier(90,95)(90,95)(90,65)
\qbezier(90,55)(90,55)(90,25)
\qbezier(90,15)(90,10)(95,10)
\qbezier(95,10)(95,10)(130,10)
\qbezier(20,60)(20,60)(20,85)
\qbezier(20,85)(20,90)(25,90)
\qbezier(25,90)(25,90)(65,90)
\qbezier(65,90)(70,90)(70,85)
\qbezier(70,85)(70,85)(70,65)
\qbezier(20,10)(20,10)(65,10)
\qbezier(65,10)(70,10)(70,15)
\qbezier(70,25)(70,25)(70,55)
\qbezier(130,100)(130,100)(115,100)
\qbezier(115,100)(110,100)(110,95)
\qbezier(110,95)(110,95)(110,85)
\qbezier(110,85)(110,80)(105,80)
\qbezier(105,80)(105,80)(95,80)
\qbezier(85,80)(85,80)(75,80)
\qbezier(65,80)(65,80)(45,80)
\qbezier(45,80)(40,80)(40,75)
\qbezier(40,75)(40,75)(40,60)
\qbezier(40,60)(40,60)(60,40)
\qbezier(60,40)(60,40)(65,40)
\qbezier(75,40)(75,40)(85,40)
\qbezier(95,40)(100,45)(100,45)
\qbezier(100,45)(100,45)(100,55)
\qbezier(100,55)(100,60)(95,60)
\qbezier(95,60)(95,60)(60,60)
\qbezier(60,60)(60,60)(54,54)
\qbezier(46,46)(46,46)(40,40)
\qbezier(40,40)(40,40)(40,25)
\qbezier(40,25)(40,20)(45,20)
\qbezier(45,20)(45,20)(115,20)
\qbezier(115,20)(120,20)(120,25)
\qbezier(120,25)(120,25)(120,85)
\qbezier(120,85)(120,90)(125,90)
\qbezier(125,90)(125,90)(130,90)
%
\put(150,100){\makebox(0,0)[cc]{$\cdots$}}
\put(150,90){\makebox(0,0)[cc]{$\cdots$}}
\put(150,10){\makebox(0,0)[cc]{$\cdots$}}
\put(150,0){\makebox(0,0)[cc]{$\cdots$}}
\end{picture}}
\put(60,0){\begin{picture}(0,100)
\thicklines
\qbezier(20,0)(20,0)(130,0)
\qbezier(20,100)(20,100)(85,100)
\qbezier(85,100)(90,100)(90,95)
\qbezier(90,95)(90,95)(90,65)
\qbezier(90,55)(90,55)(90,25)
\qbezier(90,15)(90,10)(95,10)
\qbezier(95,10)(95,10)(125,10)
\qbezier(20,90)(20,90)(65,90)
\qbezier(65,90)(70,90)(70,85)
\qbezier(70,85)(70,85)(70,65)
\qbezier(20,10)(20,10)(65,10)
\qbezier(65,10)(70,10)(70,15)
\qbezier(70,25)(70,25)(70,55)
\qbezier(130,100)(130,100)(115,100)
\qbezier(115,100)(110,100)(110,95)
\qbezier(110,95)(110,95)(110,85)
\qbezier(110,85)(110,80)(105,80)
\qbezier(105,80)(105,80)(95,80)
\qbezier(85,80)(85,80)(75,80)
\qbezier(65,80)(65,80)(45,80)
\qbezier(45,80)(40,80)(40,75)
\qbezier(40,75)(40,75)(40,60)
\qbezier(40,60)(40,60)(60,40)
\qbezier(60,40)(60,40)(65,40)
\qbezier(75,40)(75,40)(85,40)
\qbezier(95,40)(100,45)(100,45)
\qbezier(100,45)(100,45)(100,55)
\qbezier(100,55)(100,60)(95,60)
\qbezier(95,60)(95,60)(60,60)
\qbezier(60,60)(60,60)(54,54)
\qbezier(46,46)(46,46)(40,40)
\qbezier(40,40)(40,40)(40,25)
\qbezier(40,25)(40,20)(45,20)
\qbezier(45,20)(45,20)(115,20)
\qbezier(115,20)(120,20)(120,25)
\qbezier(120,25)(120,25)(120,85)
\qbezier(120,85)(120,90)(125,90)
\qbezier(125,90)(130,90)(130,85)
%
\qbezier(125,10)(130,10)(130,15)
\qbezier(130,15)(130,15)(130,85)
\qbezier(130,0)(140,0)(140,10)
\qbezier(140,10)(140,10)(140,90)
\qbezier(130,100)(140,100)(140,90)
\end{picture}}
\end{picture}
\caption{Knot $K_{m_{0}}$.} \label{fig11} 
\end{figure}
Using rule ($3$) we get $c_0(K_{m_{0}})= (x-1) \cdot c_0(L_m)$, where $L_m$ is the knot which is a non-trivial component of the link $K_{m_{0}}$. The knot $L_{m}$ is presented in Fig.~\ref{fig12}. 
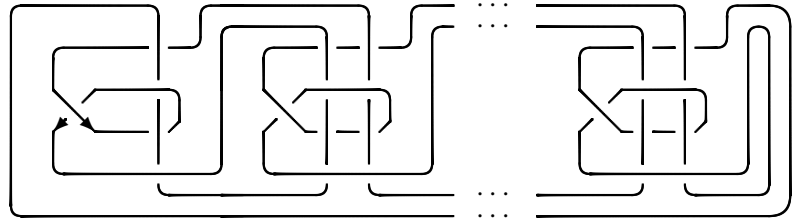
\begin{figure}
\unitlength=.28mm
\centering 
\begin{picture}(0,110)(0,0)
\put(-190,0){\begin{picture}(150,0)
\thicklines
\qbezier(25,0)(25,0)(130,0)
\qbezier(25,0)(20,0)(20,5)
\qbezier(20,5)(20,5)(20,95)
\qbezier(20,95)(20,100)(25,100)
\qbezier(25,100)(25,100)(85,100)
\qbezier(85,100)(90,100)(90,95)
\qbezier(90,95)(90,95)(90,65)
\qbezier(90,55)(90,55)(90,25)
\qbezier(90,15)(90,10)(95,10)
\qbezier(95,10)(95,10)(130,10)
\qbezier(130,100)(130,100)(115,100)
\qbezier(115,100)(110,100)(110,95)
\qbezier(110,95)(110,95)(110,85)
\qbezier(110,85)(110,80)(105,80)
\qbezier(105,80)(105,80)(95,80)
\qbezier(85,80)(85,80)(45,80)
\qbezier(45,80)(40,80)(40,75)
\qbezier(40,75)(40,75)(40,60)
\qbezier(40,60)(40,60)(60,40)
\qbezier(60,40)(60,40)(85,40)
\put(54,46){\vector(1,-1){6}}
\put(46,46){\vector(-1,-1){6}}
\qbezier(95,40)(100,45)(100,45)
\qbezier(100,45)(100,45)(100,55)
\qbezier(100,55)(100,60)(95,60)
\qbezier(95,60)(95,60)(60,60)
\qbezier(60,60)(60,60)(54,54)
\qbezier(46,46)(46,46)(40,40)
\qbezier(40,40)(40,40)(40,25)
\qbezier(40,25)(40,20)(45,20)
\qbezier(45,20)(45,20)(115,20)
\qbezier(115,20)(120,20)(120,25)
\qbezier(120,25)(120,25)(120,85)
\qbezier(120,85)(120,90)(125,90)
\qbezier(125,90)(125,90)(130,90)
\end{picture}}
\put(-90,0){\begin{picture}(0,100)
\thicklines
\qbezier(20,0)(20,0)(130,0)
\qbezier(20,100)(20,100)(85,100)
\qbezier(85,100)(90,100)(90,95)
\qbezier(90,95)(90,95)(90,65)
\qbezier(90,55)(90,55)(90,25)
\qbezier(90,15)(90,10)(95,10)
\qbezier(95,10)(95,10)(130,10)
\qbezier(20,60)(20,60)(20,85)
\qbezier(20,85)(20,90)(25,90)
\qbezier(25,90)(25,90)(65,90)
\qbezier(65,90)(70,90)(70,85)
\qbezier(70,85)(70,85)(70,65)
\qbezier(20,10)(20,10)(65,10)
\qbezier(65,10)(70,10)(70,15)
\qbezier(70,25)(70,25)(70,55)
\qbezier(130,100)(130,100)(115,100)
\qbezier(115,100)(110,100)(110,95)
\qbezier(110,95)(110,95)(110,85)
\qbezier(110,85)(110,80)(105,80)
\qbezier(105,80)(105,80)(95,80)
\qbezier(85,80)(85,80)(75,80)
\qbezier(65,80)(65,80)(45,80)
\qbezier(45,80)(40,80)(40,75)
\qbezier(40,75)(40,75)(40,60)
\qbezier(40,60)(40,60)(60,40)
\qbezier(60,40)(60,40)(65,40)
\qbezier(75,40)(75,40)(85,40)
\qbezier(95,40)(100,45)(100,45)
\qbezier(100,45)(100,45)(100,55)
\qbezier(100,55)(100,60)(95,60)
\qbezier(95,60)(95,60)(60,60)
\qbezier(60,60)(60,60)(54,54)
\qbezier(46,46)(46,46)(40,40)
\qbezier(40,40)(40,40)(40,25)
\qbezier(40,25)(40,20)(45,20)
\qbezier(45,20)(45,20)(115,20)
\qbezier(115,20)(120,20)(120,25)
\qbezier(120,25)(120,25)(120,85)
\qbezier(120,85)(120,90)(125,90)
\qbezier(125,90)(125,90)(130,90)
%
\put(150,100){\makebox(0,0)[cc]{$\cdots$}}
\put(150,90){\makebox(0,0)[cc]{$\cdots$}}
\put(150,10){\makebox(0,0)[cc]{$\cdots$}}
\put(150,0){\makebox(0,0)[cc]{$\cdots$}}
\end{picture}}
\put(60,0){\begin{picture}(0,100)
\thicklines
\qbezier(20,0)(20,0)(130,0)
\qbezier(20,100)(20,100)(85,100)
\qbezier(85,100)(90,100)(90,95)
\qbezier(90,95)(90,95)(90,65)
\qbezier(90,55)(90,55)(90,25)
\qbezier(90,15)(90,10)(95,10)
\qbezier(95,10)(95,10)(125,10)
\qbezier(20,90)(20,90)(65,90)
\qbezier(65,90)(70,90)(70,85)
\qbezier(70,85)(70,85)(70,65)
\qbezier(20,10)(20,10)(65,10)
\qbezier(65,10)(70,10)(70,15)
\qbezier(70,25)(70,25)(70,55)
\qbezier(130,100)(130,100)(115,100)
\qbezier(115,100)(110,100)(110,95)
\qbezier(110,95)(110,95)(110,85)
\qbezier(110,85)(110,80)(105,80)
\qbezier(105,80)(105,80)(95,80)
\qbezier(85,80)(85,80)(75,80)
\qbezier(65,80)(65,80)(45,80)
\qbezier(45,80)(40,80)(40,75)
\qbezier(40,75)(40,75)(40,60)
\qbezier(40,60)(40,60)(60,40)
\qbezier(60,40)(60,40)(65,40)
\qbezier(75,40)(75,40)(85,40)
\qbezier(95,40)(100,45)(100,45)
\qbezier(100,45)(100,45)(100,55)
\qbezier(100,55)(100,60)(95,60)
\qbezier(95,60)(95,60)(60,60)
\qbezier(60,60)(60,60)(54,54)
\qbezier(46,46)(46,46)(40,40)
\qbezier(40,40)(40,40)(40,25)
\qbezier(40,25)(40,20)(45,20)
\qbezier(45,20)(45,20)(115,20)
\qbezier(115,20)(120,20)(120,25)
\qbezier(120,25)(120,25)(120,85)
\qbezier(120,85)(120,90)(125,90)
\qbezier(125,90)(130,90)(130,85)
%
\qbezier(125,10)(130,10)(130,15)
\qbezier(130,15)(130,15)(130,85)
\qbezier(130,0)(140,0)(140,10)
\qbezier(140,10)(140,10)(140,90)
\qbezier(130,100)(140,100)(140,90)
\end{picture}}
\end{picture}
\caption{Knot $L_{m}$.} \label{fig12} 
\end{figure}

Now consider $L_{m}$ as a knot in a skein triple in respect to the crossing indicated by two vectors in Fig.~\ref{fig12}. This crossing point in $L_{m}$ is negative, hence we get the relation $c_0(L_{m})= x \cdot c_0(L_{m_{+}})- c_0(L_{m_{0}})$. Here $L_{m_{+}}$ is the unknot, and  link $L_{m_{0}}$, presented in Fig.~\ref{fig13}, has two components.  
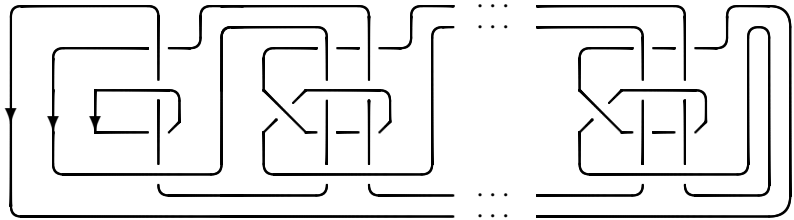
\begin{figure}
\unitlength=.28mm
\centering 
\begin{picture}(0,110)(0,0)
\put(-190,0){\begin{picture}(150,0)
\thicklines
\qbezier(25,0)(25,0)(130,0)
\qbezier(25,0)(20,0)(20,5)
\qbezier(20,5)(20,5)(20,95)
\qbezier(20,95)(20,100)(25,100)
\qbezier(25,100)(25,100)(85,100)
\qbezier(85,100)(90,100)(90,95)
\qbezier(90,95)(90,95)(90,65)
\qbezier(90,55)(90,55)(90,25)
\qbezier(90,15)(90,10)(95,10)
\qbezier(95,10)(95,10)(130,10)
\put(20,50){\vector(0,-1){6}}
\qbezier(130,100)(130,100)(115,100)
\qbezier(115,100)(110,100)(110,95)
\qbezier(110,95)(110,95)(110,85)
\qbezier(110,85)(110,80)(105,80)
\qbezier(105,80)(105,80)(95,80)
\qbezier(85,80)(85,80)(45,80)
\qbezier(45,80)(40,80)(40,75)
\qbezier(40,75)(40,75)(40,60)
\qbezier(40,60)(40,60)(40,40)
\qbezier(60,40)(60,40)(85,40)
\put(60,46){\vector(0,-1){6}}
\put(40,46){\vector(0,-1){6}}
\qbezier(95,40)(100,45)(100,45)
\qbezier(100,45)(100,45)(100,55)
\qbezier(100,55)(100,60)(95,60)
\qbezier(95,60)(95,60)(60,60)
\qbezier(60,60)(60,60)(60,40)
\qbezier(40,40)(40,40)(40,25)
\qbezier(40,25)(40,20)(45,20)
\qbezier(45,20)(45,20)(115,20)
\qbezier(115,20)(120,20)(120,25)
\qbezier(120,25)(120,25)(120,85)
\qbezier(120,85)(120,90)(125,90)
\qbezier(125,90)(125,90)(130,90)
\end{picture}}
\put(-90,0){\begin{picture}(0,100)
\thicklines
\qbezier(20,0)(20,0)(130,0)
\qbezier(20,100)(20,100)(85,100)
\qbezier(85,100)(90,100)(90,95)
\qbezier(90,95)(90,95)(90,65)
\qbezier(90,55)(90,55)(90,25)
\qbezier(90,15)(90,10)(95,10)
\qbezier(95,10)(95,10)(130,10)
\qbezier(20,60)(20,60)(20,85)
\qbezier(20,85)(20,90)(25,90)
\qbezier(25,90)(25,90)(65,90)
\qbezier(65,90)(70,90)(70,85)
\qbezier(70,85)(70,85)(70,65)
\qbezier(20,10)(20,10)(65,10)
\qbezier(65,10)(70,10)(70,15)
\qbezier(70,25)(70,25)(70,55)
\qbezier(130,100)(130,100)(115,100)
\qbezier(115,100)(110,100)(110,95)
\qbezier(110,95)(110,95)(110,85)
\qbezier(110,85)(110,80)(105,80)
\qbezier(105,80)(105,80)(95,80)
\qbezier(85,80)(85,80)(75,80)
\qbezier(65,80)(65,80)(45,80)
\qbezier(45,80)(40,80)(40,75)
\qbezier(40,75)(40,75)(40,60)
\qbezier(40,60)(40,60)(60,40)
\qbezier(60,40)(60,40)(65,40)
\qbezier(75,40)(75,40)(85,40)
\qbezier(95,40)(100,45)(100,45)
\qbezier(100,45)(100,45)(100,55)
\qbezier(100,55)(100,60)(95,60)
\qbezier(95,60)(95,60)(60,60)
\qbezier(60,60)(60,60)(54,54)
\qbezier(46,46)(46,46)(40,40)
\qbezier(40,40)(40,40)(40,25)
\qbezier(40,25)(40,20)(45,20)
\qbezier(45,20)(45,20)(115,20)
\qbezier(115,20)(120,20)(120,25)
\qbezier(120,25)(120,25)(120,85)
\qbezier(120,85)(120,90)(125,90)
\qbezier(125,90)(125,90)(130,90)
\put(150,100){\makebox(0,0)[cc]{$\cdots$}}
\put(150,90){\makebox(0,0)[cc]{$\cdots$}}
\put(150,10){\makebox(0,0)[cc]{$\cdots$}}
\put(150,0){\makebox(0,0)[cc]{$\cdots$}}
\end{picture}}
\put(60,0){\begin{picture}(0,100)
\thicklines
\qbezier(20,0)(20,0)(130,0)
\qbezier(20,100)(20,100)(85,100)
\qbezier(85,100)(90,100)(90,95)
\qbezier(90,95)(90,95)(90,65)
\qbezier(90,55)(90,55)(90,25)
\qbezier(90,15)(90,10)(95,10)
\qbezier(95,10)(95,10)(125,10)
\qbezier(20,90)(20,90)(65,90)
\qbezier(65,90)(70,90)(70,85)
\qbezier(70,85)(70,85)(70,65)
\qbezier(20,10)(20,10)(65,10)
\qbezier(65,10)(70,10)(70,15)
\qbezier(70,25)(70,25)(70,55)
\qbezier(130,100)(130,100)(115,100)
\qbezier(115,100)(110,100)(110,95)
\qbezier(110,95)(110,95)(110,85)
\qbezier(110,85)(110,80)(105,80)
\qbezier(105,80)(105,80)(95,80)
\qbezier(85,80)(85,80)(75,80)
\qbezier(65,80)(65,80)(45,80)
\qbezier(45,80)(40,80)(40,75)
\qbezier(40,75)(40,75)(40,60)
\qbezier(40,60)(40,60)(60,40)
\qbezier(60,40)(60,40)(65,40)
\qbezier(75,40)(75,40)(85,40)
\qbezier(95,40)(100,45)(100,45)
\qbezier(100,45)(100,45)(100,55)
\qbezier(100,55)(100,60)(95,60)
\qbezier(95,60)(95,60)(60,60)
\qbezier(60,60)(60,60)(54,54)
\qbezier(46,46)(46,46)(40,40)
\qbezier(40,40)(40,40)(40,25)
\qbezier(40,25)(40,20)(45,20)
\qbezier(45,20)(45,20)(115,20)
\qbezier(115,20)(120,20)(120,25)
\qbezier(120,25)(120,25)(120,85)
\qbezier(120,85)(120,90)(125,90)
\qbezier(125,90)(130,90)(130,85)
\qbezier(125,10)(130,10)(130,15)
\qbezier(130,15)(130,15)(130,85)
\qbezier(130,0)(140,0)(140,10)
\qbezier(140,10)(140,10)(140,90)
\qbezier(130,100)(140,100)(140,90)
\end{picture}}
\end{picture}
\caption{Link $L_{m_{0}}$.} \label{fig13} 
\end{figure}

Since linking number of $L_{m_{0}}$ is $-1$ and one of its component is trivial, we have $c_0(L_{m_{0}}) = x(x-1) \cdot c_0(L'_m)$ where by $L'_m$ we denoted the non-trivial component of $L_{m_{0}}$. This component diagram is  presented in Fig.~\ref{fig14}.
\begin{figure}
\unitlength=.28mm
\centering 
\begin{picture}(0,110)(0,0)
\put(-190,0){\begin{picture}(150,0)
\thicklines
\qbezier(25,0)(25,0)(130,0)
\qbezier(25,0)(20,0)(20,5)
\qbezier(20,5)(20,5)(20,95)
\qbezier(20,95)(20,100)(25,100)
\qbezier(25,100)(25,100)(85,100)
\qbezier(85,100)(90,100)(90,95)
\qbezier(90,95)(90,95)(90,25)
\qbezier(90,15)(90,10)(95,10)
\qbezier(95,10)(95,10)(130,10)
\put(75,80){\vector(-1,0){5}}
\put(90,90){\vector(0,1){5}}
\qbezier(130,100)(130,100)(115,100)
\qbezier(115,100)(110,100)(110,95)
\qbezier(110,95)(110,95)(110,85)
\qbezier(110,85)(110,80)(105,80)
\qbezier(105,80)(105,80)(95,80)
\qbezier(85,80)(85,80)(45,80)
\qbezier(45,80)(40,80)(40,75)
\qbezier(40,75)(40,75)(40,60)
\qbezier(40,60)(40,60)(40,40)
\qbezier(40,40)(40,40)(40,25)
\qbezier(40,25)(40,20)(45,20)
\qbezier(45,20)(45,20)(115,20)
\qbezier(115,20)(120,20)(120,25)
\qbezier(120,25)(120,25)(120,85)
\qbezier(120,85)(120,90)(125,90)
\qbezier(125,90)(125,90)(130,90)
\end{picture}}
\put(-90,0){\begin{picture}(0,100)
\thicklines
\qbezier(20,0)(20,0)(130,0)
\qbezier(20,100)(20,100)(85,100)
\qbezier(85,100)(90,100)(90,95)
\qbezier(90,95)(90,95)(90,65)
\qbezier(90,55)(90,55)(90,25)
\qbezier(90,15)(90,10)(95,10)
\qbezier(95,10)(95,10)(130,10)
\qbezier(20,60)(20,60)(20,85)
\qbezier(20,85)(20,90)(25,90)
\qbezier(25,90)(25,90)(65,90)
\qbezier(65,90)(70,90)(70,85)
\qbezier(70,85)(70,85)(70,65)
\qbezier(20,10)(20,10)(65,10)
\qbezier(65,10)(70,10)(70,15)
\qbezier(70,25)(70,25)(70,55)
\qbezier(130,100)(130,100)(115,100)
\qbezier(115,100)(110,100)(110,95)
\qbezier(110,95)(110,95)(110,85)
\qbezier(110,85)(110,80)(105,80)
\qbezier(105,80)(105,80)(95,80)
\qbezier(85,80)(85,80)(75,80)
\qbezier(65,80)(65,80)(45,80)
\qbezier(45,80)(40,80)(40,75)
\qbezier(40,75)(40,75)(40,60)
\qbezier(40,60)(40,60)(60,40)
\qbezier(60,40)(60,40)(65,40)
\qbezier(75,40)(75,40)(85,40)
\qbezier(95,40)(100,45)(100,45)
\qbezier(100,45)(100,45)(100,55)
\qbezier(100,55)(100,60)(95,60)
\qbezier(95,60)(95,60)(60,60)
\qbezier(60,60)(60,60)(54,54)
\qbezier(46,46)(46,46)(40,40)
\qbezier(40,40)(40,40)(40,25)
\qbezier(40,25)(40,20)(45,20)
\qbezier(45,20)(45,20)(115,20)
\qbezier(115,20)(120,20)(120,25)
\qbezier(120,25)(120,25)(120,85)
\qbezier(120,85)(120,90)(125,90)
\qbezier(125,90)(125,90)(130,90)
\put(150,100){\makebox(0,0)[cc]{$\cdots$}}
\put(150,90){\makebox(0,0)[cc]{$\cdots$}}
\put(150,10){\makebox(0,0)[cc]{$\cdots$}}
\put(150,0){\makebox(0,0)[cc]{$\cdots$}}
\end{picture}}
\put(60,0){\begin{picture}(0,100)
\thicklines
\qbezier(20,0)(20,0)(130,0)
\qbezier(20,100)(20,100)(85,100)
\qbezier(85,100)(90,100)(90,95)
\qbezier(90,95)(90,95)(90,65)
\qbezier(90,55)(90,55)(90,25)
\qbezier(90,15)(90,10)(95,10)
\qbezier(95,10)(95,10)(125,10)
\qbezier(20,90)(20,90)(65,90)
\qbezier(65,90)(70,90)(70,85)
\qbezier(70,85)(70,85)(70,65)
\qbezier(20,10)(20,10)(65,10)
\qbezier(65,10)(70,10)(70,15)
\qbezier(70,25)(70,25)(70,55)
\qbezier(130,100)(130,100)(115,100)
\qbezier(115,100)(110,100)(110,95)
\qbezier(110,95)(110,95)(110,85)
\qbezier(110,85)(110,80)(105,80)
\qbezier(105,80)(105,80)(95,80)
\qbezier(85,80)(85,80)(75,80)
\qbezier(65,80)(65,80)(45,80)
\qbezier(45,80)(40,80)(40,75)
\qbezier(40,75)(40,75)(40,60)
\qbezier(40,60)(40,60)(60,40)
\qbezier(60,40)(60,40)(65,40)
\qbezier(75,40)(75,40)(85,40)
\qbezier(95,40)(100,45)(100,45)
\qbezier(100,45)(100,45)(100,55)
\qbezier(100,55)(100,60)(95,60)
\qbezier(95,60)(95,60)(60,60)
\qbezier(60,60)(60,60)(54,54)
\qbezier(46,46)(46,46)(40,40)
\qbezier(40,40)(40,40)(40,25)
\qbezier(40,25)(40,20)(45,20)
\qbezier(45,20)(45,20)(115,20)
\qbezier(115,20)(120,20)(120,25)
\qbezier(120,25)(120,25)(120,85)
\qbezier(120,85)(120,90)(125,90)
\qbezier(125,90)(130,90)(130,85)
\qbezier(125,10)(130,10)(130,15)
\qbezier(130,15)(130,15)(130,85)
\qbezier(130,0)(140,0)(140,10)
\qbezier(140,10)(140,10)(140,90)
\qbezier(130,100)(140,100)(140,90)
\end{picture}}
\end{picture}
\caption{Knot $L'_{m}$.} \label{fig14} 
\end{figure}
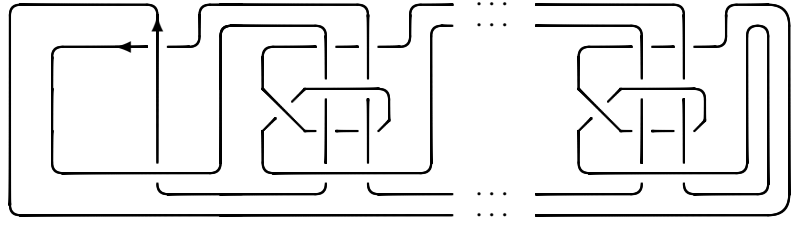

Two vectors in Fig.~\ref{fig14} indicate one of crossings of the diagram of $L'_{m}$. Let us consider the skien triple corresponding to this crossings. Diagram $L'_{m}$ is an element of the skein triple, corresponding to the positive crossings. Denote two other elements of the skein triple by $L'_{m_{-}}$ and $L'_{m_{0}}$. It easy to see than $L'_{m_{-}}$ is the unknot and $L'_{m_{0}}$ is the 2-component link presented in Fig.~\ref{fig15}. One its component is the unknot, and the another component is the same as the knot $L_{m-1}$, compare with the Fig.~\ref{fig12}. Its linking number is equal to $0$.
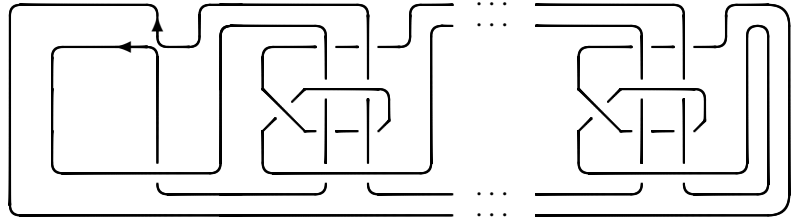
\begin{figure}
\unitlength=.28mm
\centering 
\begin{picture}(0,110)(0,0)
\put(-190,0){\begin{picture}(150,0)
\thicklines
\qbezier(25,0)(25,0)(130,0)
\qbezier(25,0)(20,0)(20,5)
\qbezier(20,5)(20,5)(20,95)
\qbezier(20,95)(20,100)(25,100)
\qbezier(25,100)(25,100)(85,100)
\qbezier(85,100)(90,100)(90,95)
\qbezier(90,90)(90,90)(90,85)
\qbezier(90,85)(90,80)(95,80)
\qbezier(85,80)(90,80)(90,75)
\qbezier(90,75)(90,75)(90,25)
\qbezier(90,15)(90,10)(95,10)
\qbezier(95,10)(95,10)(130,10)
\put(75,80){\vector(-1,0){5}}
\put(90,90){\vector(0,1){5}}
\qbezier(130,100)(130,100)(115,100)
\qbezier(115,100)(110,100)(110,95)
\qbezier(110,95)(110,95)(110,85)
\qbezier(110,85)(110,80)(105,80)
\qbezier(105,80)(105,80)(95,80)
\qbezier(85,80)(85,80)(45,80)
\qbezier(45,80)(40,80)(40,75)
\qbezier(40,75)(40,75)(40,60)
\qbezier(40,60)(40,60)(40,40)
\qbezier(40,40)(40,40)(40,25)
\qbezier(40,25)(40,20)(45,20)
\qbezier(45,20)(45,20)(115,20)
\qbezier(115,20)(120,20)(120,25)
\qbezier(120,25)(120,25)(120,85)
\qbezier(120,85)(120,90)(125,90)
\qbezier(125,90)(125,90)(130,90)
\end{picture}}
\put(-90,0){\begin{picture}(0,100)
\thicklines
\qbezier(20,0)(20,0)(130,0)
\qbezier(20,100)(20,100)(85,100)
\qbezier(85,100)(90,100)(90,95)
\qbezier(90,95)(90,95)(90,65)
\qbezier(90,55)(90,55)(90,25)
\qbezier(90,15)(90,10)(95,10)
\qbezier(95,10)(95,10)(130,10)
\qbezier(20,60)(20,60)(20,85)
\qbezier(20,85)(20,90)(25,90)
\qbezier(25,90)(25,90)(65,90)
\qbezier(65,90)(70,90)(70,85)
\qbezier(70,85)(70,85)(70,65)
\qbezier(20,10)(20,10)(65,10)
\qbezier(65,10)(70,10)(70,15)
\qbezier(70,25)(70,25)(70,55)
\qbezier(130,100)(130,100)(115,100)
\qbezier(115,100)(110,100)(110,95)
\qbezier(110,95)(110,95)(110,85)
\qbezier(110,85)(110,80)(105,80)
\qbezier(105,80)(105,80)(95,80)
\qbezier(85,80)(85,80)(75,80)
\qbezier(65,80)(65,80)(45,80)
\qbezier(45,80)(40,80)(40,75)
\qbezier(40,75)(40,75)(40,60)
\qbezier(40,60)(40,60)(60,40)
\qbezier(60,40)(60,40)(65,40)
\qbezier(75,40)(75,40)(85,40)
\qbezier(95,40)(100,45)(100,45)
\qbezier(100,45)(100,45)(100,55)
\qbezier(100,55)(100,60)(95,60)
\qbezier(95,60)(95,60)(60,60)
\qbezier(60,60)(60,60)(54,54)
\qbezier(46,46)(46,46)(40,40)
\qbezier(40,40)(40,40)(40,25)
\qbezier(40,25)(40,20)(45,20)
\qbezier(45,20)(45,20)(115,20)
\qbezier(115,20)(120,20)(120,25)
\qbezier(120,25)(120,25)(120,85)
\qbezier(120,85)(120,90)(125,90)
\qbezier(125,90)(125,90)(130,90)
\put(150,100){\makebox(0,0)[cc]{$\cdots$}}
\put(150,90){\makebox(0,0)[cc]{$\cdots$}}
\put(150,10){\makebox(0,0)[cc]{$\cdots$}}
\put(150,0){\makebox(0,0)[cc]{$\cdots$}}
\end{picture}}
\put(60,0){\begin{picture}(0,100)
\thicklines
\qbezier(20,0)(20,0)(130,0)
\qbezier(20,100)(20,100)(85,100)
\qbezier(85,100)(90,100)(90,95)
\qbezier(90,95)(90,95)(90,65)
\qbezier(90,55)(90,55)(90,25)
\qbezier(90,15)(90,10)(95,10)
\qbezier(95,10)(95,10)(125,10)
\qbezier(20,90)(20,90)(65,90)
\qbezier(65,90)(70,90)(70,85)
\qbezier(70,85)(70,85)(70,65)
\qbezier(20,10)(20,10)(65,10)
\qbezier(65,10)(70,10)(70,15)
\qbezier(70,25)(70,25)(70,55)
\qbezier(130,100)(130,100)(115,100)
\qbezier(115,100)(110,100)(110,95)
\qbezier(110,95)(110,95)(110,85)
\qbezier(110,85)(110,80)(105,80)
\qbezier(105,80)(105,80)(95,80)
\qbezier(85,80)(85,80)(75,80)
\qbezier(65,80)(65,80)(45,80)
\qbezier(45,80)(40,80)(40,75)
\qbezier(40,75)(40,75)(40,60)
\qbezier(40,60)(40,60)(60,40)
\qbezier(60,40)(60,40)(65,40)
\qbezier(75,40)(75,40)(85,40)
\qbezier(95,40)(100,45)(100,45)
\qbezier(100,45)(100,45)(100,55)
\qbezier(100,55)(100,60)(95,60)
\qbezier(95,60)(95,60)(60,60)
\qbezier(60,60)(60,60)(54,54)
\qbezier(46,46)(46,46)(40,40)
\qbezier(40,40)(40,40)(40,25)
\qbezier(40,25)(40,20)(45,20)
\qbezier(45,20)(45,20)(115,20)
\qbezier(115,20)(120,20)(120,25)
\qbezier(120,25)(120,25)(120,85)
\qbezier(120,85)(120,90)(125,90)
\qbezier(125,90)(130,90)(130,85)
\qbezier(125,10)(130,10)(130,15)
\qbezier(130,15)(130,15)(130,85)
\qbezier(130,0)(140,0)(140,10)
\qbezier(140,10)(140,10)(140,90)
\qbezier(130,100)(140,100)(140,90)
\end{picture}}
\end{picture}
\caption{Link $L'_{m_{0}}$.} \label{fig15} 
\end{figure}

Since the linking number of $L'_{m_{0}}$ is equal to $0$, by rule ($3$) for $L'_m$ we have $c_0(L'_m)= \frac{1}{x} \left( c_0(L'_{m_{-}})+ c_0(L'_{m_{0}}) \right)$. Using that $L'_{m_{0}}$ has linking number $0$ so we get $c_0(L'_{m_{0}})= (x-1)c_0(L_{m-1})$.

Combining all the above steps involved to compute $c_0(K_m; x)$ we obtain 
$$
c_0(K_m;x)= \dfrac{1}{x} \left( 1+ (x-1) \, c_0(L_{m};x) \right),  
$$
where 
\begin{equation}
 \begin{aligned}
          c_0(L_m;x) &= x- x(x-1)\dfrac{1}{x} \left( 1+ (x-1)c_0(L_{m-1};x) \right) \\
                     &= x-(x-1)-(x-1)^2c_0(L_{m-1};x)\\
                     &= 1- (x-1)^2 c_0(L_{m-1};x),
 \end{aligned}
\end{equation}
From rule ($2$) for any skein triple we have  $x c_0(L_+;x) - c_0(L_-;x) = c_0(L_0;x)$ where, $L_0$ is a 2-component link, say  $L_0=L_0^1 \cup L_0^2$ having linking number $\lambda$. From rule ($3$) we get $c_0(L_0;x) = (x-1) x^{-\lambda} c_0(L_0^1;x) c_0(L_0^2;x)$ and putting $x=1$ in both equations gives $c_0(L_+;1) - c_0(L_-;1) = c_0(L_0;1)=0$, i.e., $ c_0(L_+;1) = c_0(L_-;1) $. Since the crossing chamge is an unknotting operation and $c_{0} (U)=1$, we obtain $c_0(L;1)= 1$ for every knot $L$. Therefore, $c_0(L_i;x)$ are non-zero polynomials for all $i=1, 2, \ldots, m$ and hence,
$$
  \operatorname{maxdeg} c_0(L_m;x) > \operatorname{maxdeg} c_0(L_{m-1};x),  
$$
 i.e. all the polynomials $c_0(K_m;x)$ have different degrees which concludes that knots $K_i$'s are distinct, so $\sigma_n$ forms an $n$-simplex in $\mathcal{G}_R$ that completes the proof of Theorem~\ref{thm1}.
\end{proof}

\begin{corollary}\label{corollary2-4}
There exists infinite collection of knots $K_0, K_1, \ldots$ such that $d_R(K_i,K_j)=1$ for $i\neq j$.
\end{corollary}

\begin{proof}
Consider the infinite family of knots $\{K_0, K_1, \ldots,K_m, \ldots\}$, where $K_m$ for $m \geq 1$ are the knots as seen in the proof of Theorem \ref{thm1}. Family satisfies the condition as required. 
\end{proof}

\begin{remark}
{\rm  As pointed out earlier, pass move and $\overline{\sharp}$-move are specific examples of r.c.c.. Therefore, infinite family of knots constructed in Gordian complexes $\mathcal{G}_{P}$ and $\mathcal{G}_{\overline{\sharp}}$ also serve as example in $\mathcal{G}_{R}$. However the family of knots $\{K_m\}_{m\geq0}$ given in the proof of Theorem~\ref{thm1} does not work in the case of pass-move and $\overline{\sharp}$-move as r.c.c at regions $R_i$ in $K_m$ is different from pass-move and $\overline{\sharp}$-move.} 
\end{remark}

While discussing the Gordian complex of knots by region crossing change we note that $d_R(K,K')$, the distance between any two knots $K, K'$, is always less than or equal to two. It will be interesting to find any pair of knots such that their distance is exactly equal to two. 

\begin{problem}  
Find pair of knots $K, K'$ such that $d_R(K,K')=2$? 
\end{problem}

For a given $0$-simplex $\sigma= \{K_0\}$ in a Gordian complex, the existence of $n$-simplex $\sigma_n = \{K_0, K_1,...,K_n\}$ 
containing $K_0$ for all $n \geq 1$ is well studied. However for a $m$-simplex $\sigma_m$ for $m\geq1$, the question of finding simplexes of higher dimension containing $\sigma_m$ remains unexplored. 

\begin{problem} 
Find whether for a given $m$-simplex $\sigma_m$ for $m\geq1$ there exist $n$-simplex $\sigma_n$ such that $\sigma_m \subset  \sigma_n$ for all $n > m$? 
\end{problem}

For $m=1$ the existence is shown in the Gordian complex by crossing change in~\cite{a}.

\section{Gordian complex of virtual knots}

L.~Kauffman~\cite{d} introduced virtual knots, and initiated the virtual knot theory. A virtual knot diagram is a 4-regular planar graph with each vertex replaced by two types of crossings, classical and virtual crossings(see Fig.~\ref{fig16}). 
\begin{figure}[!ht]
\centering 
\unitlength=0.6mm
\begin{picture}(0,30)(0,5)
\thicklines
\qbezier(-70,10)(-70,10)(-50,30)
\qbezier(-70,30)(-70,30)(-62,22) 
\qbezier(-50,10)(-50,10)(-58,18)
\qbezier(10,10)(10,10)(-10,30)
\qbezier(10,30)(10,30)(2,22) 
\qbezier(-10,10)(-10,10)(-2,18)
\qbezier(70,10)(70,10)(50,30)
\qbezier(70,30)(70,30)(50,10) 
\put(60,20){\circle{4}}
\end{picture}
\caption{Classical and virtual crossings.} \label{fig16}
\end{figure}
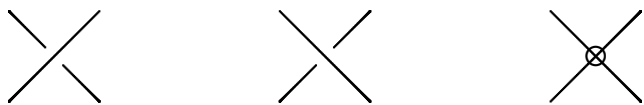
A virtual crossing is depicted by placing a small circle at the vertex. Two virtual knot diagrams correspond to same virtual knot if one can be obtained from the other via a sequence of generalized Reidemeister moves shown in Fig.~\ref{fig17}.
A virtual knot is defined as equivalence class of virtual knot diagrams modulo generalized Reidemeister moves.
\begin{figure}[th]
\centerline{\includegraphics[width=0.9\linewidth]{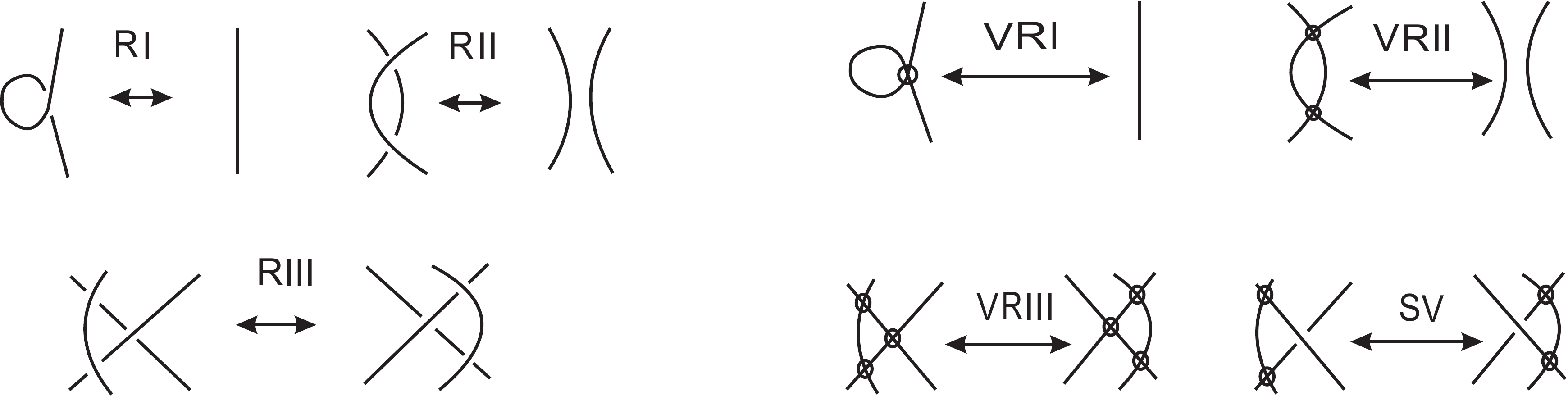}} 
\caption{Generalized Reidemeister moves.\label{fig17}}
\end{figure}

The authors of \cite{b} extended the concept of Gordian complex to virtual knots using a local move called $v$-move which changes a classical crossing into virtual crossing. Later, S.~Horiuchi and Y.~Ohyama~\cite{c} also defined Gordian complex of virtual knots using local moves known as forbidden moves. In both Gordian complexes of virtual knots~\cite{b,c}, it was proved that every $0$-simplex $\{K_0\}$ is always contained in a $n$-simplex for each positive integer~$n$.

\emph{Arc shift move} \cite{g} is another local move in virtual knot diagram which results in reversing orientation between two consecutive crossings locally. An arc in a virtual knot diagram $D$ is defined as the segment in the knot diagram passing through exactly one pair of crossings, the crossings can be either classical or virtual. Arc shift move on an arc (a,b) is applied by first cutting the arc at two points near the crossings and then joining the free ends by switching the sides as shown in Fig.~\ref{fig18} with one classical and one virtual crossings. 

\begin{figure}[th]
\centerline{\includegraphics[width=1.0\linewidth]{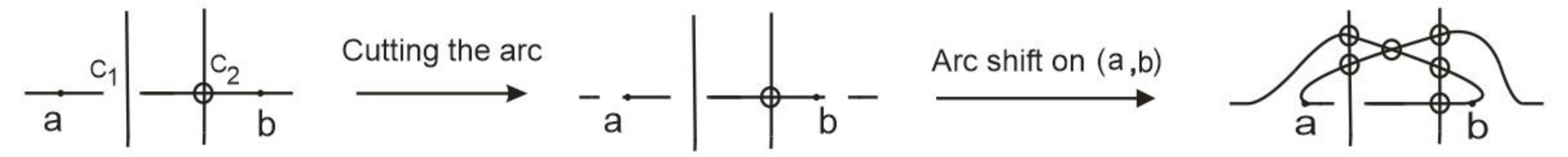}} 
\caption{Arc shift on arc (a,b).\label{fig18}}
\end{figure}

Arc shift move together with generalized Reidemeister moves is an unknotting operation  for virtual knots~\cite{g} and hence, any virtual knot $K$ can be transformed into any other virtual knot $K'$ using arc shift moves. For any two virtual knots $K_1$, $K_2$, the distance $d_{A}(K_1,K_2)$ between $K_1$ and $K_2$ is defined as minimum number of arc shift moves needed to convert a diagram of $K_1$ into a diagram of $K_2$. We define Gordian complex of virtual knots by arc shift move $\mathcal{G_A}$ as follows.
\begin{defi}
The Gordian complex $\mathcal{G_A}$ of virtual knots by arc shift move is a simplicial complex defined by the following;
\begin{enumerate}
\item Each virtual knot $K$ is a vertex of $\mathcal{G_A}$.
\item Any set $\{k_i\}_{i=0}^{n}$ of $n+1$ vertices spans an n-simplex if and only if $d_{A}(K_i,K_j)=1$ for $0 \leq i, j \leq n$, $i\neq j$.
\end{enumerate}
\end{defi}

\begin{example}
Virtual knots shown in Fig.~\ref{fig19} span a 2-simplex in $\mathcal{G_A}$.
\begin{figure}[th]
\centerline{\includegraphics[width=0.6\linewidth]{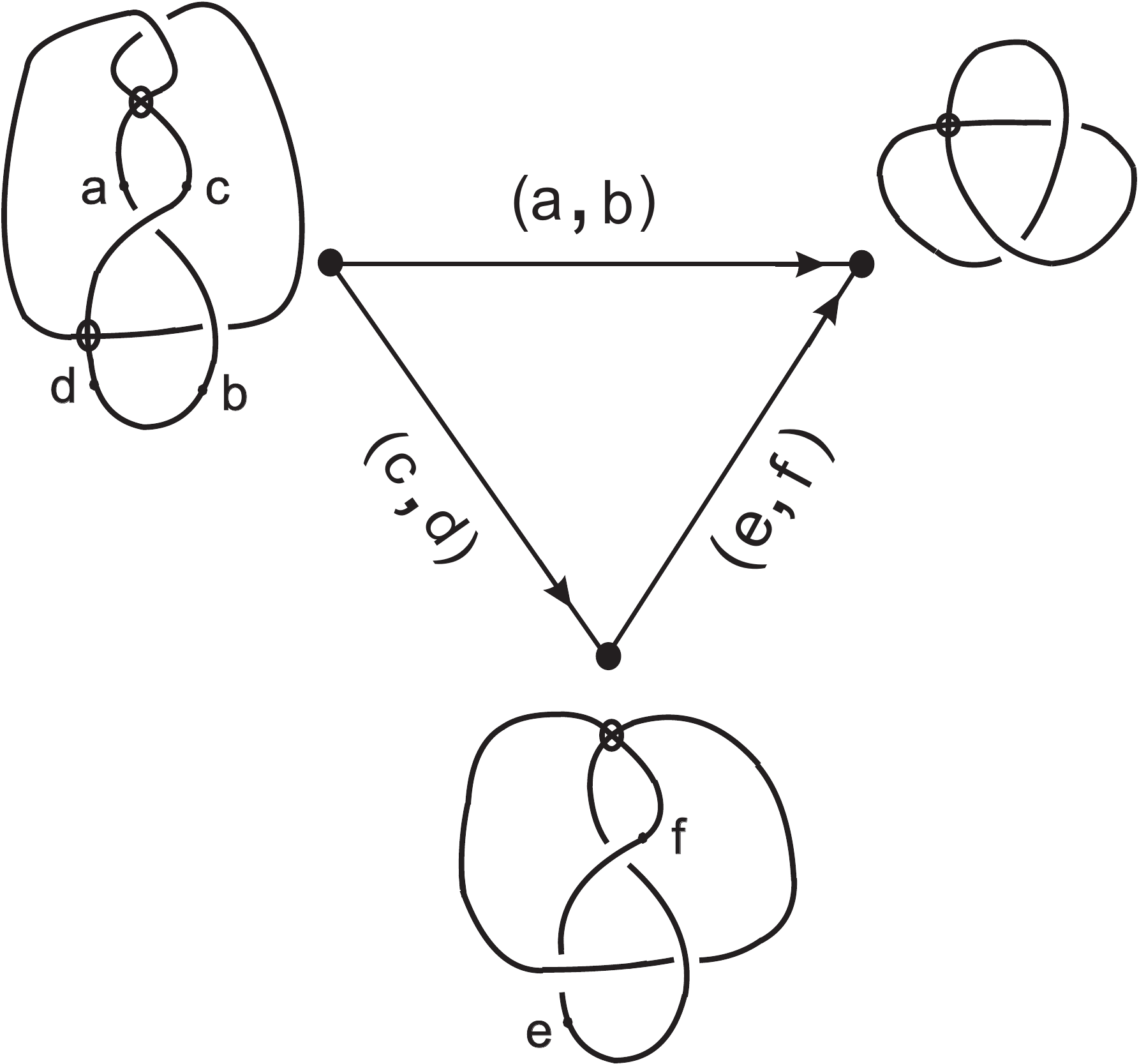}} 
\caption{2-simplex in $\mathcal{G_A}$.}    \label{fig19}
\end{figure}

\end{example}
In further exploring the simplexes in $\mathcal{G_A}$ we prove following results which gives us enough examples of simplexes in $\mathcal{G_A}$.

\begin{theorem}\label{1}
There exists an infinite family of virtual knots $\{VK_n\}_{n\geq 1}$ such that $d_{A}(VK_t,VK_s)=1$ for distinct $t,s\geq 1$.
\end{theorem}

As a corollary of Theorem ~\ref{1} we prove the following

\begin{corollary}\label{2}
Given any $0$-simplex $\{L\}$ in $\mathcal{G_A}$, there exists an $n$-simplex containing $L$ for each positive integer $n$.
\end{corollary}

\section{Proof of Theorem~\ref{1} and Corollary~\ref{2}}

\begin{proof}
M.~Hirasawa and Y.~Uchida \cite{a} gave a nice construction of a family of knots to show the existence of a $n$-simplex for each $n\in \mathbb{N}$ containing a given 1-simplex in $\mathcal{G}$. \emph{Horiuchi et. al.} \cite{b} modified this family by changing some of the classical crossings into virtual crossings to show the existence of an $n$-simplex for each $n\in \mathbb{N}$ in the Gordian complex of virtual knots by $v$-move. We add some extra classical and virtual crossings to the family obtained in \cite{b} and prove that the collection of virtual knots so obtained satisfies the required conditions to prove Theorem ~\ref{1}. 

Consider the family of virtual knots $\{VK_n\}_{n\geq 1}$, where each $VK_n$ is the virtual knot shown in Fig.~\ref{fig20}. Each $VK_n$ contains $n+1$ blocks separated by dashed lines and denoted as $c_0, c_1 \ldots, c_n$. Each block contains a  horizontal $2$-strand braid with  two virtual and two classical crossings alternately.  
\begin{figure}
\unitlength=.21mm
\centering 
\begin{picture}(0,90)(0,0)
\put(-300,0){\begin{picture}(0,0)
{\thinlines
\qbezier(90,-13)(90,-13)(90,-8)
\qbezier(90,-3)(90,-3)(90,2)
\qbezier(90,7)(90,7)(90,12)
\qbezier(90,17)(90,17)(90,22)
\qbezier(90,27)(90,27)(90,32)
\qbezier(90,37)(90,37)(90,42)
\qbezier(90,47)(90,47)(90,52)
\qbezier(90,57)(90,57)(90,62)
\qbezier(90,67)(90,67)(90,72)
\qbezier(90,77)(90,77)(90,82)
\qbezier(90,87)(90,87)(90,92)
}
{\thinlines
\qbezier(250,-13)(250,-13)(250,-8)
\qbezier(250,-3)(250,-3)(250,2)
\qbezier(250,7)(250,7)(250,12)
\qbezier(250,17)(250,17)(250,22)
\qbezier(250,27)(250,27)(250,32)
\qbezier(250,37)(250,37)(250,42)
\qbezier(250,47)(250,47)(250,52)
\qbezier(250,57)(250,57)(250,62)
\qbezier(250,67)(250,67)(250,72)
\qbezier(250,77)(250,77)(250,82)
\qbezier(250,87)(250,87)(250,92)
\put(45,90){\makebox(0,0)[cc]{$c_{0}$}}
\put(170,90){\makebox(0,0)[cc]{$c_{1}$}}
}
\thicklines
\qbezier(0,0)(0,0)(90,0)
\qbezier(0,0)(0,0)(0,30)
\qbezier(0,50)(0,50)(0,80)
\qbezier(0,80)(0,80)(90,80)
\qbezier(0,30)(0,30)(20,50)
\qbezier(0,50)(0,50)(6,44)
\qbezier(14,36)(20,30)(20,30)
\qbezier(20,30)(20,30)(40,50)
\qbezier(20,50)(20,50)(40,30)
\qbezier(40,30)(40,30)(60,50)
\qbezier(40,50)(40,50)(46,44)
\qbezier(54,36)(60,30)(60,30)
\qbezier(60,30)(60,30)(80,50)
\qbezier(60,50)(60,50)(80,30)
\qbezier(80,50)(80,50)(80,70)
\qbezier(80,70)(80,70)(90,70)
\qbezier(80,30)(80,30)(80,10)
\qbezier(80,10)(80,10)(90,10)
\put(30,40){\circle{6}}
\put(70,40){\circle{6}}
\qbezier(90,0)(90,0)(250,0)
\qbezier(90,10)(90,10)(120,10)
\qbezier(120,10)(120,10)(120,16)
\qbezier(120,24)(120,24)(120,26)
\qbezier(120,34)(120,34)(120,70)
\qbezier(120,70)(120,70)(90,70)
\qbezier(80,80)(80,80)(130,80)
\qbezier(130,80)(130,80)(130,34)
\qbezier(130,26)(130,26)(130,24)
\qbezier(130,16)(130,16)(130,10)
\qbezier(130,10)(130,10)(250,10)
\qbezier(100,20)(100,20)(100,60)
\qbezier(100,20)(100,20)(240,20)
\qbezier(240,20)(240,20)(240,70)
\qbezier(240,70)(240,70)(250,70)
\qbezier(100,20)(100,20)(100,60)
\qbezier(100,60)(100,60)(220,60)
\qbezier(110,30)(110,30)(110,50)
\qbezier(110,30)(110,30)(140,30)
\qbezier(110,50)(110,50)(140,50)
\put(120,50){\circle{6}}
\put(120,60){\circle{6}}
\put(130,50){\circle{6}}
\put(130,60){\circle{6}}
\qbezier(140,30)(140,30)(160,50)
\qbezier(140,50)(140,50)(160,30)
\qbezier(160,30)(160,30)(180,50)
\qbezier(160,50)(160,50)(166,44)
\qbezier(174,36)(180,30)(180,30)
\qbezier(180,30)(180,30)(200,50)
\qbezier(180,50)(180,50)(200,30)
\qbezier(200,30)(200,30)(220,50)
\qbezier(200,50)(200,50)(206,44)
\qbezier(214,36)(220,30)(220,30)
\put(150,40){\circle{6}}
\put(190,40){\circle{6}}
\qbezier(220,50)(220,50)(220,60)
\qbezier(220,30)(220,30)(230,30)
\qbezier(230,30)(230,30)(230,80)
\qbezier(230,80)(230,80)(250,80)
\end{picture}}
\put(-140,0){\begin{picture}(0,100)
\thicklines
{\thinlines
\qbezier(250,-13)(250,-13)(250,-8)
\qbezier(250,-3)(250,-3)(250,2)
\qbezier(250,7)(250,7)(250,12)
\qbezier(250,17)(250,17)(250,22)
\qbezier(250,27)(250,27)(250,32)
\qbezier(250,37)(250,37)(250,42)
\qbezier(250,47)(250,47)(250,52)
\qbezier(250,57)(250,57)(250,62)
\qbezier(250,67)(250,67)(250,72)
\qbezier(250,77)(250,77)(250,82)
\qbezier(250,87)(250,87)(250,92)
\put(170,90){\makebox(0,0)[cc]{$c_{2}$}}
}
\qbezier(90,0)(90,0)(250,0)
\qbezier(90,10)(90,10)(120,10)
\qbezier(120,10)(120,10)(120,16)
\qbezier(120,24)(120,24)(120,26)
\qbezier(120,34)(120,34)(120,70)
\qbezier(120,70)(120,70)(90,70)
\qbezier(80,80)(80,80)(130,80)
\qbezier(130,80)(130,80)(130,34)
\qbezier(130,26)(130,26)(130,24)
\qbezier(130,16)(130,16)(130,10)
\qbezier(130,10)(130,10)(250,10)
\qbezier(100,20)(100,20)(100,60)
\qbezier(100,20)(100,20)(240,20)
\qbezier(240,20)(240,20)(240,70)
\qbezier(240,70)(240,70)(250,70)
\qbezier(100,20)(100,20)(100,60)
\qbezier(100,60)(100,60)(220,60)
\qbezier(110,30)(110,30)(110,50)
\qbezier(110,30)(110,30)(140,30)
\qbezier(110,50)(110,50)(140,50)
\put(120,50){\circle{6}}
\put(120,60){\circle{6}}
\put(130,50){\circle{6}}
\put(130,60){\circle{6}}
\qbezier(140,30)(140,30)(160,50)
\qbezier(140,50)(140,50)(160,30)
\qbezier(160,30)(160,30)(180,50)
\qbezier(160,50)(160,50)(166,44)
\qbezier(174,36)(180,30)(180,30)
\qbezier(180,30)(180,30)(200,50)
\qbezier(180,50)(180,50)(200,30)
\qbezier(200,30)(200,30)(220,50)
\qbezier(200,50)(200,50)(206,44)
\qbezier(214,36)(220,30)(220,30)
\put(150,40){\circle{6}}
\put(190,40){\circle{6}}
\qbezier(220,50)(220,50)(220,60)
\qbezier(220,30)(220,30)(230,30)
\qbezier(230,30)(230,30)(230,80)
\qbezier(230,80)(230,80)(250,80)
\put(270,80){\makebox(0,0)[cc]{$\cdots$}}
\put(270,70){\makebox(0,0)[cc]{$\cdots$}}
\put(270,10){\makebox(0,0)[cc]{$\cdots$}}
\put(270,0){\makebox(0,0)[cc]{$\cdots$}}
\end{picture}}
\put(50,0){\begin{picture}(0,100)
{\thinlines
\qbezier(90,-13)(90,-13)(90,-8)
\qbezier(90,-3)(90,-3)(90,2)
\qbezier(90,7)(90,7)(90,12)
\qbezier(90,17)(90,17)(90,22)
\qbezier(90,27)(90,27)(90,32)
\qbezier(90,37)(90,37)(90,42)
\qbezier(90,47)(90,47)(90,52)
\qbezier(90,57)(90,57)(90,62)
\qbezier(90,67)(90,67)(90,72)
\qbezier(90,77)(90,77)(90,82)
\qbezier(90,87)(90,87)(90,92)
\put(170,90){\makebox(0,0)[cc]{$c_{n}$}}
}
\thicklines
\qbezier(90,0)(90,0)(250,0)
\qbezier(90,10)(90,10)(120,10)
\qbezier(120,10)(120,10)(120,16)
\qbezier(120,24)(120,24)(120,26)
\qbezier(120,34)(120,34)(120,70)
\qbezier(120,70)(120,70)(90,70)
\qbezier(90,80)(90,80)(130,80)
\qbezier(130,80)(130,80)(130,34)
\qbezier(130,26)(130,26)(130,24)
\qbezier(130,16)(130,16)(130,10)
\qbezier(130,10)(130,10)(250,10)
\qbezier(100,20)(100,20)(100,60)
\qbezier(100,20)(100,20)(240,20)
\qbezier(240,20)(240,20)(240,70)
\qbezier(240,70)(240,70)(250,70)
\qbezier(100,20)(100,20)(100,60)
\qbezier(100,60)(100,60)(220,60)
\qbezier(110,30)(110,30)(110,50)
\qbezier(110,30)(110,30)(140,30)
\qbezier(110,50)(110,50)(140,50)
\put(120,50){\circle{6}}
\put(120,60){\circle{6}}
\put(130,50){\circle{6}}
\put(130,60){\circle{6}}
\qbezier(140,30)(140,30)(160,50)
\qbezier(140,50)(140,50)(160,30)
\qbezier(160,30)(160,30)(180,50)
\qbezier(160,50)(160,50)(166,44)
\qbezier(174,36)(180,30)(180,30)
\qbezier(180,30)(180,30)(200,50)
\qbezier(180,50)(180,50)(200,30)
\qbezier(200,30)(200,30)(220,50)
\qbezier(200,50)(200,50)(206,44)
\qbezier(214,36)(220,30)(220,30)
\put(150,40){\circle{6}}
\put(190,40){\circle{6}}
\qbezier(220,50)(220,50)(220,60)
\qbezier(220,30)(220,30)(230,30)
\qbezier(230,30)(230,30)(230,80)
\qbezier(230,80)(230,80)(250,80)
\qbezier(250,10)(250,10)(250,70)
\qbezier(250,0)(250,0)(260,0)
\qbezier(250,80)(250,80)(260,80)
\qbezier(260,0)(260,0)(260,80)
\end{picture}}
\end{picture}
\caption{Virtual knot $VK_{n}$.} \label{fig20} 
\end{figure}
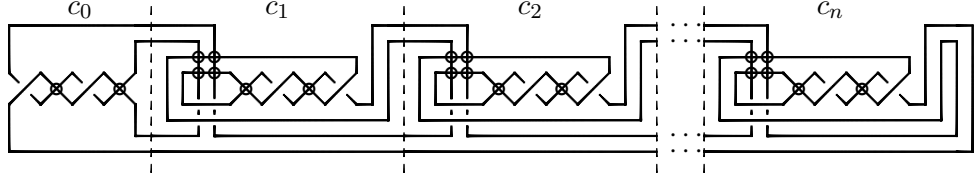

One such block of horizontal $2$-strands braid containing two virtual and two classical crossings in alternation is shown in Fig.~\ref{fig21}. On one of the strand in $2$-braid choose an arc $(a,b)$ passing through a classical and virtual crossing and apply an arc shift move on the arc $(a,b)$, as shown in the Fig.~\ref{fig21}. Resulting virtual knot diagram obtained by arc shift move can be simplified further using VRII and RII moves such that this local $2$-strand forming the braid becomes parallel with no crossings. This whole process of making $2$-strand in any block $c_i$ parallel and free from crossings need a single arc shift move and some generalized Reidemeister moves.
\begin{figure}[th]
\centerline{\includegraphics[width=0.8\linewidth]{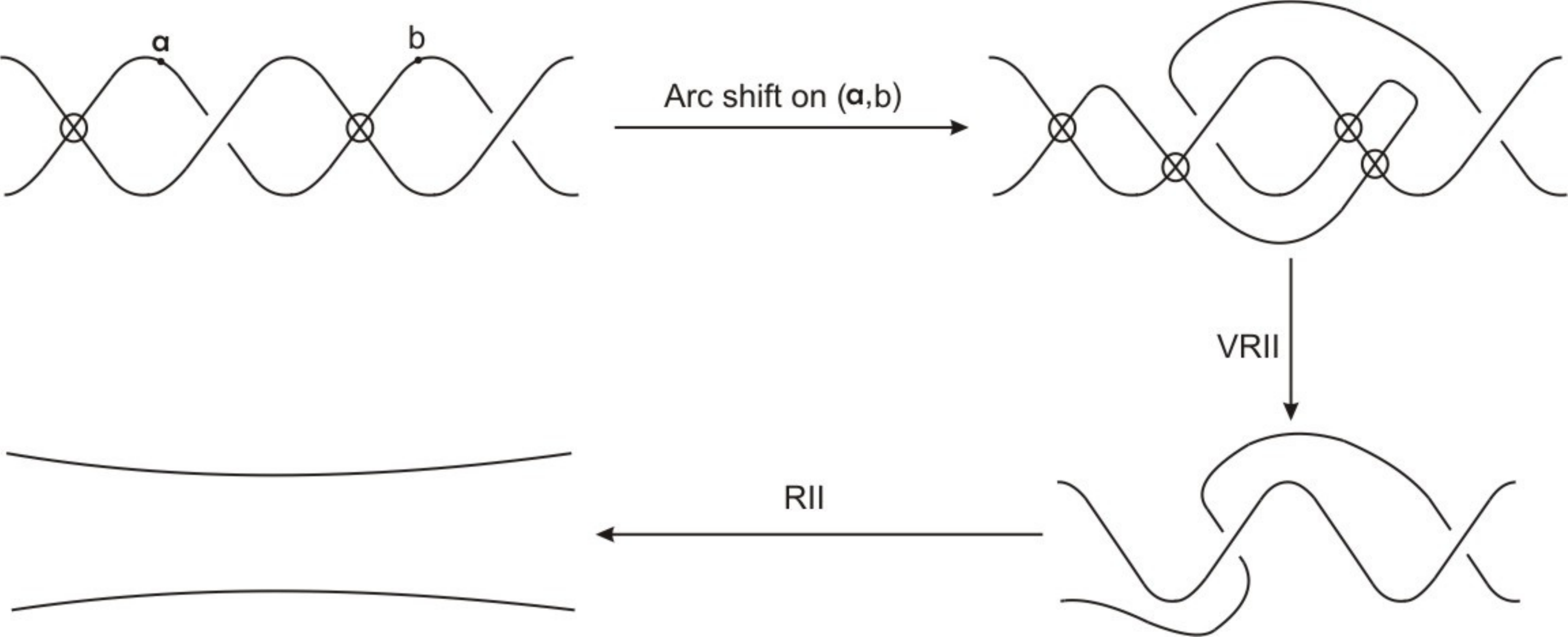}} 
\caption{Turning $2$-strand braid free of crossings.} \label{fig21}
\end{figure}

It is easy to observe from the diagram of $VK_n$(see Fig.~\ref{fig20}) that, if $2$-strand braid in the block $c_j$ becomes parallel, then crossings in the $VK_n$ to the right side of block $c_j$ will vanish by applying the moves RII and VRII for multiple times. Thus by applying a single arc shift move in the blocks $c_1, c_2, \ldots, c_n$ in $VK_n$, we can obtain the knots $VK_0, VK_1, VK_2, \ldots, VK_{n-1}$ respectively, where $VK_0$ is assumed as trivial knot. Thus $d_{A}(VK_t,VK_s) \leq 1$ for distinct $t,s \geq 0$. In order to show the equality, it is enough to ensure that all the $VK_n's$ are distinct virtual knots. We use the Kauffman $f$-polynomial to show that virtual knots $VK_n$ are distinct and follow the method of tangles to compute maximal degree of $f$-polynomial as given in \cite{b}.

For a reader convenience we recall the definition of Kauffman's $f$-polynomial. 
A state of a virtual knot diagram $D$ is a collection of loops obtained by splitting each classical crossing as either a $A$-split or a $B$-split as shown in Fig.~\ref{fig22}.
\begin{figure}[!ht]
\centering 
\unitlength=0.6mm
\begin{picture}(0,20)(0,10)
\thicklines
\qbezier(-70,10)(-60,20)(-70,30)
\qbezier(-50,10)(-60,20)(-50,30) 
\put(-20,20){\vector(-1,0){20}}
\qbezier(-10,10)(-10,10)(10,30)
\qbezier(-10,30)(-10,30)(-2,22) 
\qbezier(10,10)(10,10)(2,18)
\put(20,20){\vector(1,0){20}}
\put(-30,25){\makebox(0,0)[cc]{$A$-split}}
\put(30,25){\makebox(0,0)[cc]{$B$-split}}
\qbezier(50,30)(60,20)(70,30)
\qbezier(50,10)(60,20)(70,10) 
\end{picture}
\caption{$A$- and $B$-splits.} \label{fig22}
\end{figure}
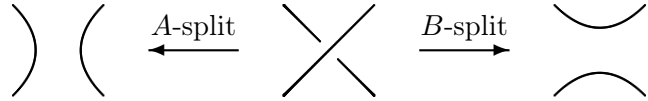

 Loops in a state might have intersections in virtual crossings. L. Kauffman \cite{d} defined bracket polynomial $\langle D \rangle$ for a diagram $D$ of virtual knot $K$ as follows:
 \[\langle D \rangle = \sum \limits_ { S } A ^ { a ( S ) - b ( S ) } \left( - A ^ { 2 } - A ^ { - 2 } \right) ^ { ||S|| - 1 },\]
 where sum runs over all the states $S$ of $D$, $a(S)$ and $b(S)$ denotes number of $A$-splits and  $B$-splits respectively in the state $S$. $||S||$ denotes number of closed loops in the state $S$.

Bracket polynomial $\langle D \rangle$ is invariant under all generalized Reidemeister moves except the RI move. $\langle D \rangle$ is modified into \emph{f-polynomial} $f_{D}(A)$ by normalizing it using writhe $w(D)$ of the diagram. $f_{D}(A)$ is given as follows:
\[f _ { D } ( A ) = \left( - A ^ { 3 } \right) ^ { - w ( D ) } \langle D \rangle.\]
Here, writhe $w(D)$ is the sum of signs of all the classical crossings of $D$. This \emph{f-polynomial} is an invariant of virtual knots and polynomial obtained by change of variable $A\rightarrow t^{-1/4}$ in \emph{f} is also known as Jones polynomial $V_{D}(t)$ of $D$.\\
A (2,2)-\emph{tangle} in a virtual knot diagram $D$ is local part of the knot lying inside a circle that intersects $D$ in exactly four points. By joining four boundary points of a (2,2)-tangle $T$ in all possible ways we get virtual links which are denoted by $D(T)$, $N(T)$ and $X(T)$ as shown in the Fig.~\ref{fig23}. We call them respectively $D$, $N$ and $X$-closures of tangle $T$.
\begin{figure}[!ht]
\centering 
\unitlength=0.6mm
\begin{picture}(0,40)(0,0)
\thicklines
\put(-60,20){\circle{20}}
\put(60,20){\circle{20}}
\qbezier(-70,15)(-75,15)(-75,20)
\qbezier(-70,25)(-75,25)(-75,20)
\qbezier(-50,15)(-45,15)(-45,20) 
\qbezier(-50,25)(-45,25)(-45,20)
\put(-60,20){\makebox(0,0)[cc]{$T$}}
\put(0,20){\circle{20}}
\qbezier(-5,11)(-5,5)(0,5)
\qbezier(5,11 )(5,5)(0,5) 
\qbezier(-5,29)(-5,35)(0,35)
\qbezier(5,29 )(5,35)(0,35)
\put(0,20){\makebox(0,0)[cc]{$T$}}
\put(-60,0){\makebox(0,0)[cc]{$D(T)$}}
\put(0,0){\makebox(0,0)[cc]{$N(T)$}}
\put(60,0){\makebox(0,0)[cc]{$X(T)$}}
\qbezier(55,29)(55,35)(70,35)
\qbezier(70,35)(85,35)(85,30)
\qbezier(85,30)(85,30)(85,25)
\qbezier(55,11)(55,5)(70,5) 
\qbezier(70,5)(85,5)(85,10)
\qbezier(85,10)(85,10)(85,15) 
\qbezier(65,29)(70,35)(75,25)
\qbezier(65,11)(70,5)(75,15) 
\put(60,20){\makebox(0,0)[cc]{$T$}}
\qbezier(75,25)(75,25)(85,15)
\qbezier(75,15)(75,15)(85,25)
\put(80,20){\circle{3}}
\end{picture}
\caption{Closures of tangle $T$.} \label{fig23}
\end{figure}
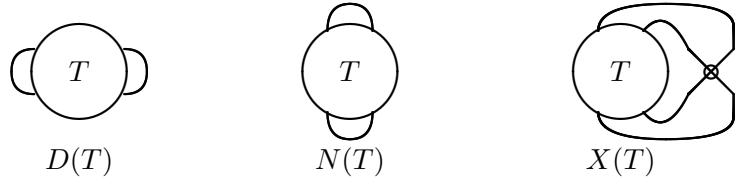

Sum of two (2,2)-tangles $T$ and $S$ is defined as the tangle shown in Fig.~\ref{fig24}.
\begin{figure}[!ht]
\centering 
\unitlength=0.6mm
\begin{picture}(0,20)(0,10)
\thicklines
\put(-20,20){\circle{20}}
\put(20,20){\circle{20}}
\put(-20,20){\makebox(0,0)[cc]{$T$}}
\put(20,20){\makebox(0,0)[cc]{$S$}}
\qbezier(-30,15)(-30,15)(-35,15)
\qbezier(-30,25)(-30,25)(-35,25)
\qbezier(-10,15)(-10,15)(10,15) 
\qbezier(-10,25)(-10,25)(10,25) 
\qbezier(30,25)(30,25)(35,25)
\qbezier(30,15)(30,15)(35,15)
\end{picture}
\caption{Sum $T+S$ of tangles $T$ and $S$.} \label{fig24}
\end{figure}
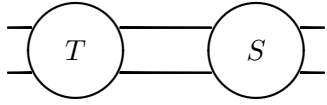
In~\cite{b} there is presented an alternative way to calculate bracket polynomial $\langle D \rangle$, where $D$ can be written as $D= N(T+S)$ for some tangles $T$ and $S$.

\begin{lemma} \cite{b} \label{3} The following relation holds: 
$$
\langle N ( T + S ) \rangle = ( \langle D ( T ) \rangle , \langle N ( T ) \rangle , \langle X ( T ) \rangle )  \cdot \mathcal { A } \cdot  ( \langle D ( S ) \rangle,  \langle N ( S ) \rangle, \langle X ( S ) \rangle)^{T}, 
$$ 
where $T$ denotes the transpose of vector and $\mathcal A$ is $3\times3$-matrix: 
$$\mathcal { A } = \left( \begin{array} { l l l } { \alpha } & { \beta } & { \beta } \\ { \beta } & { \alpha } & { \beta } \\ { \beta } & { \beta } & { \alpha } \end{array} \right),
$$
where 
$$
\alpha = \frac { \mu + 1 } { ( \mu - 1 ) ( \mu + 2 ) } , \quad \beta = \frac { - 1 } { ( \mu - 1 ) ( \mu + 2 ) } \quad \text { and } 
\quad \mu = - A ^ { 2 } - A ^ { - 2 }.
$$
\end{lemma}

Now coming back to the proof of Theorem~\ref{1}, observe that $VK_n$ can be written as $N(T+S_n)$, where $T$ and $S_n$ are the tangles shown in Fig.~\ref{fig25} and Fig.~\ref{fig26} respectively.
\begin{figure}
\unitlength=.24mm
\centering 
\begin{picture}(90,90)(0,0)
\thicklines
\qbezier(0,0)(0,0)(90,0)
\qbezier(0,0)(0,0)(0,30)
\qbezier(0,50)(0,50)(0,80)
\qbezier(0,80)(0,80)(90,80)
\qbezier(0,30)(0,30)(20,50)
\qbezier(0,50)(0,50)(6,44)
\qbezier(14,36)(20,30)(20,30)
\qbezier(20,30)(20,30)(40,50)
\qbezier(20,50)(20,50)(40,30)
\qbezier(40,30)(40,30)(60,50)
\qbezier(40,50)(40,50)(46,44)
\qbezier(54,36)(60,30)(60,30)
\qbezier(60,30)(60,30)(80,50)
\qbezier(60,50)(60,50)(80,30)
\qbezier(80,50)(80,50)(80,70)
\qbezier(80,70)(80,70)(90,70)
\qbezier(80,30)(80,30)(80,10)
\qbezier(80,10)(80,10)(90,10)
\put(30,40){\circle{6}}
\put(70,40){\circle{6}}
\end{picture}
\caption{Tangle $T$.} \label{fig25} 
\end{figure}
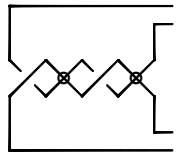

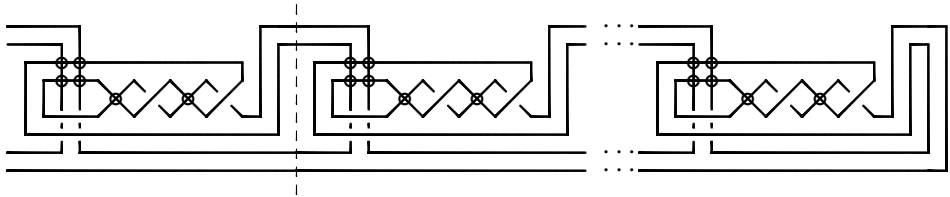
\begin{figure}
\unitlength=.24mm
\centering 
\begin{picture}(0,90)(0,0)
\put(-345,0){\begin{picture}(0,0)
\thicklines 
{\thinlines
\qbezier(250,-13)(250,-13)(250,-8)
\qbezier(250,-3)(250,-3)(250,2)
\qbezier(250,7)(250,7)(250,12)
\qbezier(250,17)(250,17)(250,22)
\qbezier(250,27)(250,27)(250,32)
\qbezier(250,37)(250,37)(250,42)
\qbezier(250,47)(250,47)(250,52)
\qbezier(250,57)(250,57)(250,62)
\qbezier(250,67)(250,67)(250,72)
\qbezier(250,77)(250,77)(250,82)
\qbezier(250,87)(250,87)(250,92)
}
\qbezier(90,0)(90,0)(250,0)
\qbezier(90,10)(90,10)(120,10)
\qbezier(120,10)(120,10)(120,16)
\qbezier(120,24)(120,24)(120,26)
\qbezier(120,34)(120,34)(120,70)
\qbezier(120,70)(120,70)(90,70)
\qbezier(90,80)(90,80)(130,80)
\qbezier(130,80)(130,80)(130,34)
\qbezier(130,26)(130,26)(130,24)
\qbezier(130,16)(130,16)(130,10)
\qbezier(130,10)(130,10)(250,10)
\qbezier(100,20)(100,20)(100,60)
\qbezier(100,20)(100,20)(240,20)
\qbezier(240,20)(240,20)(240,70)
\qbezier(240,70)(240,70)(250,70)
\qbezier(100,20)(100,20)(100,60)
\qbezier(100,60)(100,60)(220,60)
\qbezier(110,30)(110,30)(110,50)
\qbezier(110,30)(110,30)(140,30)
\qbezier(110,50)(110,50)(140,50)
\put(120,50){\circle{6}}
\put(120,60){\circle{6}}
\put(130,50){\circle{6}}
\put(130,60){\circle{6}}
\qbezier(140,30)(140,30)(160,50)
\qbezier(140,50)(140,50)(160,30)
\qbezier(160,30)(160,30)(180,50)
\qbezier(160,50)(160,50)(166,44)
\qbezier(174,36)(180,30)(180,30)
\qbezier(180,30)(180,30)(200,50)
\qbezier(180,50)(180,50)(200,30)
\qbezier(200,30)(200,30)(220,50)
\qbezier(200,50)(200,50)(206,44)
\qbezier(214,36)(220,30)(220,30)
\put(150,40){\circle{6}}
\put(190,40){\circle{6}}
\qbezier(220,50)(220,50)(220,60)
\qbezier(220,30)(220,30)(230,30)
\qbezier(230,30)(230,30)(230,80)
\qbezier(230,80)(230,80)(250,80)
\end{picture}}
\put(-185,0){\begin{picture}(0,100)
\thicklines
\qbezier(90,0)(90,0)(250,0)
\qbezier(90,10)(90,10)(120,10)
\qbezier(120,10)(120,10)(120,16)
\qbezier(120,24)(120,24)(120,26)
\qbezier(120,34)(120,34)(120,70)
\qbezier(120,70)(120,70)(90,70)
\qbezier(80,80)(80,80)(130,80)
\qbezier(130,80)(130,80)(130,34)
\qbezier(130,26)(130,26)(130,24)
\qbezier(130,16)(130,16)(130,10)
\qbezier(130,10)(130,10)(250,10)
\qbezier(100,20)(100,20)(100,60)
\qbezier(100,20)(100,20)(240,20)
\qbezier(240,20)(240,20)(240,70)
\qbezier(240,70)(240,70)(250,70)
\qbezier(100,20)(100,20)(100,60)
\qbezier(100,60)(100,60)(220,60)
\qbezier(110,30)(110,30)(110,50)
\qbezier(110,30)(110,30)(140,30)
\qbezier(110,50)(110,50)(140,50)
\put(120,50){\circle{6}}
\put(120,60){\circle{6}}
\put(130,50){\circle{6}}
\put(130,60){\circle{6}}
\qbezier(140,30)(140,30)(160,50)
\qbezier(140,50)(140,50)(160,30)
\qbezier(160,30)(160,30)(180,50)
\qbezier(160,50)(160,50)(166,44)
\qbezier(174,36)(180,30)(180,30)
\qbezier(180,30)(180,30)(200,50)
\qbezier(180,50)(180,50)(200,30)
\qbezier(200,30)(200,30)(220,50)
\qbezier(200,50)(200,50)(206,44)
\qbezier(214,36)(220,30)(220,30)
\put(150,40){\circle{6}}
\put(190,40){\circle{6}}
\qbezier(220,50)(220,50)(220,60)
\qbezier(220,30)(220,30)(230,30)
\qbezier(230,30)(230,30)(230,80)
\qbezier(230,80)(230,80)(250,80)
\put(270,80){\makebox(0,0)[cc]{$\cdots$}}
\put(270,70){\makebox(0,0)[cc]{$\cdots$}}
\put(270,10){\makebox(0,0)[cc]{$\cdots$}}
\put(270,0){\makebox(0,0)[cc]{$\cdots$}}
\end{picture}}
\put(5,0){\begin{picture}(0,100)
\thicklines
\qbezier(90,0)(90,0)(250,0)
\qbezier(90,10)(90,10)(120,10)
\qbezier(120,10)(120,10)(120,16)
\qbezier(120,24)(120,24)(120,26)
\qbezier(120,34)(120,34)(120,70)
\qbezier(120,70)(120,70)(90,70)
\qbezier(90,80)(90,80)(130,80)
\qbezier(130,80)(130,80)(130,34)
\qbezier(130,26)(130,26)(130,24)
\qbezier(130,16)(130,16)(130,10)
\qbezier(130,10)(130,10)(250,10)
\qbezier(100,20)(100,20)(100,60)
\qbezier(100,20)(100,20)(240,20)
\qbezier(240,20)(240,20)(240,70)
\qbezier(240,70)(240,70)(250,70)
\qbezier(100,20)(100,20)(100,60)
\qbezier(100,60)(100,60)(220,60)
\qbezier(110,30)(110,30)(110,50)
\qbezier(110,30)(110,30)(140,30)
\qbezier(110,50)(110,50)(140,50)
\put(120,50){\circle{6}}
\put(120,60){\circle{6}}
\put(130,50){\circle{6}}
\put(130,60){\circle{6}} 
\qbezier(140,30)(140,30)(160,50)
\qbezier(140,50)(140,50)(160,30)
\qbezier(160,30)(160,30)(180,50)
\qbezier(160,50)(160,50)(166,44)
\qbezier(174,36)(180,30)(180,30)
\qbezier(180,30)(180,30)(200,50)
\qbezier(180,50)(180,50)(200,30)
\qbezier(200,30)(200,30)(220,50)
\qbezier(200,50)(200,50)(206,44)
\qbezier(214,36)(220,30)(220,30)
\put(150,40){\circle{6}}
\put(190,40){\circle{6}}
\qbezier(220,50)(220,50)(220,60)
\qbezier(220,30)(220,30)(230,30)
\qbezier(230,30)(230,30)(230,80)
\qbezier(230,80)(230,80)(250,80)
\qbezier(250,10)(250,10)(250,70)
\qbezier(250,0)(250,0)(260,0)
\qbezier(250,80)(250,80)(260,80)
\qbezier(260,0)(260,0)(260,80)
\end{picture}}
\end{picture}
\caption{Tangle $S_{n}$.} \label{fig26} 
\end{figure}

From Lemma~\ref{3}, we have
\begin{eqnarray}\label{10} 
\langle VK_n \rangle & = & \langle N ( T + S_n ) \rangle   \cr 
&= & ( \langle D ( T ) \rangle , \langle N ( T ) \rangle , \langle X ( T ) \rangle ) \cdot \mathcal { A } \cdot ( \langle D ( S_n ) \rangle, \langle N ( S_n ) \rangle,  \langle X ( S_n ) \rangle)^{T},
\end{eqnarray}
where $\mathcal{A}$ is the matrix given in Lemma~\ref{3}. Further, $D(S_n)$ can be seen as $N$-closure of sum of two tangles by breaking $S_n$ at the dashed line shown in the diagram (Fig.~\ref{fig26}). Thus we have 
\begin{equation}\label{6}
\langle D ( S _ { n })\rangle = ( \langle D_{1} \rangle , \langle D_{2}\rangle, \langle D_{3}\rangle ) \cdot \mathcal { A } \cdot (  \langle D ( S_{n-1})\rangle, \langle N (S_{n-1}) \rangle, \langle X(S_{n-1}) \rangle )^{T},
\end{equation}
where $D_1, D_2$ and $ D_3$ are the virtual links shown in Fig.~\ref{fig27}.
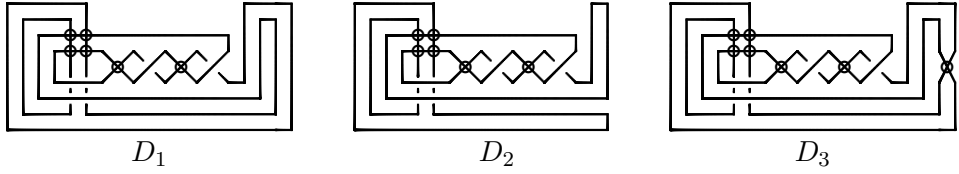
\begin{figure}
\unitlength=.21mm
\centering 
\begin{picture}(0,100)(0,-10)
\put(-220,0){\begin{picture}(140,0)(170,-10)
\thicklines 
\put(170,-15){\makebox(0,0)[cc]{$D_{1}$}}
\qbezier(80,0)(80,0)(80,80)
\qbezier(90,10)(90,10)(90,70)
\qbezier(80,0)(80,0)(250,0)
\qbezier(90,10)(90,10)(120,10)
\qbezier(120,10)(120,10)(120,16)
\qbezier(120,24)(120,24)(120,26)
\qbezier(120,34)(120,34)(120,70)
\qbezier(120,70)(120,70)(90,70)
\qbezier(80,80)(80,80)(130,80)
\qbezier(130,80)(130,80)(130,34)
\qbezier(130,26)(130,26)(130,24)
\qbezier(130,16)(130,16)(130,10)
\qbezier(130,10)(130,10)(250,10)
\qbezier(100,20)(100,20)(100,60)
\qbezier(100,20)(100,20)(240,20)
\qbezier(240,20)(240,20)(240,70)
\qbezier(240,70)(240,70)(250,70)
\qbezier(100,20)(100,20)(100,60)
\qbezier(100,60)(100,60)(220,60)
\qbezier(110,30)(110,30)(110,50)
\qbezier(110,30)(110,30)(140,30)
\qbezier(110,50)(110,50)(140,50)
\put(120,50){\circle{6}}
\put(120,60){\circle{6}}
\put(130,50){\circle{6}}
\put(130,60){\circle{6}}
\qbezier(140,30)(140,30)(160,50)
\qbezier(140,50)(140,50)(160,30)
\qbezier(160,30)(160,30)(180,50)
\qbezier(160,50)(160,50)(166,44)
\qbezier(174,36)(180,30)(180,30)
\qbezier(180,30)(180,30)(200,50)
\qbezier(180,50)(180,50)(200,30)
\qbezier(200,30)(200,30)(220,50)
\qbezier(200,50)(200,50)(206,44)
\qbezier(214,36)(220,30)(220,30)
\put(150,40){\circle{6}}
\put(190,40){\circle{6}}
\qbezier(220,50)(220,50)(220,60)
\qbezier(220,30)(220,30)(230,30)
\qbezier(230,30)(230,30)(230,80)
\qbezier(230,80)(230,80)(250,80)
\qbezier(250,10)(250,10)(250,70)
\qbezier(250,0)(250,0)(260,0)
\qbezier(250,80)(250,80)(260,80)
\qbezier(260,0)(260,0)(260,80)
\end{picture}}
\put(0,0){\begin{picture}(140,0)(170,-10)
\put(170,-15){\makebox(0,0)[cc]{$D_{2}$}}
\thicklines 
\qbezier(80,0)(80,0)(80,80)
\qbezier(90,10)(90,10)(90,70)
\qbezier(80,0)(80,0)(240,0)
\qbezier(240,0)(240,0)(240,10)
\qbezier(90,10)(90,10)(120,10)
\qbezier(120,10)(120,10)(120,16)
\qbezier(120,24)(120,24)(120,26)
\qbezier(120,34)(120,34)(120,70)
\qbezier(120,70)(120,70)(90,70)
\qbezier(80,80)(80,80)(130,80)
\qbezier(130,80)(130,80)(130,34)
\qbezier(130,26)(130,26)(130,24)
\qbezier(130,16)(130,16)(130,10)
\qbezier(130,10)(130,10)(240,10)
\qbezier(100,20)(100,20)(100,60)
\qbezier(100,20)(100,20)(240,20)
\qbezier(240,20)(240,20)(240,80)
\qbezier(100,20)(100,20)(100,60)
\qbezier(100,60)(100,60)(220,60)
\qbezier(110,30)(110,30)(110,50)
\qbezier(110,30)(110,30)(140,30)
\qbezier(110,50)(110,50)(140,50)
\put(120,50){\circle{6}}
\put(120,60){\circle{6}}
\put(130,50){\circle{6}}
\put(130,60){\circle{6}} 
\qbezier(140,30)(140,30)(160,50)
\qbezier(140,50)(140,50)(160,30)
\qbezier(160,30)(160,30)(180,50)
\qbezier(160,50)(160,50)(166,44)
\qbezier(174,36)(180,30)(180,30)
\qbezier(180,30)(180,30)(200,50)
\qbezier(180,50)(180,50)(200,30)
\qbezier(200,30)(200,30)(220,50)
\qbezier(200,50)(200,50)(206,44)
\qbezier(214,36)(220,30)(220,30)
\put(150,40){\circle{6}}
\put(190,40){\circle{6}}
\qbezier(220,50)(220,50)(220,60)
\qbezier(220,30)(220,30)(230,30)
\qbezier(230,30)(230,30)(230,80)
\qbezier(230,80)(230,80)(240,80)
\end{picture}}
\put(200,0){\begin{picture}(140,0)(170,-10)
\put(170,-15){\makebox(0,0)[cc]{$D_{3}$}}
\thicklines 
\qbezier(80,0)(80,0)(80,80)
\qbezier(90,10)(90,10)(90,70)
\qbezier(80,0)(80,0)(250,0)
\qbezier(90,10)(90,10)(120,10)
\qbezier(120,10)(120,10)(120,16)
\qbezier(120,24)(120,24)(120,26)
\qbezier(120,34)(120,34)(120,70)
\qbezier(120,70)(120,70)(90,70)
\qbezier(80,80)(80,80)(130,80)
\qbezier(130,80)(130,80)(130,34)
\qbezier(130,26)(130,26)(130,24)
\qbezier(130,16)(130,16)(130,10)
\qbezier(130,10)(130,10)(250,10)
\qbezier(100,20)(100,20)(100,60)
\qbezier(100,20)(100,20)(240,20)
\qbezier(240,20)(240,20)(240,70)
\qbezier(240,70)(240,70)(250,70)
\qbezier(100,20)(100,20)(100,60)
\qbezier(100,60)(100,60)(220,60)
\qbezier(110,30)(110,30)(110,50)
\qbezier(110,30)(110,30)(140,30)
\qbezier(110,50)(110,50)(140,50)
\put(120,50){\circle{6}}
\put(120,60){\circle{6}}
\put(130,50){\circle{6}}
\put(130,60){\circle{6}}
\qbezier(140,30)(140,30)(160,50)
\qbezier(140,50)(140,50)(160,30)
\qbezier(160,30)(160,30)(180,50)
\qbezier(160,50)(160,50)(166,44)
\qbezier(174,36)(180,30)(180,30)
\qbezier(180,30)(180,30)(200,50)
\qbezier(180,50)(180,50)(200,30)
\qbezier(200,30)(200,30)(220,50)
\qbezier(200,50)(200,50)(206,44)
\qbezier(214,36)(220,30)(220,30)
\put(150,40){\circle{6}}
\put(190,40){\circle{6}}
\qbezier(220,50)(220,50)(220,60)
\qbezier(220,30)(220,30)(230,30)
\qbezier(230,30)(230,30)(230,80)
\qbezier(230,80)(230,80)(250,80)
\qbezier(250,10)(250,10)(250,30)
\qbezier(250,50)(250,50)(250,70)
\qbezier(250,0)(250,0)(260,0)
\qbezier(250,80)(250,80)(260,80)
\qbezier(260,80)(260,80)(260,50)
\qbezier(260,0)(260,0)(260,30)
\qbezier(250,30)(250,30)(260,50)
\qbezier(250,50)(250,50)(260,30)
\put(255,40){\circle{6}}
\end{picture}}
\end{picture}
\caption{Virtual links $D_{1}$, $D_{2}$ and $D_{3}$.} \label{fig27} 
\end{figure}
$N(S_n)$ can be similarly seen as $N$-closure of sum of two tangles and we have
\begin{equation}\label{7}
\langle N( S_{n}) \rangle = (\langle D_{4} \rangle,\langle D_{5} \rangle, \langle D_{6} \rangle ) \cdot \mathcal { A } \cdot (\langle D ( S _{n-1}) \rangle, \langle N (S_{n-1}) \rangle, \langle X (S_{n-1}) \rangle)^{T},
\end{equation}
where $ D_4$, $D_5$ and $ D_6$ are the virtual links shown in the Fig.~\ref{fig28}.
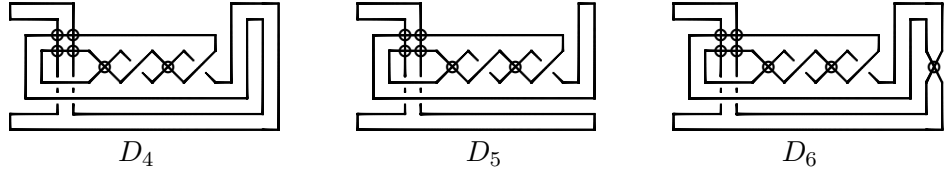
\begin{figure}
\unitlength=.21mm
\centering 
\begin{picture}(0,100)(0,-10)
\put(-220,0){\begin{picture}(140,0)(170,-10)
\thicklines 
\put(170,-15){\makebox(0,0)[cc]{$D_{4}$}}
\qbezier(90,0)(90,0)(90, 10)
\qbezier(90,70)(90,70)(90,80)
\qbezier(90,0)(90,0)(250,0)
\qbezier(90,10)(90,10)(120,10)
\qbezier(120,10)(120,10)(120,16)
\qbezier(120,24)(120,24)(120,26)
\qbezier(120,34)(120,34)(120,70)
\qbezier(120,70)(120,70)(90,70)
\qbezier(90,80)(90,80)(130,80)
\qbezier(130,80)(130,80)(130,34)
\qbezier(130,26)(130,26)(130,24)
\qbezier(130,16)(130,16)(130,10)
\qbezier(130,10)(130,10)(250,10)
\qbezier(100,20)(100,20)(100,60)
\qbezier(100,20)(100,20)(240,20)
\qbezier(240,20)(240,20)(240,70)
\qbezier(240,70)(240,70)(250,70)
\qbezier(100,20)(100,20)(100,60)
\qbezier(100,60)(100,60)(220,60)
\qbezier(110,30)(110,30)(110,50)
\qbezier(110,30)(110,30)(140,30)
\qbezier(110,50)(110,50)(140,50)
\put(120,50){\circle{6}}
\put(120,60){\circle{6}}
\put(130,50){\circle{6}}
\put(130,60){\circle{6}}
\qbezier(140,30)(140,30)(160,50)
\qbezier(140,50)(140,50)(160,30)
\qbezier(160,30)(160,30)(180,50)
\qbezier(160,50)(160,50)(166,44)
\qbezier(174,36)(180,30)(180,30)
\qbezier(180,30)(180,30)(200,50)
\qbezier(180,50)(180,50)(200,30)
\qbezier(200,30)(200,30)(220,50)
\qbezier(200,50)(200,50)(206,44)
\qbezier(214,36)(220,30)(220,30)
\put(150,40){\circle{6}}
\put(190,40){\circle{6}}
\qbezier(220,50)(220,50)(220,60)
\qbezier(220,30)(220,30)(230,30)
\qbezier(230,30)(230,30)(230,80)
\qbezier(230,80)(230,80)(250,80)
\qbezier(250,10)(250,10)(250,70)
\qbezier(250,0)(250,0)(260,0)
\qbezier(250,80)(250,80)(260,80)
\qbezier(260,0)(260,0)(260,80)
\end{picture}}
\put(0,0){\begin{picture}(140,0)(170,-10)
\put(170,-15){\makebox(0,0)[cc]{$D_{5}$}}
\thicklines 
\qbezier(90,0)(90,0)(90, 10)
\qbezier(90,70)(90,70)(90,80)
\qbezier(90,0)(90,0)(240,0)
\qbezier(240,0)(240,0)(240,10)
\qbezier(90,10)(90,10)(120,10)
\qbezier(120,10)(120,10)(120,16)
\qbezier(120,24)(120,24)(120,26)
\qbezier(120,34)(120,34)(120,70)
\qbezier(120,70)(120,70)(90,70)
\qbezier(90,80)(90,80)(130,80)
\qbezier(130,80)(130,80)(130,34)
\qbezier(130,26)(130,26)(130,24)
\qbezier(130,16)(130,16)(130,10)
\qbezier(130,10)(130,10)(240,10)
\qbezier(100,20)(100,20)(100,60)
\qbezier(100,20)(100,20)(240,20)
\qbezier(240,20)(240,20)(240,80)
\qbezier(100,20)(100,20)(100,60)
\qbezier(100,60)(100,60)(220,60)
\qbezier(110,30)(110,30)(110,50)
\qbezier(110,30)(110,30)(140,30)
\qbezier(110,50)(110,50)(140,50)
\put(120,50){\circle{6}}
\put(120,60){\circle{6}}
\put(130,50){\circle{6}}
\put(130,60){\circle{6}}
\qbezier(140,30)(140,30)(160,50)
\qbezier(140,50)(140,50)(160,30)
\qbezier(160,30)(160,30)(180,50)
\qbezier(160,50)(160,50)(166,44)
\qbezier(174,36)(180,30)(180,30)
\qbezier(180,30)(180,30)(200,50)
\qbezier(180,50)(180,50)(200,30)
\qbezier(200,30)(200,30)(220,50)
\qbezier(200,50)(200,50)(206,44)
\qbezier(214,36)(220,30)(220,30)
\put(150,40){\circle{6}}
\put(190,40){\circle{6}}
\qbezier(220,50)(220,50)(220,60)
\qbezier(220,30)(220,30)(230,30)
\qbezier(230,30)(230,30)(230,80)
\qbezier(230,80)(230,80)(240,80)
\end{picture}}
\put(200,0){\begin{picture}(140,0)(170,-10)
\put(170,-15){\makebox(0,0)[cc]{$D_{6}$}}
\thicklines 
\qbezier(90,0)(90,0)(90, 10)
\qbezier(90,70)(90,70)(90,80)
\qbezier(90,0)(90,0)(250,0)
\qbezier(90,10)(90,10)(120,10)
\qbezier(120,10)(120,10)(120,16)
\qbezier(120,24)(120,24)(120,26)
\qbezier(120,34)(120,34)(120,70)
\qbezier(120,70)(120,70)(90,70)
\qbezier(90,80)(90,80)(130,80)
\qbezier(130,80)(130,80)(130,34)
\qbezier(130,26)(130,26)(130,24)
\qbezier(130,16)(130,16)(130,10)
\qbezier(130,10)(130,10)(250,10)
\qbezier(100,20)(100,20)(100,60)
\qbezier(100,20)(100,20)(240,20)
\qbezier(240,20)(240,20)(240,70)
\qbezier(240,70)(240,70)(250,70)
\qbezier(100,20)(100,20)(100,60)
\qbezier(100,60)(100,60)(220,60)
\qbezier(110,30)(110,30)(110,50)
\qbezier(110,30)(110,30)(140,30)
\qbezier(110,50)(110,50)(140,50)
\put(120,50){\circle{6}}
\put(120,60){\circle{6}}
\put(130,50){\circle{6}}
\put(130,60){\circle{6}}
\qbezier(140,30)(140,30)(160,50)
\qbezier(140,50)(140,50)(160,30)
\qbezier(160,30)(160,30)(180,50)
\qbezier(160,50)(160,50)(166,44)
\qbezier(174,36)(180,30)(180,30)
\qbezier(180,30)(180,30)(200,50)
\qbezier(180,50)(180,50)(200,30)
\qbezier(200,30)(200,30)(220,50)
\qbezier(200,50)(200,50)(206,44)
\qbezier(214,36)(220,30)(220,30)
\put(150,40){\circle{6}}
\put(190,40){\circle{6}}
\qbezier(220,50)(220,50)(220,60)
\qbezier(220,30)(220,30)(230,30)
\qbezier(230,30)(230,30)(230,80)
\qbezier(230,80)(230,80)(250,80)
\qbezier(250,10)(250,10)(250,30)
\qbezier(250,50)(250,50)(250,70)
\qbezier(250,0)(250,0)(260,0)
\qbezier(250,80)(250,80)(260,80)
\qbezier(260,80)(260,80)(260,50)
\qbezier(260,0)(260,0)(260,30)
\qbezier(250,30)(250,30)(260,50)
\qbezier(250,50)(250,50)(260,30)
\put(255,40){\circle{6}}
\end{picture}}
\end{picture}
\caption{Virtual links $D_{4}$, $D_{5}$ and $D_{6}$.} \label{fig28} 
\end{figure}

Similarly for $X(S_n)$ we have
\begin{equation}\label{8}
\langle X (S_{n}) \rangle = ( \langle D_{7} \rangle, \langle D_{8} \rangle, \langle D_{9} \rangle) \cdot \mathcal {A} \cdot (\langle D (S_{n-1}) \rangle, \langle N(S_{n-1})\rangle, \langle X(S_{n-1}) \rangle)^{T},
\end{equation}
where $ D_7, D_8$ and $D_9$ are the virtual links shown in the Fig.~\ref{fig29}.
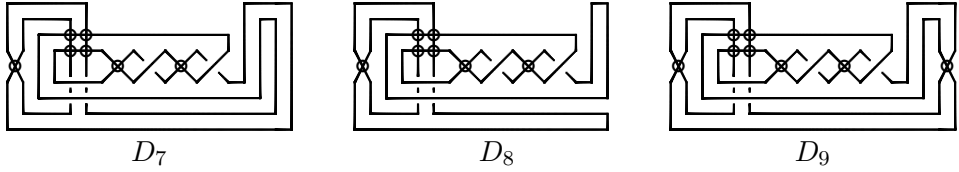
\begin{figure}
\unitlength=.21mm
\centering 
\begin{picture}(0,100)(0,-10)
\put(-220,0){\begin{picture}(140,0)(170,-10)
\thicklines 
\put(170,-15){\makebox(0,0)[cc]{$D_{7}$}}
\qbezier(80,0)(80,0)(80,30)
\qbezier(80,50)(80,50)(80,80)
\qbezier(90,10)(90,10)(90,30)
\qbezier(90,50)(90,50)(90,70)
\qbezier(80,30)(80,30)(90,50)
\qbezier(80,50)(80,50)(90,30)
\put(85,40){\circle{6}}
\qbezier(80,0)(80,0)(250,0)
\qbezier(90,10)(90,10)(120,10)
\qbezier(120,10)(120,10)(120,16)
\qbezier(120,24)(120,24)(120,26)
\qbezier(120,34)(120,34)(120,70)
\qbezier(120,70)(120,70)(90,70)
\qbezier(80,80)(80,80)(130,80)
\qbezier(130,80)(130,80)(130,34)
\qbezier(130,26)(130,26)(130,24)
\qbezier(130,16)(130,16)(130,10)
\qbezier(130,10)(130,10)(250,10)
\qbezier(100,20)(100,20)(100,60)
\qbezier(100,20)(100,20)(240,20)
\qbezier(240,20)(240,20)(240,70)
\qbezier(240,70)(240,70)(250,70)
\qbezier(100,20)(100,20)(100,60)
\qbezier(100,60)(100,60)(220,60)
\qbezier(110,30)(110,30)(110,50)
\qbezier(110,30)(110,30)(140,30)
\qbezier(110,50)(110,50)(140,50)
\put(120,50){\circle{6}}
\put(120,60){\circle{6}}
\put(130,50){\circle{6}}
\put(130,60){\circle{6}}
\qbezier(140,30)(140,30)(160,50)
\qbezier(140,50)(140,50)(160,30)
\qbezier(160,30)(160,30)(180,50)
\qbezier(160,50)(160,50)(166,44)
\qbezier(174,36)(180,30)(180,30)
\qbezier(180,30)(180,30)(200,50)
\qbezier(180,50)(180,50)(200,30)
\qbezier(200,30)(200,30)(220,50)
\qbezier(200,50)(200,50)(206,44)
\qbezier(214,36)(220,30)(220,30)
\put(150,40){\circle{6}}
\put(190,40){\circle{6}}
\qbezier(220,50)(220,50)(220,60)
\qbezier(220,30)(220,30)(230,30)
\qbezier(230,30)(230,30)(230,80)
\qbezier(230,80)(230,80)(250,80)
\qbezier(250,10)(250,10)(250,70)
\qbezier(250,0)(250,0)(260,0)
\qbezier(250,80)(250,80)(260,80)
\qbezier(260,0)(260,0)(260,80)
\end{picture}}
\put(0,0){\begin{picture}(140,0)(170,-10)
\put(170,-15){\makebox(0,0)[cc]{$D_{8}$}}
\thicklines 
\qbezier(80,0)(80,0)(80,30)
\qbezier(80,50)(80,50)(80,80)
\qbezier(90,10)(90,10)(90,30)
\qbezier(90,50)(90,50)(90,70)
\qbezier(80,30)(80,30)(90,50)
\qbezier(80,50)(80,50)(90,30)
\put(85,40){\circle{6}}
\qbezier(80,0)(80,0)(240,0)
\qbezier(240,0)(240,0)(240,10)
\qbezier(90,10)(90,10)(120,10)
\qbezier(120,10)(120,10)(120,16)
\qbezier(120,24)(120,24)(120,26)
\qbezier(120,34)(120,34)(120,70)
\qbezier(120,70)(120,70)(90,70)
\qbezier(80,80)(80,80)(130,80)
\qbezier(130,80)(130,80)(130,34)
\qbezier(130,26)(130,26)(130,24)
\qbezier(130,16)(130,16)(130,10)
\qbezier(130,10)(130,10)(240,10)
\qbezier(100,20)(100,20)(100,60)
\qbezier(100,20)(100,20)(240,20)
\qbezier(240,20)(240,20)(240,80)
\qbezier(100,20)(100,20)(100,60)
\qbezier(100,60)(100,60)(220,60)
\qbezier(110,30)(110,30)(110,50)
\qbezier(110,30)(110,30)(140,30)
\qbezier(110,50)(110,50)(140,50)
\put(120,50){\circle{6}}
\put(120,60){\circle{6}}
\put(130,50){\circle{6}}
\put(130,60){\circle{6}}
\qbezier(140,30)(140,30)(160,50)
\qbezier(140,50)(140,50)(160,30)
\qbezier(160,30)(160,30)(180,50)
\qbezier(160,50)(160,50)(166,44)
\qbezier(174,36)(180,30)(180,30)
\qbezier(180,30)(180,30)(200,50)
\qbezier(180,50)(180,50)(200,30)
\qbezier(200,30)(200,30)(220,50)
\qbezier(200,50)(200,50)(206,44)
\qbezier(214,36)(220,30)(220,30)
\put(150,40){\circle{6}}
\put(190,40){\circle{6}}
\qbezier(220,50)(220,50)(220,60)
\qbezier(220,30)(220,30)(230,30)
\qbezier(230,30)(230,30)(230,80)
\qbezier(230,80)(230,80)(240,80)
\end{picture}}
\put(200,0){\begin{picture}(140,0)(170,-10)
\put(170,-15){\makebox(0,0)[cc]{$D_{9}$}}
\thicklines 
\qbezier(80,0)(80,0)(80,30)
\qbezier(80,50)(80,50)(80,80)
\qbezier(90,10)(90,10)(90,30)
\qbezier(90,50)(90,50)(90,70)
\qbezier(80,30)(80,30)(90,50)
\qbezier(80,50)(80,50)(90,30)
\put(85,40){\circle{6}}
\qbezier(80,0)(80,0)(250,0)
\qbezier(90,10)(90,10)(120,10)
\qbezier(120,10)(120,10)(120,16)
\qbezier(120,24)(120,24)(120,26)
\qbezier(120,34)(120,34)(120,70)
\qbezier(120,70)(120,70)(90,70)
\qbezier(80,80)(80,80)(130,80)
\qbezier(130,80)(130,80)(130,34)
\qbezier(130,26)(130,26)(130,24)
\qbezier(130,16)(130,16)(130,10)
\qbezier(130,10)(130,10)(250,10)
\qbezier(100,20)(100,20)(100,60)
\qbezier(100,20)(100,20)(240,20)
\qbezier(240,20)(240,20)(240,70)
\qbezier(240,70)(240,70)(250,70)
\qbezier(100,20)(100,20)(100,60)
\qbezier(100,60)(100,60)(220,60)
\qbezier(110,30)(110,30)(110,50)
\qbezier(110,30)(110,30)(140,30)
\qbezier(110,50)(110,50)(140,50)
\put(120,50){\circle{6}}
\put(120,60){\circle{6}}
\put(130,50){\circle{6}}
\put(130,60){\circle{6}}
\qbezier(140,30)(140,30)(160,50)
\qbezier(140,50)(140,50)(160,30)
\qbezier(160,30)(160,30)(180,50)
\qbezier(160,50)(160,50)(166,44)
\qbezier(174,36)(180,30)(180,30)
\qbezier(180,30)(180,30)(200,50)
\qbezier(180,50)(180,50)(200,30)
\qbezier(200,30)(200,30)(220,50)
\qbezier(200,50)(200,50)(206,44)
\qbezier(214,36)(220,30)(220,30)
\put(150,40){\circle{6}}
\put(190,40){\circle{6}}
\qbezier(220,50)(220,50)(220,60)
\qbezier(220,30)(220,30)(230,30)
\qbezier(230,30)(230,30)(230,80)
\qbezier(230,80)(230,80)(250,80)
\qbezier(250,10)(250,10)(250,30)
\qbezier(250,50)(250,50)(250,70)
\qbezier(250,0)(250,0)(260,0)
\qbezier(250,80)(250,80)(260,80)
\qbezier(260,80)(260,80)(260,50)
\qbezier(260,0)(260,0)(260,30)
\qbezier(250,30)(250,30)(260,50)
\qbezier(250,50)(250,50)(260,30)
\put(255,40){\circle{6}}
\end{picture}}
\end{picture}
\caption{Virtual links $D_{7}$, $D_{8}$ and $D_{9}$.} \label{fig29} 
\end{figure}

Using equations~(\ref{6}),~(\ref{7}) and~(\ref{8}), we can write
\begin{equation}\label{9}
  \begin{aligned}
        \left( \begin{array} { c } { \left\langle D \left( S _ { n } \right) \right\rangle } \\ { \left\langle N \left( S _ { n } \right) \right\rangle } \\ { \left\langle X \left( S _ { n } \right) \right\rangle } \end{array} \right)  & = \left( \begin{array} { c c c } { \left\langle D _ { 1 } \right\rangle } & { \left\langle D _ { 2 } \right\rangle } & { \left\langle D _ { 3 } \right\rangle } \\ { \left\langle D _ { 4 } \right\rangle } & { \left\langle D _ { 5 } \right\rangle } & { \left\langle D _ { 6 } \right\rangle } \\ { \left\langle D _ { 7 } \right\rangle } & { \left\langle D _ { 8 } \right\rangle } & { \left\langle D _ { 9 } \right\rangle } \end{array} \right) \mathcal { A } \left( \begin{array} { c } { \left\langle D \left( S _ { n - 1 } \right) \right\rangle } \\ { \left\langle N \left( S _ { n - 1 } \right) \right\rangle } \\ { \left\langle X \left( S _ { n - 1 } \right) \right\rangle } \end{array} \right) \\
               & = \mathcal { B } \mathcal { A } \left( \begin{array} { c } { \left\langle D \left( S _ { n - 1 } \right) \right\rangle } \\ { \left\langle N \left( S _ { n - 1 } \right) \right\rangle } \\ { \left\langle X \left( S _ { n - 1 } \right) \right\rangle } \end{array} \right),
 \end{aligned}
\end{equation}
where $\mathcal{B}$ is the $3\times3$-matrix,

\[\mathcal { B } = \left( \begin{array} { c c c } { \left\langle D _ { 1 } \right\rangle } & { \left\langle D _ { 2 } \right\rangle } & { \left\langle D _ { 3 } \right\rangle } \\ { \left\langle D _ { 4 } \right\rangle } & { \left\langle D _ { 5 } \right\rangle } & { \left\langle D _ { 6 } \right\rangle } \\ { \left\langle D _ { 7 } \right\rangle } & { \left\langle D _ { 8 } \right\rangle } & { \left\langle D _ { 9 } \right\rangle } \end{array} \right).\]

We compute the values of $\langle D_j\rangle$ for all $j=1, 2, \ldots, 9$ and we have  
$$
\mathcal { B } = \left( \begin{array} {ccc} 
b_{11} & b_{12} & b_{13} \cr 
b_{21} & b_{22} & b_{23} \cr 
b_{31} & b_{22} & b_{33}  
\end{array} \right), 
$$
where
\begin{eqnarray*} 
b_{11} & = & -A^{12}+1-A^{4}-A^{- 4}, \cr 
b_{12} & = & A^{10}+3A^6+A^2-A^{-2}, \cr 
b_{13} & = & A^{10}-A^6-3A^4-A^2+1+A^{-2},  \cr 
b_{21} & = & A^6, \cr 
b_{22} & = &  -A^{8}-A^4, \cr 
b_{23} & = & A^6, \cr  
b_{31} & = & A^{10}+A^8 -2A^4-A^2+1+A^{-2}, \cr 
b_{32 }& = &  -A^8-A^4-1, \cr 
b_{33} & = & -A^4+A^2+2+A^{-2}-A^{-4}-A^{- 6}.
\end{eqnarray*} 
Computing the values of $\langle D(S_1) \rangle, \langle N(S_1) \rangle$ and $\langle X(S_1) \rangle$ we get
$$
\left( \begin{array} { c } { \left\langle D \left( S _ { 1 } \right) \right\rangle } \\ { \left\langle N \left( S _ { 1 } \right) \right\rangle } \\ { \left\langle X \left( S _ { 1 } \right) \right\rangle } \end{array} \right) = \left( \begin{array} { c } { - A ^ { 12 } + 1 - A ^ { 4 } - A ^ { - 4 } } \\ {  A ^ { 6 } } \\ { A ^ { 10} + A^ {8}-2 A^{4}-A^2 + 1+ A^{-2} } \end{array} \right) 
$$
and
$$
( \langle D ( T ) \rangle , \langle N ( T ) \rangle , \langle X ( T ) \rangle ) = ( A^6 ,-A^4 - A^{-4} ,-A^4 +1+ A^{-2} ).
$$
Also 
$$
\mathcal A  = \frac{1}{A^4 - A^2 -A^{-2} + A^{-4}} \left( \begin{array} {ccc} 
a_{11} & a_{12} & a_{13} \\ 
a_{21} & a_{22} & a_{23} \\ 
a_{31} & a_{32} & a_{33}
 \end{array} \right). 
$$
where 
$$
\begin{gathered} 
a_{11}  =  a_{22} = a_{33} = - A^{2}+1-A^{-2},  \cr 
a_{12} = a_{13} = a_{21} = a_{23} = a_{31} = a_{32} = -1.
\end{gathered} 
$$
Denote 
$$
\kappa = \frac{1}{A^4 - A^2 -A^{-2} + A^{-4}}. 
$$
By computations, we get
$$
\mathcal {BA} = \kappa \left( \begin{array} {ccc} f_1&f_2&f_3 \\ 
0&f_4&0 \\f_5 &f_6&f_7\end{array} \right),
$$
where
\begin{eqnarray*}
f_1&=&A^ {14}-A^{12}-A^{10}-A^6+2A^4-A^{-4}+A^{-6}, \cr  
f_2 & = & 2(-2A^{8} + 2 A^6 + A ^ { 2 } -1 - A ^ { - 2 } +A^{-4}), \cr 
f_3 &=&  - A ^ { 6 } + A^{-2}, \cr  
f_4 &=& A^{10}- A^{8}- A^4 + A^2, \cr 
f_5 &=& -A^{12} + A ^ { 8 } + A ^ { 6 } + A ^ { 4}-A ^ {2}- A^{-2}+A^{-6}, \cr  
f_6 &=& -2A^8 +2A^6 +2A^4+2A^2-4-A^{-2}+A^{-4}+A^{-6}, \cr 
f_7 &=& -A ^ { 10} + A^6 + A^4 +A^2 - A^{-2} -A^{-4} +A^{-8}  
\end{eqnarray*} 
are Laurent polynomials in the ring $\mathbb{Z} [A,A^{-1}]$. Further we get,
$$ 
\begin{gathered} 
\langle D ( T ) \rangle , \langle N ( T ) \rangle , \langle X ( T ) \rangle) \mathcal{A} \cr = \kappa ( -A^8 +A^6+A^4 -1 -A^{-2} +A^{-4} , A^2 -1 - A^{-4} +A^{-6} , 0 ).
\end{gathered} 
$$
Now for a given matrix $H=(h_{ij})$ where entries $h_{ij} \in \mathbb{Q} [A,A^{-1}]$, construct a new matrix $H_{max}= (\overline{h_{ij}^{*}})$ with its $(i,j)^{th}$ entry defined as,
$$
\begin{array} { c }  {h_{ij}^{*} = \left\{ \begin{array} { l l } \text{max. deg.} f - \text{max. deg.} g, & { \text { if $h_{ij}$}= \frac {f}{g} \hspace{.2cm} \text{for some } f,g \in \mathbb{Z} [A,A^{-1}] } \\ 0,  & { \text { if } h_{ij}= 0 } \end{array} \right. }\end{array}
$$
Whenever $h_{ij}= 0$ in matrix $H$, we put corresponding  $\overline{h_{ij}^{*}}=0$ in the matrix $H_{max}$. We have,
\begin{equation}\label{11}
( \mathcal { B A } ) _ { \mathrm { max } } = \left( \begin{array} { c c c } { \overline { 10 } } & { \overline { 4 } } & {\overline{2}} \\ { 0 } & { \overline { 6} } & { 0 } \\ {\overline{8}} & { \overline { 4 } } & { \overline{6} } \end{array} \right),
\end{equation}
\begin{equation}\label{12}
 (( \langle D ( T ) \rangle , \langle N ( T ) \rangle , \langle X ( T ) \rangle ) \mathcal{A}) _ { \max } = ( \overline { 4 } , \overline { -2 } ,  { 0 } ) 
\end{equation}
and 
\begin{equation}\label{5}
\left( \begin{array} { c } { \left\langle D \left( S _ { 1 } \right) \right\rangle } \\ { \left\langle N \left( S _ { 1 } \right) \right\rangle } \\ { \left\langle X \left( S _ { 1 } \right) \right\rangle } \end{array} \right) _ { \max } = \left( \begin{array} { c } {\overline { 12 }  } \\ { \overline { 6 } } \\ { \overline { 10 } } \end{array} \right).
\end{equation}

\begin{proposition}\label{4}
For each $n \in \mathbb{N}$ we have, 
\[\left( \begin{array} { c } { \left\langle D \left( S _ { n } \right) \right\rangle } \\ { \left\langle N \left( S _ { n } \right) \right\rangle } \\ { \left\langle X \left( S _ { n } \right) \right\rangle } \end{array} \right) _ { \max } = \left( \begin{array} { c } {\overline { 10n + 2 }  } \\ { \overline { 6n } } \\ { \overline { 10n } } \end{array} \right).\]
\end{proposition}

\begin{proof}
To prove this result we follow induction on $n$. For $n=1$, result follows from equation~(\ref{5}), assume the result is true for $n=k$, i.e.,
\[\left( \begin{array} { c } { \left\langle D \left( S _ { k } \right) \right\rangle } \\ { \left\langle N \left( S _ { k } \right) \right\rangle } \\ { \left\langle X \left( S _ { k } \right) \right\rangle } \end{array} \right) _ { \max } = \left( \begin{array} { c } {\overline { 10k + 2 }  } \\ { \overline { 6k } } \\ { \overline { 10k } } \end{array} \right).\]
For $n={k+1}$ from equation~(\ref{9}), we have 
$$\left( \begin{array} {c} { \left\langle D \left( S _ { k + 1 } \right) \right\rangle } \\ { \left\langle N \left( S _ { k + 1 } \right) \right\rangle } \\ { \left\langle X \left( S _ { k + 1 } \right) \right\rangle } \end{array} \right) = \mathcal { B } \mathcal { A } \left( \begin{array} { c } { \left\langle D \left( S _ { k } \right) \right\rangle } \\ { \left\langle N \left( S _ { k } \right) \right\rangle } \\ { \left\langle X \left( S _ { k } \right) \right\rangle } \end{array} \right).$$
From equation~(\ref{11}),
$$
( \mathcal { B A } ) _ { \mathrm { max } } = \left( \begin{array} { c c c } { \overline { 10 } } & { \overline { 4 } } & {\overline{2}} \\ { 0 } & { \overline { 6} } & { 0 } \\ {\overline{8}} & { \overline { 4 } } & { \overline{6} } \end{array} \right),$$
and from induction hypothesis for the case $n=k$, we have
$$\left( \begin{array} { l } { \left\langle D \left( S _ { k + 1 } \right) \right\rangle } \\ { \left\langle N \left( S _ { k + 1 } \right) \right\rangle } \\ { \left\langle X \left( S _ { k + 1 } \right) \right\rangle } \end{array} \right) _ { \max } = \left( \begin{array} { c } \overline{10+10k+ 2}\\ { \overline { 6+6k } } \\ { \overline { 10+10k } } \end{array} \right)=\left( \begin{array} { c } \overline{10(k+1)+ 2}\\ { \overline { 6(k+1) } } \\ { \overline {10(k+1) } } \end{array} \right).$$
Thus, the result holds true for $n={k+1}$ and hence the proof is complete by induction.
\end{proof} 

Now from equation~(\refeq{10}), we have
$$
\langle VK_n \rangle = (( \langle D ( T ) \rangle , \langle N ( T ) \rangle , \langle X ( T ) \rangle ) \cdot \mathcal { A } ) \cdot   \left( \begin{array} { c } { \langle D ( S_n ) \rangle } \\ { \langle N ( S_n ) \rangle } \\ { \langle X ( S_n ) \rangle } \end{array} \right).$$
Thus, using equation~(\ref{12}) and proposition~(\ref{4}), we have
\begin{equation}
 \begin{aligned}
 (\langle VK_n \rangle)_{max} & = ( ( \langle D ( T ) \rangle , \langle N ( T ) \rangle , \langle X ( T ) \rangle ) \mathcal { A })_{max} \left( \begin{array} { c } { \langle D ( S_n ) \rangle } \\ { \langle N ( S_n ) \rangle } \\ { \langle X ( S_n ) \rangle } \end{array} \right)_{max}\\
 & = ( \overline { 4 } , \overline { -2 } ,  { 0 } ) \left( \begin{array} { c } {\overline { 10n + 2 }  } \\ { \overline { 6n } } \\ { \overline { 10n } } \end{array} \right).
 \end{aligned}
\end{equation}
Therefore, maximum degree of $\langle VK_n \rangle$ is $10n+6$. Also as $f_{VK_n}(A) = \left( - A ^ { 3 } \right) ^ { - w ( VK_n ) } \langle VK_n \rangle$ and writhe $w(VK_n)= 2n+2$, so highest power of $f _ { VK_n } ( A )$ is $10n+6-3(2n+2)$, i.e., $4n$ for each $n \geq 1$. Hence, all virtual knots $VK_n$ for $n \geq 0$ are distinct and thus follows the Theorem.\hfill $\square$
\end{proof}

\smallskip

\noindent \textbf{Proof of Corollary ~\ref{2}. }
For any given virtual knot $L$, consider the family of virtual knots $\{L\sharp VK_0, L\sharp VK_1, \ldots, L\sharp VK_n\}$ obtained by taking connected sum of $L$ with each virtual knot $VK_i$ for $i=0,1, \ldots, n$ constructed in Theorem~\ref{1}. It can be observed that since $VK_0$ is trivial, so $L\sharp VK_0$ is equivalent to $L$ and hence the result follows.\hfill $\square$

\section{Affine index polynomial for the virtual knots $VK_n$}

Let $D$ be an oriented virtual knot diagram and $C(D)$ denotes the set of all classical crossings in $D$. By an \emph{arc} we mean an edge between two consecutive classical crossings along the orientation. Note that the notion of \emph{arc} here is slightly different from the way \emph{arc} was defined while discussing \emph{arc shift move}. The sign of classical crossing $c \in C(D)$, denoted by $\operatorname{sgn}(c)$, is defined as in Fig.~\ref{fig30}. Now assign an integer value to each arc in $D$ in such a way that the labeling around each crossing point of $D$ follows the rule as shown in Fig.~\ref{fig31}.
\begin{figure}[!ht]
\centering 
\unitlength=0.6mm
\begin{picture}(0,30)
\thicklines
\qbezier(-40,10)(-40,10)(-20,30)
\qbezier(-40,30)(-40,30)(-32,22) 
\qbezier(-20,10)(-20,10)(-28,18)
\put(-35,25){\vector(-1,1){5}}
\put(-25,25){\vector(1,1){5}}
\put(-30,0){\makebox(0,0)[cc]{$\operatorname{sgn}(c)=+1$}}
\qbezier(40,10)(40,10)(20,30)
\qbezier(40,30)(40,30)(32,22) 
\qbezier(20,10)(20,10)(28,18)
\put(25,25){\vector(-1,1){5}}
\put(35,25){\vector(1,1){5}}
\put(30,0){\makebox(0,0)[cc]{$\operatorname{sgn}(c)=-1$}}
\end{picture}
\caption{Sign of crossings.} \label{fig30}
\end{figure}
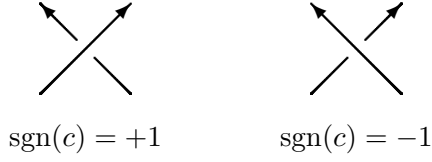

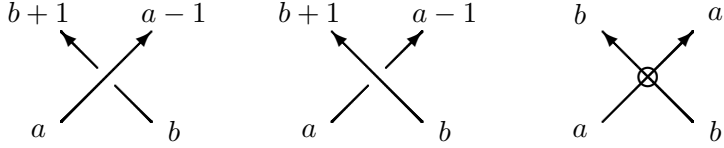
\begin{figure}[!ht]
\centering 
\unitlength=0.6mm
\begin{picture}(0,30)(0,5)
\thicklines
\qbezier(-70,10)(-70,10)(-50,30)
\qbezier(-70,30)(-70,30)(-62,22) 
\qbezier(-50,10)(-50,10)(-58,18)
\put(-65,25){\vector(-1,1){5}}
\put(-55,25){\vector(1,1){5}}
\put(-75,34){\makebox(0,0)[cc]{$b+1$}}
\put(-75,8){\makebox(0,0)[cc]{$a$}}
\put(-45,8){\makebox(0,0)[cc]{$b$}}
\put(-45,34){\makebox(0,0)[cc]{$a-1$}}
\qbezier(10,10)(10,10)(-10,30)
\qbezier(10,30)(10,30)(2,22) 
\qbezier(-10,10)(-10,10)(-2,18)
\put(-5,25){\vector(-1,1){5}}
\put(5,25){\vector(1,1){5}}
\put(-15,34){\makebox(0,0)[cc]{$b+1$}}
\put(-15,8){\makebox(0,0)[cc]{$a$}}
\put(15,8){\makebox(0,0)[cc]{$b$}}
\put(15,34){\makebox(0,0)[cc]{$a-1$}}
\qbezier(70,10)(70,10)(50,30)
\qbezier(70,30)(70,30)(50,10) 
\put(55,25){\vector(-1,1){5}}
\put(65,25){\vector(1,1){5}}
\put(60,20){\circle{4}}
\put(45,34){\makebox(0,0)[cc]{$b$}}
\put(45,8){\makebox(0,0)[cc]{$a$}}
\put(75,8){\makebox(0,0)[cc]{$b$}}
\put(75,34){\makebox(0,0)[cc]{$a$}}
\end{picture}
\caption{Labeling of arcs.} \label{fig31}
\end{figure}

After labeling assign a weight $W_{D}(c)$ to each classical crossing $c$ defined in \cite{p} as 
$$
W_{D}(c) = \operatorname{sgn} (c)(a-b-1). 
$$ 
Then the Kauffman's \emph{affine index polynomial} \cite{p} of virtual knot diagram $D$ is defined as
\begin{equation}
P_{D}(t) = \sum _{c \in C(D)} \operatorname{sgn}(c)(t^{W_{D}(c)}-1) \label{eqn1}
\end{equation} 
where the summation runs over the set $C(D)$ of classical crossings of $D$. The weight $W_{D}(c)$ assigned to each classical crossing $c$ is also called as \emph{index value} of $c$ and denoted by $\operatorname{Ind}(c)$.  

\begin{proposition} \label{prop51} For any $n \geq 1$ affine index polynomial of the virtual knot $VK_{n}$ is equal to $2 - t^{2} - t^{-2}$. 
\end{proposition} 

\begin{proof} 
To compute $P_{VK_n}(t)$ we will use the labeling of diagram $VK_n$ as shown in Fig.~\ref{fig32} and Fig.~\ref{fig33} for parts of it. We start with some arc labelled by $a$ and obtain induced labeling of all other arc by the rule presented in Fig.~\ref{fig31}. Because of the size of the pictures some labels are not presented in Figs.~\ref{fig32} and~\ref{fig33}, but we are sure they can be easy reconstructed by a reader. Note that labels of arcs are identical in all the blocks $c_i$ for $i= 2, \ldots, n$ in $VK_n$.
\begin{figure}
\unitlength=.5mm
\centering 
\begin{picture}(250,90)(0,0)
\put(0,60){\vector(0,-1){5}}
\put(5,35){\vector(-1,-1){5}}
\put(15,35){\vector(1,-1){5}}
\put(45,35){\vector(-1,-1){5}}
\put(55,35){\vector(1,-1){5}}
\put(5,55){\makebox(0,0)[cc]{\tiny ${a}$}}
\put(20,55){\makebox(0,0)[cc]{\tiny ${a+1}$}}
\put(40,55){\makebox(0,0)[cc]{\tiny $a+1$}}
\put(60,55){\makebox(0,0)[cc]{\tiny $a+2$}}
\put(5,25){\makebox(0,0)[cc]{\tiny ${a}$}}
\put(20,25){\makebox(0,0)[cc]{\tiny ${a+1}$}}
\put(40,25){\makebox(0,0)[cc]{\tiny $a+1$}}
\put(60,25){\makebox(0,0)[cc]{\tiny $a+2$}}
\thicklines
\qbezier(0,0)(0,0)(90,0)
\qbezier(0,0)(0,0)(0,30)
\qbezier(0,50)(0,50)(0,80)
\qbezier(0,80)(0,80)(90,80)
\qbezier(0,30)(0,30)(20,50)
\qbezier(0,50)(0,50)(6,44)
\qbezier(14,36)(20,30)(20,30)
\qbezier(20,30)(20,30)(40,50)
\qbezier(20,50)(20,50)(40,30)
\qbezier(40,30)(40,30)(60,50)
\qbezier(40,50)(40,50)(46,44)
\qbezier(54,36)(60,30)(60,30)
\qbezier(60,30)(60,30)(80,50)
\qbezier(60,50)(60,50)(80,30)
\qbezier(80,50)(80,50)(80,70)
\qbezier(80,70)(80,70)(90,70)
\qbezier(80,30)(80,30)(80,10)
\qbezier(80,10)(80,10)(90,10)
\put(30,40){\circle{6}}
\put(70,40){\circle{6}}
\put(120,40){\vector(0,-1){5}}
\put(130,40){\vector(0,1){5}}
\put(100,10){\vector(-1,0){5}}
\put(190,10){\vector(-1,0){5}}
\qbezier(90,0)(90,0)(250,0)
\qbezier(90,10)(90,10)(120,10)
\qbezier(120,10)(120,10)(120,16)
\qbezier(120,24)(120,24)(120,26)
\qbezier(120,34)(120,34)(120,70)
\qbezier(120,70)(120,70)(90,70)
\qbezier(80,80)(80,80)(130,80)
\qbezier(130,80)(130,80)(130,34)
\qbezier(130,26)(130,26)(130,24)
\qbezier(130,16)(130,16)(130,10)
\qbezier(130,10)(130,10)(250,10)
\qbezier(100,20)(100,20)(100,60)
\qbezier(100,20)(100,20)(240,20)
\qbezier(240,20)(240,20)(240,70)
\qbezier(240,70)(240,70)(250,70)
\qbezier(100,20)(100,20)(100,60)
\qbezier(100,60)(100,60)(220,60)
\qbezier(110,30)(110,30)(110,50)
\qbezier(110,30)(110,30)(140,30)
\qbezier(110,50)(110,50)(140,50)
\put(120,50){\circle{6}}
\put(120,60){\circle{6}}
\put(130,50){\circle{6}}
\put(130,60){\circle{6}}
\put(160,55){\makebox(0,0)[cc]{\tiny $a+2$}}
\put(180,55){\makebox(0,0)[cc]{\tiny ${a+1}$}}
\put(200,55){\makebox(0,0)[cc]{\tiny $a+1$}}
\put(215,55){\makebox(0,0)[cc]{\tiny $a$}} 
\put(160,25){\makebox(0,0)[cc]{\tiny $a+2$}}
\put(180,25){\makebox(0,0)[cc]{\tiny ${a+1}$}}
\put(200,25){\makebox(0,0)[cc]{\tiny $a+1$}}
\put(215,25){\makebox(0,0)[cc]{\tiny $a$}}
\put(140,25){\makebox(0,0)[cc]{\tiny $a+2$}}
\put(110,25){\makebox(0,0)[cc]{\tiny $a+2$}}
\put(100,15){\makebox(0,0)[cc]{\tiny $a+2$}}
\put(190,15){\makebox(0,0)[cc]{\tiny $a$}}
\put(235,25){\makebox(0,0)[cc]{\tiny $a$}}
\put(245,85){\makebox(0,0)[cc]{\tiny $a$}}
\put(245,75){\makebox(0,0)[cc]{\tiny $a$}}
\put(245,15){\makebox(0,0)[cc]{\tiny $a$}}
\put(245,5){\makebox(0,0)[cc]{\tiny $a$}}
\put(165,45){\vector(-1,1){5}}
\put(175,45){\vector(1,1){5}}
\put(205,45){\vector(-1,1){5}}
\put(215,45){\vector(1,1){5}}
\put(140,30){\vector(-1,0){5}}
\put(235,20){\vector(1,0){5}}
\put(245,80){\vector(-1,0){5}}
\put(245,70){\vector(1,0){5}}
\put(245,10){\vector(-1,0){5}}
\put(245,0){\vector(1,0){5}}
\put(15,40){\makebox(0,0)[cc]{\tiny $c_{0}^{1}$}}
\put(55,40){\makebox(0,0)[cc]{\tiny $c_{0}^{2}$}}
\put(175,40){\makebox(0,0)[cc]{\tiny $c_{1}^{5}$}}
\put(215,40){\makebox(0,0)[cc]{\tiny $c_{1}^{6}$}}
\put(115,35){\makebox(0,0)[cc]{\tiny $c_{1}^{1}$}}
\put(135,35){\makebox(0,0)[cc]{\tiny $c_{1}^{2}$}}
\put(115,15){\makebox(0,0)[cc]{\tiny $c_{1}^{3}$}}
\put(135,15){\makebox(0,0)[cc]{\tiny $c_{1}^{4}$}}
\qbezier(140,30)(140,30)(160,50)
\qbezier(140,50)(140,50)(160,30)
\qbezier(160,30)(160,30)(180,50)
\qbezier(160,50)(160,50)(166,44)
\qbezier(174,36)(180,30)(180,30)
\qbezier(180,30)(180,30)(200,50)
\qbezier(180,50)(180,50)(200,30)
\qbezier(200,30)(200,30)(220,50)
\qbezier(200,50)(200,50)(206,44)
\qbezier(214,36)(220,30)(220,30)
\put(150,40){\circle{6}}
\put(190,40){\circle{6}}
\qbezier(220,50)(220,50)(220,60)
\qbezier(220,30)(220,30)(230,30)
\qbezier(230,30)(230,30)(230,80)
\qbezier(230,80)(230,80)(250,80)
\end{picture}
\caption{Labelling in blocks $c_{0}$ and $c_{1}$.} \label{fig32} 
\end{figure}
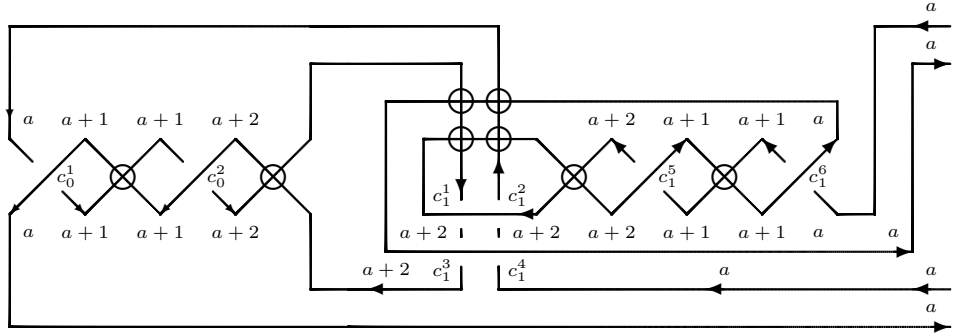
Denote the classical crossings lying in blocks $c_i$ for $i=0,\ldots, n $ as shown in Figs.~\ref{fig32} and~\ref{fig33} and calculate signs of classical crossings. 

\begin{figure}
\unitlength=.4mm
\centering 
\begin{picture}(0,90)(0,0)
\put(-245,0){\begin{picture}(0,100)
\thicklines
\put(95,85){\makebox(0,0)[cc]{\tiny $a$}}
\put(95,75){\makebox(0,0)[cc]{\tiny $a$}}
\put(95,15){\makebox(0,0)[cc]{\tiny $a$}}
\put(95,5){\makebox(0,0)[cc]{\tiny $a$}}
\put(95,80){\vector(-1,0){5}} 
\put(95,70){\vector(1,0){5}} 
\put(95,10){\vector(-1,0){5}} 
\put(95,0){\vector(1,0){5}} 
\put(160,55){\makebox(0,0)[cc]{\tiny $a+2$}}
\put(180,55){\makebox(0,0)[cc]{\tiny ${a+1}$}}
\put(200,55){\makebox(0,0)[cc]{\tiny $a+1$}}
\put(215,55){\makebox(0,0)[cc]{\tiny $a$}} 
\put(160,25){\makebox(0,0)[cc]{\tiny $a+2$}}
\put(180,25){\makebox(0,0)[cc]{\tiny ${a+1}$}}
\put(200,25){\makebox(0,0)[cc]{\tiny $a+1$}}
\put(215,25){\makebox(0,0)[cc]{\tiny $a$}}
\put(140,25){\makebox(0,0)[cc]{\tiny $a+2$}}
\put(110,25){\makebox(0,0)[cc]{\tiny $a+2$}}
\put(190,15){\makebox(0,0)[cc]{\tiny $a$}}
\put(235,25){\makebox(0,0)[cc]{\tiny $a$}}
\put(245,85){\makebox(0,0)[cc]{\tiny $a$}}
\put(245,75){\makebox(0,0)[cc]{\tiny $a$}}
\put(245,15){\makebox(0,0)[cc]{\tiny $a$}}
\put(245,5){\makebox(0,0)[cc]{\tiny $a$}}
\put(165,45){\vector(-1,1){5}}
\put(175,45){\vector(1,1){5}}
\put(205,45){\vector(-1,1){5}}
\put(215,45){\vector(1,1){5}}
\put(140,30){\vector(-1,0){5}}
\put(235,20){\vector(1,0){5}}
\put(245,80){\vector(-1,0){5}}
\put(245,70){\vector(1,0){5}}
\put(245,10){\vector(-1,0){5}}
\put(245,0){\vector(1,0){5}}
\put(175,40){\makebox(0,0)[cc]{\tiny $c_{2}^{5}$}}
\put(215,40){\makebox(0,0)[cc]{\tiny $c_{2}^{6}$}}
\put(115,35){\makebox(0,0)[cc]{\tiny $c_{2}^{1}$}}
\put(135,35){\makebox(0,0)[cc]{\tiny $c_{2}^{2}$}}
\put(115,15){\makebox(0,0)[cc]{\tiny $c_{2}^{3}$}}
\put(135,15){\makebox(0,0)[cc]{\tiny $c_{2}^{4}$}}
\qbezier(90,0)(90,0)(250,0)
\qbezier(90,10)(90,10)(120,10)
\qbezier(120,10)(120,10)(120,16)
\qbezier(120,24)(120,24)(120,26)
\qbezier(120,34)(120,34)(120,70)
\qbezier(120,70)(120,70)(90,70)
\qbezier(90,80)(90,80)(130,80)
\qbezier(130,80)(130,80)(130,34)
\qbezier(130,26)(130,26)(130,24)
\qbezier(130,16)(130,16)(130,10)
\qbezier(130,10)(130,10)(250,10)
\qbezier(100,20)(100,20)(100,60)
\qbezier(100,20)(100,20)(240,20)
\qbezier(240,20)(240,20)(240,70)
\qbezier(240,70)(240,70)(250,70)
\qbezier(100,20)(100,20)(100,60)
\qbezier(100,60)(100,60)(220,60)
\qbezier(110,30)(110,30)(110,50)
\qbezier(110,30)(110,30)(140,30)
\qbezier(110,50)(110,50)(140,50)
\put(120,50){\circle{6}}
\put(120,60){\circle{6}}
\put(130,50){\circle{6}}
\put(130,60){\circle{6}}
\qbezier(140,30)(140,30)(160,50)
\qbezier(140,50)(140,50)(160,30)
\qbezier(160,30)(160,30)(180,50)
\qbezier(160,50)(160,50)(166,44)
\qbezier(174,36)(180,30)(180,30)
\qbezier(180,30)(180,30)(200,50)
\qbezier(180,50)(180,50)(200,30)
\qbezier(200,30)(200,30)(220,50)
\qbezier(200,50)(200,50)(206,44)
\qbezier(214,36)(220,30)(220,30)
\put(150,40){\circle{6}}
\put(190,40){\circle{6}}
\qbezier(220,50)(220,50)(220,60)
\qbezier(220,30)(220,30)(230,30)
\qbezier(230,30)(230,30)(230,80)
\qbezier(230,80)(230,80)(250,80)
\put(260,80){\makebox(0,0)[cc]{$\cdots$}}
\put(260,70){\makebox(0,0)[cc]{$\cdots$}}
\put(260,10){\makebox(0,0)[cc]{$\cdots$}}
\put(260,0){\makebox(0,0)[cc]{$\cdots$}}
\end{picture}}
\put(-65,0){\begin{picture}(0,100)
\thicklines
\put(95,85){\makebox(0,0)[cc]{\tiny $a$}}
\put(95,75){\makebox(0,0)[cc]{\tiny $a$}}
\put(95,15){\makebox(0,0)[cc]{\tiny $a$}}
\put(95,5){\makebox(0,0)[cc]{\tiny $a$}}
\put(95,80){\vector(-1,0){5}} 
\put(95,70){\vector(1,0){5}} 
\put(95,10){\vector(-1,0){5}} 
\put(95,0){\vector(1,0){5}} 
\put(160,55){\makebox(0,0)[cc]{\tiny $a+2$}}
\put(180,55){\makebox(0,0)[cc]{\tiny ${a+1}$}}
\put(200,55){\makebox(0,0)[cc]{\tiny $a+1$}}
\put(215,55){\makebox(0,0)[cc]{\tiny $a$}} 
\put(160,25){\makebox(0,0)[cc]{\tiny $a+2$}}
\put(180,25){\makebox(0,0)[cc]{\tiny ${a+1}$}}
\put(200,25){\makebox(0,0)[cc]{\tiny $a+1$}}
\put(215,25){\makebox(0,0)[cc]{\tiny $a$}}
\put(140,25){\makebox(0,0)[cc]{\tiny $a+2$}}
\put(110,25){\makebox(0,0)[cc]{\tiny $a+2$}}
\put(190,15){\makebox(0,0)[cc]{\tiny $a$}}
\put(165,45){\vector(-1,1){5}}
\put(175,45){\vector(1,1){5}}
\put(205,45){\vector(-1,1){5}}
\put(215,45){\vector(1,1){5}}
\put(140,30){\vector(-1,0){5}}
\put(190,10){\vector(-1,0){5}}
\put(175,40){\makebox(0,0)[cc]{\tiny $c_{n}^{5}$}}
\put(215,40){\makebox(0,0)[cc]{\tiny $c_{n}^{6}$}}
\put(115,35){\makebox(0,0)[cc]{\tiny $c_{n}^{1}$}}
\put(135,35){\makebox(0,0)[cc]{\tiny $c_{n}^{2}$}}
\put(115,15){\makebox(0,0)[cc]{\tiny $c_{n}^{3}$}}
\put(135,15){\makebox(0,0)[cc]{\tiny $c_{n}^{4}$}}
\qbezier(90,0)(90,0)(230,0)
\qbezier(90,10)(90,10)(120,10)
\qbezier(120,10)(120,10)(120,16)
\qbezier(120,24)(120,24)(120,26)
\qbezier(120,34)(120,34)(120,70)
\qbezier(120,70)(120,70)(90,70)
\qbezier(90,80)(90,80)(130,80)
\qbezier(130,80)(130,80)(130,34)
\qbezier(130,26)(130,26)(130,24)
\qbezier(130,16)(130,16)(130,10)
\qbezier(130,10)(130,10)(220,10)
\qbezier(100,20)(100,20)(100,60)
\qbezier(100,20)(100,20)(220,20)
\qbezier(100,20)(100,20)(100,60)
\qbezier(100,60)(100,60)(220,60)
\qbezier(110,30)(110,30)(110,50)
\qbezier(110,30)(110,30)(140,30)
\qbezier(110,50)(110,50)(140,50)
\put(120,50){\circle{6}}
\put(120,60){\circle{6}}
\put(130,50){\circle{6}}
\put(130,60){\circle{6}} 
\qbezier(140,30)(140,30)(160,50)
\qbezier(140,50)(140,50)(160,30)
\qbezier(160,30)(160,30)(180,50)
\qbezier(160,50)(160,50)(166,44)
\qbezier(174,36)(180,30)(180,30)
\qbezier(180,30)(180,30)(200,50)
\qbezier(180,50)(180,50)(200,30)
\qbezier(200,30)(200,30)(220,50)
\qbezier(200,50)(200,50)(206,44)
\qbezier(214,36)(220,30)(220,30)
\put(150,40){\circle{6}}
\put(190,40){\circle{6}}
\qbezier(220,50)(220,50)(220,60)
\qbezier(220,30)(220,30)(230,30)
\qbezier(220,10)(220,10)(220,20)
\qbezier(230,0)(230,0)(230,30)
\end{picture}}
\end{picture}
\caption{Labelling in blocks $c_{2}, \ldots, c_{n}$.} \label{fig33} 
\end{figure}
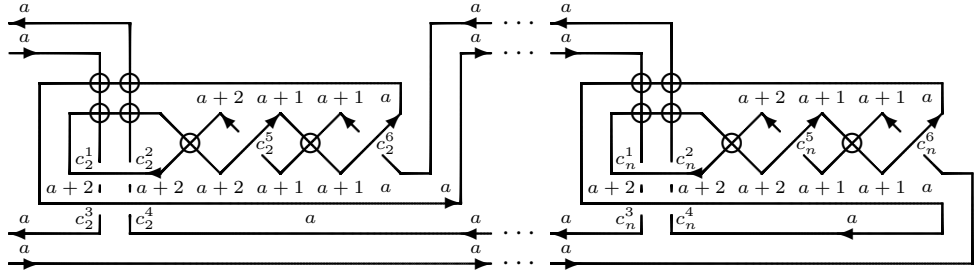

Since not all labels around crossings $c_{i}^{k}$, $k=1, \ldots, 4$, $i=1, \ldots, n$. are presented in Figs.~\ref{fig32} and~\ref{fig33}, we present labels around these crossings in Figs.~\ref{fig101} and~\ref{fig102}. 
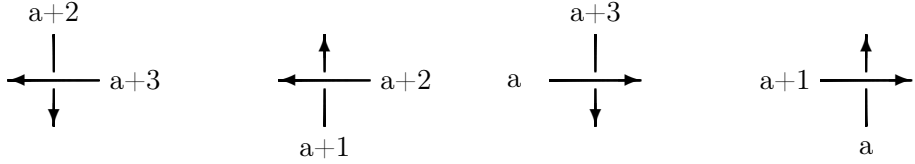
\begin{figure}[!ht]
\centering 
\unitlength=0.6mm
\begin{picture}(0,30)(0,5)
\put(-90,0){\begin{picture}(0,30) 
\thicklines
\qbezier(-10,20)(-10,20)(10,20)
\qbezier(0,30)(0,30)(0,22) 
\qbezier(0,10)(0,10)(0,18)
\put(-5,20){\vector(-1,0){5}}
\put(0,15){\vector(0,-1){5}}
\put(0,35){\makebox(0,0)[cc]{a+2}}
\put(18,20){\makebox(0,0)[cc]{a+3}}
\end{picture}}
\put(-30,0){\begin{picture}(0,30) 
\thicklines
\qbezier(-10,20)(-10,20)(10,20)
\qbezier(0,30)(0,30)(0,22) 
\qbezier(0,10)(0,10)(0,18)
\put(-5,20){\vector(-1,0){5}}
\put(0,25){\vector(0,1){5}}
\put(0,5){\makebox(0,0)[cc]{a+1}}
\put(18,20){\makebox(0,0)[cc]{a+2}}
\end{picture}}
\put(30,0){\begin{picture}(0,30) 
\thicklines
\qbezier(-10,20)(-10,20)(10,20)
\qbezier(0,30)(0,30)(0,22) 
\qbezier(0,10)(0,10)(0,18)
\put(5,20){\vector(1,0){5}}
\put(0,15){\vector(0,-1){5}}
\put(0,35){\makebox(0,0)[cc]{a+3}}
\put(-18,20){\makebox(0,0)[cc]{a}}
\end{picture}}
\put(90,0){\begin{picture}(0,30) 
\thicklines
\qbezier(-10,20)(-10,20)(10,20)
\qbezier(0,30)(0,30)(0,22) 
\qbezier(0,10)(0,10)(0,18)
\put(5,20){\vector(1,0){5}}
\put(0,25){\vector(0,1){5}}
\put(0,5){\makebox(0,0)[cc]{a}}
\put(-18,20){\makebox(0,0)[cc]{a+1}}
\end{picture}}
\end{picture}
\caption{Labels around crossings $c_{1}^{1}$, $c_{1}^{2}$, $c_{1}^{3}$ and $c_{1}^{4}$.} \label{fig101}
\end{figure}

\begin{figure}[!ht]
\centering 
\unitlength=0.6mm
\begin{picture}(0,30)(0,5)
\put(-90,0){\begin{picture}(0,30) 
\thicklines
\qbezier(-10,20)(-10,20)(10,20)
\qbezier(0,30)(0,30)(0,22) 
\qbezier(0,10)(0,10)(0,18)
\put(-5,20){\vector(-1,0){5}}
\put(0,15){\vector(0,-1){5}}
\put(0,35){\makebox(0,0)[cc]{a}}
\put(18,20){\makebox(0,0)[cc]{a+3}}
\end{picture}}
\put(-30,0){\begin{picture}(0,30) 
\thicklines
\qbezier(-10,20)(-10,20)(10,20)
\qbezier(0,30)(0,30)(0,22) 
\qbezier(0,10)(0,10)(0,18)
\put(-5,20){\vector(-1,0){5}}
\put(0,25){\vector(0,1){5}}
\put(0,5){\makebox(0,0)[cc]{a+1}}
\put(18,20){\makebox(0,0)[cc]{a+2}}
\end{picture}}
\put(30,0){\begin{picture}(0,30) 
\thicklines
\qbezier(-10,20)(-10,20)(10,20)
\qbezier(0,30)(0,30)(0,22) 
\qbezier(0,10)(0,10)(0,18)
\put(5,20){\vector(1,0){5}}
\put(0,15){\vector(0,-1){5}}
\put(0,35){\makebox(0,0)[cc]{a+1}}
\put(-18,20){\makebox(0,0)[cc]{a}}
\end{picture}}
\put(90,0){\begin{picture}(0,30) 
\thicklines
\qbezier(-10,20)(-10,20)(10,20)
\qbezier(0,30)(0,30)(0,22) 
\qbezier(0,10)(0,10)(0,18)
\put(5,20){\vector(1,0){5}}
\put(0,25){\vector(0,1){5}}
\put(0,5){\makebox(0,0)[cc]{a}}
\put(-18,20){\makebox(0,0)[cc]{a+1}}
\end{picture}}
\end{picture}
\caption{Labels around crossings $c_{i}^{1}$, $c_{i}^{2}$, $c_{i}^{3}$ and $c_{i}^{4}$ for $i=2, \ldots, n$.} \label{fig102}
\end{figure}
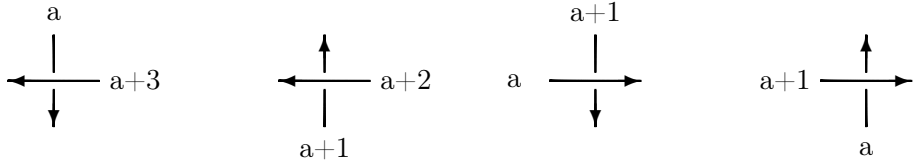

It is easy to see that for the block $c_{0}$ we get  $\operatorname{Ind}(c_0^1)=  \operatorname{Ind}(c_0^2)= 0$. For the block $c_{1}$ we get $\operatorname{Ind}(c_1^1)=  \operatorname{Ind}(c_1^4)=  \operatorname{Ind}(c_1^5)=  \operatorname{Ind}(c_1^6)= 0$ and $\operatorname{Ind}(c_1^2)= 2,  \operatorname{Ind}(c_1^3)= -2$. Analogously, for the blocks $c_i$, $i=2, \ldots, n$, we get $\operatorname{Ind}(c_i^1)=  \operatorname{Ind}(c_i^2)= 2$ and $\operatorname{Ind}(c_i^3)=  \operatorname{Ind}(c_i^4)=  \operatorname{Ind}(c_i^5)=  \operatorname{Ind}(c_i^6)= 0$.


Summarizing, from all index values computed for classical crossings of $VK_n$ we obtain the affine index polynomial $P_{VK_n}(t)$: 
\begin{equation}
P_{VK_n}(t) = \operatorname{sgn}(c_0^1)(t^0-1)+ \operatorname{sgn}(c_0^2)(t^0-1)+ \sum _{\substack{1\leq i\leq n\\
                  1\leq j\leq 6
                  }} \operatorname{sgn}(c_i^j)(t^{\operatorname{Ind}(c_i^j)}-1)
\end{equation} 
Since crossings with $\operatorname{Ind}(c)=0$ contributes zero to the polynomial, we have 
\begin{equation}
 \begin{aligned}
P_{VK_n}(t) &= \operatorname{sgn}(c_1^2)(t^2-1)+ \operatorname{sgn}(c_1^3)(t^{-2}-1)+ \sum _{\substack{2\leq i\leq n\\
                  j=1,2
                  }} \operatorname{sgn}(c_i^j)(t^{\operatorname{Ind}(c_i^j)}-1)\\
            &= (-1)(t^2-1)+ (-1)(t^{-2}-1)+ \sum _{\substack{2\leq i\leq n
                  }} \left(1(t^2-1)+(-1)(t^2-1)\right) \\
            &= 2-t^2-t^{-2}, 
 \end{aligned}                  
\end{equation} 
where we used $ \operatorname{sgn}(c_1^2) = \operatorname{sgn}(c_1^3) -=1$ and  $\operatorname{sgn}(c_i^1)=1$,  $\operatorname{sgn}(c_i^2) = -1$ for $i=2, \ldots, n$. 

Therefore for all $n= 1,2,\ldots$ virtual knots $VK_n$ have the same affine index polynomial $P_{VK_n}(t) = 2-t^2-t^{-2}$. 
\end{proof}



As a generalization of the affine index polynomial, in~\cite{o} there were defined two sequences of two-variable polynomial invariants, $L_{K}^{m}(t, \ell)$ and $F_{K}^{m}(t, \ell)$, of virtual knots. These polynomial were called $L$ and $F$ for shortest. They arise from flat virtual knot invariants. $F$-polynomials are stronger invariants than $L$-polynomials; replacing negative powers of $\ell$ in $F$-polynomials by absolute values gives back $L$-polynomials. Except for finitely many values of $m$, $F$-polynomials $F_{K}^{m}(t, \ell)$ turns out to be identical with $P_K(t)$. It was also shown in \cite{o} that for $\ell =1$ both $L_{VK}^{m}(t, 1)$ and $F_{K}^{m}(t, 1)$ for all $m \geq 1$ coincide with affine index polynomial $P_K(t)$. 


Since the affine index polynomial didn't distinguish our virtual knots, we tried to do it by its generalization. 
We calculated $F$-Polynomials for the virtual knots $VK_n$,  $n= 1. \ldots, 10$.  It was obtained  that $F_{VK_n}^{2}(t, \ell)$, $F_{VK_n}^{4}(t, \ell)$ and $F_{VK_n}^{6}(t, \ell)$ are the only $F$-polynomials that turn out to be different from $P_{VK_n}(t)$. Polynomials  $F_{VK_n}^{2}(t, \ell)$, $F_{VK_n}^{4}(t, \ell)$ and $F_{VK_n}^{6}(t, \ell)$ for $n=1, 2, 3$ are presented in~Table~\ref{table1}.  For all $n =4, \ldots, 10$ the polynomials are the same as for the case $n=3$. 

\begin{table}[ht]
\caption{$F$-Polynomials for virtual knots $VK_1, \ldots, VK_{10}$.}
{\begin{tabular}{ |p{1.4cm}|p{3.2cm}|p{3.2cm}|p{2.2cm}|} 
\hline
$n$  &$F_{VK_n}^{2}(t, \ell)$  & $F_{VK_n}^{4}(t, \ell)$&$F_{VK_n}^{6}(t, \ell)$ \\
\hline
$1$ & $ \ell^{-2}+\ell^2-t^2\ell^{-2}-\ell^2t^{-2} $ & $ \ell-\ell t^2+\ell^{-1}-\ell^{-1}t^{-2} $  & $2-t^2-t^{-2}$ \\
\hline
$2$ &$\ell^{-1}+\ell^{-2}+2\ell^2-t^{-2}-t^2+t^2\ell^{-1}-t^2\ell^{-4}-2 $ & $\ell+3 \ell^{-1}-t^{-2}-t^2-2$  & $\ell-t^{-2}-2t^2+t^2\ell^{-1}+1$   \\
\hline
$3,\ldots,10$ &$\ell^{-1}+\ell^{-2}+2\ell^2-t^{-2}-t^2\ell^{-2}+t^2\ell^{-3}-t^2\ell^{-4}-2$  &$\ell-\ell t^2+3\ell^{-1}-t^{-2}-\ell^2t^2+\ell^3t^2-2$  & $\ell-t^{-2}-2t^2+t^2\ell^{-1}+1$  \\
\hline
\end{tabular}}
\label{table1}
\end{table}


\begin{problem}
Are virtual knots $VK_n$ distinguishable by other polynomial invariants? 
\end{problem}

\begin{problem} 
Does there exist an infinite family of virtual knots $\{K_n\}_{n\geq 0}$ as required in Theorem~\ref{1} which are distinguishable by $F$-polynomials or Affine index polynomial?
 \end{problem}

\end{document}